\documentclass[preprint, 11pt]{elsarticle}
\usepackage{lineno}
\pdfoutput=1

\usepackage{amsmath}
\usepackage{amssymb}
\usepackage{mathtools}
\usepackage{mathrsfs}
\usepackage{color}
\usepackage{xcolor}
\usepackage{graphicx}
\usepackage[percent]{overpic}
\usepackage{url}
\usepackage{caption}
\usepackage{subcaption}
\usepackage{enumerate}
\usepackage{overpic}
\usepackage{amsthm}
\usepackage{moreverb}
\usepackage{multirow}
\usepackage{multicol}
\usepackage[pdftex,colorlinks,bookmarksopen,bookmarksnumbered,citecolor=red,urlcolor=red]{hyperref}
\usepackage{mwe}
\usepackage{graphbox}
\usepackage{textcomp}
\usepackage{tabularx}
\usepackage{algorithm}
\usepackage{algorithmicx}
\usepackage{algpseudocode}
\newtheorem{theorem}{Theorem}[section]
\newtheorem{lemma}[theorem]{Lemma}

\def\b{{\bm b}}

\def\x{{\bm x}}

\def\0{\boldsymbol{0}}

\def \b0{{\mathbf 0}}

\newcommand{\bm}[1]{\mbox{\boldmath{$#1$}}}
\DeclarePairedDelimiter\abs{\lvert}{\rvert}%
\newcommand\dual[1]{\left\langle#1\right\rangle}%
\definecolor{ForestGreen}{RGB}{34,139,34}
\newcommand\norm[1]{\left\lVert#1\right\rVert}

\newcommand{\lander}[2][orange]{{\textcolor{#1}{#2}}}

\usepackage{verbatim}
\usepackage{graphicx}
\usepackage{epstopdf}
\begin{document}

%% ELSEVIER %%
\begin{frontmatter}

\title{Finite element approximation for a reformulation of a 3D fluid–2D plate interaction system}
%\tnotetext[mytitlenote]{Fully documented templates are available in the elsarticle package on \href{http://www.ctan.org/tex-archive/macros/latex/contrib/elsarticle}{CTAN}.}

%% Group authors per affiliation:
%\author{Elsevier\fnref{myfootnote}}
%\address{Radarweg 29, Amsterdam}
%\fntext[myfootnote]{Since 1880.}

%% or include affiliations in footnotes:
%\author[mymainaddress,mysecondaryaddress]{Elsevier Inc}
%\ead[url]{www.elsevier.com}

%\author[mysecondaryaddress]{Global Customer Service\corref{mycorrespondingauthor}}
%\cortext[mycorrespondingauthor]{Corresponding author}
%\ead{support@elsevier.com}

%\address[mymainaddress]{1600 John F Kennedy Boulevard, Philadelphia}
%\address[mysecondaryaddress]{360 Park Avenue South, New York}

%% or include affiliations in footnotes:
\author{Lander Besabe}
\ead{lbesabe@clemson.edu}

\author{Hyesuk Lee}
\ead{hklee@clemson.edu}

\address{School of Mathematical and Statistical Sciences, Clemson University, Clemson, SC 29634-0975, USA}

\begin{abstract}
We study a finite element approximation of a coupled fluid-structure interaction consisting of a three-dimensional incompressible viscous fluid governed by the unsteady Stokes equations and a
two-dimensional elastic plate.  To avoid the use of $H^2-$conforming or nonconforming $\mathbb{P}_2$-Morley plate elements, the fourth-order plate equation is reformulated into a system of coupled second-order equations using an auxiliary variable. The coupling condition is enforced using a Lagrange multiplier representing the trace of the mean-zero fluid pressure on the interface.

We establish well-posedness and stability results for the time-discrete and fully-discrete problems, and derive a priori error estimates. A partitioned domain decomposition algorithm based on a fixed-point iteration is employed for the numerical solution. Numerical experiments verify the theoretical rates of convergence in space and time using manufactured solutions, and demonstrate the applicability of the method to a physical problem.
\end{abstract}


\end{frontmatter}

\section{Introduction} \label{sec:intro}

Blood flow \cite{Quaini2011, Duca2025, Quarteroni2000}, aeroelasticity \cite{Zheng2023, Svacek2008}, microfluidic devices \cite{Hashim2012}, and systems involving flexible walls \cite{Sabbar2018, Nicolici2013} are commonly modeled as fluid-structure interaction (FSI) systems of a viscous fluid coupled with a thin elastic structure, such as a plate. These problems consist of an unsteady viscous flow in a three-dimensional domain interacting with an elastic structure on (a portion of) the boundary of the domain. Even in this simplified setting, this interaction may exhibit nontrivial behavior such as added-mass effects \cite{Brummelen2009, Causin2005, Richter2017} and strong coupling, which introduce issues in the stability of numerical simulations.

From a numerical standpoint, the development of accurate and efficient simulation of the interaction of a 3D fluid and a 2D plate remains challenging \cite{Geredeli_Kunwar_Lee2024}. Aside from the fact that the interface requires careful treatment for stability, the plate models are governed by a fourth-order PDE. This means that standard conforming finite element (FE) discretization requires globally $H^2$ continuous shape functions, which typically restricts the choice to higher-order elements or nonconforming methods like $\mathbb{P}_2$-Morley elements \cite{Gallistl2015}. Nonconforming FE approaches alleviate some of these constraints but introduce additional technical complications and can complicate interface coupling or domain decomposition implementation \cite{Brenner2018}.

A classical alternative to these approaches is to reformulate the fourth-order plate subproblem as a mixed system involving only second-order operators, thereby avoiding the need for $H^2$-conforming elements. The FEM analysis of such problems involving the dynamics of a Kirchhoff plate under various boundary conditions have been extensively studied, see e.g., in \cite{Das2024, Gudi2008, Li2023}.

Relatively few works in the literature have considered the 3D viscous fluid–2D Kirchhoff plate interaction problem, especially with rigorous FEM analysis. The key differences in the works discussed below lie  in how they handle the fourth-order operator associated to the plate dynamics. In \cite{Avalos2014}, quintic Argyris basis functions are used for the plate displacement, as they are the lowest order $H^2_0$-conforming finite elements available \cite{Solin2005}. In \cite{Geredeli_Kunwar_Lee2024, Avalos2025}, nonconforming $\mathbb{P}_2$-Morley elements are utilized for the spatial discretization of the plate equation. In \cite{Geredeli2026}, $H^2$-conforming Hermite basis functions were used for the plate subdomain of the interaction problem. Lastly, \cite{Cheng2008} considers a slightly different but related problem in which a viscous acoustic fluid interacts with a plate. For the structure component, the framework for handling Mindlin plates developed in \cite{Hinton1986} was employed. Hence, to the best of our knowledge, existing numerical approaches for this class of problems typically either use (i) $H^2$-conforming plate elements (e.g., Argyris elements) or (ii) avoid high-order elements by using nonconforming elements like $\mathbb{P}_2$-Morley.

The main contribution of this work is the formulation and analysis of a finite element method for the 3D fluid-2D plate interaction problem with hinged boundary conditions. First, the plate model, originally stated as a fourth-order PDE, is reformulated into two coupled second-order PDEs which avoids the need for $H^2$-conforming or nonconforming $\mathbb{P}_2$-Morley elements, and allows for considerable flexibility in the choice of discrete approximation spaces for the plate displacement. Second, to enforce the coupling on the interface, we use a Lagrange multiplier approach, as done in \cite{Geredeli_Kunwar_Lee2024, Avalos2025}. For the discrete-in-time, continuous-in-space problem, we establish well-posedness and stability results. For the fully discrete finite element approximation, we show well-posedness and provide a convergence analysis.  

The remainder of the paper is organized as follows. The coupled fluid-plate model, its weak formulation, and the use of a Lagrange multiplier method to enforce interface coupling are discussed in Sec.~\ref{sec:gov_eq}. In Sec.~\ref{sec:semi-disc}, we establish the well-posedness and stability of the semi-discrete continuous-in-space problem. We present the fully discrete problem, establish its well-posedness, and the corresponding error analysis
%{\color{red} and the domain decomposition algorithm - delete if the algorithm is deleted from the paper.}
in Sec.~\ref{sec:fully-disc}. In Sec.~\ref{sec:num_res}, two numerical experiments are presented which show that the method achieves the optimal rates of convergence and its applicability to a physical problem. Lastly, we draw some conclusions and future outlook in Sec.~\ref{sec:conclusions}.

\begin{comment}
Novelty/highlights in the method proposed:
\begin{itemize}
    \item this formulation (reducing the biharmonic operator to a second-order equation and adding a Poisson problem) allows for flexibility on the choice of the discrete FE spaces; as opposed to having the choice of P4/5 or P2Morley
    \item Use of domain decomposition: often improves efficiency and scalability while remaining flexible, compared to monolithic approach.
\end{itemize}

\cite{Ciarlet1974} established the well-posedness of the mixed variational formulation of the biharmonic problem with clamped boundary.
\end{comment}

\section{Governing Equations}\label{sec:gov_eq}

We consider a fluid in a bounded domain $\Omega_f\subset\mathbb{R}^3$ with sufficiently smooth boundary $\partial\Omega_f = \bar{S}\cup\bar{\Omega}_p$, with $\Omega_p\cap S = \emptyset$, and
\begin{equation*}
    \Omega_p \subset \{\bm{x} = (x_1, x_2, 0)\} \text{ and surface } S \subset \{\bm{x} = (x_1, x_2, x_3):x_3<0\}.
\end{equation*}
Note that the dynamics of the fluid is described by the linear unsteady Stokes equations. We also consider the dynamics of the hinged plate structure in $\Omega_p$ coupled with the dynamics of the fluid underneath it. We describe its dynamics with either the "Euler-Bernoulli" (no rotation) or "Kirchhoff" (with rotation) equations. With the rotational inertia parameter $\rho \geq 0$, the full problem reads: for a final time $T>0$, find velocity $\bm{u} = (u_1, u_2, u_3)^T:\Omega_f\times(0,T)\rightarrow\mathbb{R}^3$, pressure $p:\Omega_f\times(0,T)\rightarrow\mathbb{R}$, and plate displacement $w:\Omega_p\times(0,T)\rightarrow\mathbb{R}$ such that
\begin{align} \label{eq:momentum}
    &\rho_f\partial_t\bm{u} - \nu_f \Delta \bm{u} + \nabla p = \bm{f} &&\text{in }\Omega_f\times(0,T), \\
    &\nabla\cdot\bm{u} = 0, &&\text{in }\Omega_f\times(0,T),\label{eq:continuity} \\
    &\bm{u} = \bm{0} &&\text{on }S\times(0,T),\label{eq:fluid_bc} \\
    &\partial_{tt}w - \rho\partial_{tt}\Delta w + \Delta^2 w = p\vert_{\Omega_p} &&\text{in }\Omega_p\times(0,T), \label{eq:plate_eq} \\
    &w = \Delta w = 0 &&\text{on }\partial\Omega_p\times(0,T), \label{eq:clapmedbc} \\
    &\bm{u} = (u_1, u_2, u_3)^T = (0, 0, \dot{w})^T &&\text{on }\partial\Omega_p\times(0,T), \label{eq:interface}
\end{align}
where $\partial_t u = \frac{\partial \bm{u}}{\partial t}$, $\partial_{tt}w = \frac{\partial^2 w}{\partial t^2}$, $\rho_f$ is the fluid density, $\nu_f$ denotes the kinematic viscosity of the fluid, $\bm{f}$ represents the external body force. To close the system, we impose the initial conditions
\begin{equation} \label{eq:IC}
    (\bm{u}, w, \partial_t w) = (\bm{u}_0, w_0, w_{t0}).
\end{equation}

Note that in \cite{Geredeli_Kunwar_Lee2024, Avalos2014}, the plate displacement $w$ is assumed to be in $H^2(\Omega_p)$. One of our goals is to relax this regularity for $w$. Hence, we lower the spatial order of the plate equation by introducing an auxiliary variable $z$ such that 
\begin{equation}\label{eq:poisson_variable}
    z = -\Delta w.
\end{equation}
This transforms \eqref{eq:plate_eq}-\eqref{eq:clapmedbc} into the system of second-order PDEs given by
\begin{align}
    &\partial_{tt}w + \rho \partial_{tt}z - \Delta z = p\vert_{\Omega_p} &&\text{in }\Omega_p\times(0,T), \label{eq:plate_eq_order2}\\
    &z + \Delta w = 0 &&\text{in }\Omega_p\times(0,T),\label{eq:plate_poisson}\\
    &w = z = 0 &&\text{on }\partial\Omega_p\times(0,T). \label{eq:clapmedbc2}
\end{align}
This requires adding two more initial conditions
\begin{equation*}
    (z, \partial_t z) = (z_0, z_{t0}).
\end{equation*}
Similar technique has been analyzed for fourth-order problems in \cite{Ciarlet1974, Das2024, Gudi2008, Glowinski1979, Monk1987, Li2023} with various types of boundary conditions in both the continuous and discrete setting.

Following \cite{Geredeli_Kunwar_Lee2024}, we begin by defining the following spaces
\begin{align*}
    &\bm{U} = \{\bm{v} = (v_1, v_2, v_3)^T\in [H^1(\Omega_f)]^3: v_1 = v_2 = 0 \text{ on }\Omega_p, \bm{v} = \bm{0} \text{ on }S\},\\
    &Q = L^2(\Omega_f), \quad Q_0 = L^2_0(\Omega_f) = \left\{q\in L^2(\Omega_f): \int_{\Omega_f} q\,d\Omega_f = 0\right\}, \text{ (mean zero pressure space)}\\
    &\bm{V} = \{\bm{v}\in\bm{U}: (q, \nabla\cdot\bm{v})_{\Omega_f} = 0, \forall q\in Q\},\text{ (div-fee space)} \\
    &W = \left\{\varphi\in H^1(\Omega_p):\varphi = 0 \text{ on }\partial\Omega_p\right\}.%,\\
    %&Z = H^1(\Omega_p).
\end{align*}

Similar to the approach in \cite{Avalos2025}, we decompose the pressure $p$ appearing in \eqref{eq:momentum}, \eqref{eq:plate_eq} into its zero mean component $p_0\in Q_0$ and spatial average $s\in \mathbb{R}$. More precisely, we write
\begin{equation} \label{eq:press_decomp}
    p(\bm{x}, t) = p_0(\bm{x}, t) + s(t), \text{ where } s(t) = \int_{\Omega_f} p(\bm{x}, t)\,d\Omega_f.
\end{equation}
To enforce the coupling condition on the interface $\Omega_p$, we introduce a Lagrange multiplier $g\in G$ defined as the trace of the zero-mean pressure $p_0$ on the plate domain:  
\begin{equation} \label{eq:LM}
    g = p_0|_{\Omega_p}.
\end{equation}
%Let $\bm{\sigma}_f = \nu_f\nabla\bm{u} - p\bm{I}$ be the stress tensor. We introduce the Lagrange multiplier $g\in G = H^{-1/2}(\Omega_p)$ to be
%\begin{equation}
%    g \coloneqq (\bm{\sigma}\bm{n}_f)_3 = -p \quad \text{on }\Omega_p 
%\end{equation}
%with the outward unit normal to $\Omega_f$ defined to be $\bm{n}_f = (0, 0, 1)^T$. This last equality comes from the fact that $\bm{n}_f = (0, 0, 1)^T$, $(u_1, u_2) = (0, 0)$ on $\Omega_p$, and $\nabla \cdot \bm{u} = 0$. So, the term $\partial_zu_3$ should also be zero and only the pressure term survives.
Indeed, it is also easy to see that
\begin{equation} \label{eq:wdot_zeromean}
    \int_{\Omega_p} \partial_t w \,d\Omega_p = 0
\end{equation}
which stems from the matching third component of the velocity of the fluid with the velocity of the plate and the divergence-free condition for $\bm{u}$.
Readers who are interested in a deeper discussion on the model are referred to, e.g., \cite{Avalos2014, Chueshov2013}.

Let $(\cdot,\cdot)_\gamma$ be the $L^2$ inner product over $\gamma$, $\norm{\cdot}_\gamma$ be the corresponding norm, and $\dual{\cdot, \cdot}_\gamma$ be the dual pairing between $H^{1/2}(\gamma)$ and $H^{-1/2}(\gamma)$. We also define the $H^1$ Sobolev norm of $f$ by $\norm{f}_{1, \gamma}^2 = \norm{f}_{\gamma}^2 + \norm{\nabla f}_{\gamma}^2$. The variational problem associated to \eqref{eq:momentum}-\eqref{eq:fluid_bc}, \eqref{eq:plate_eq_order2}-\eqref{eq:clapmedbc2} with \eqref{eq:press_decomp}-\eqref{eq:LM} reads: find $(\bm{u}, p_0, w, z, g, s)\in \bm{U}\times Q_0\times W\times Z\times G \times \mathbb{R}$ such that
\begin{align} \label{eq:momentum_weak}
    &(\partial_t\bm{u}, \bm{v})_{\Omega_f} + \nu_f(\nabla \bm{u}, \nabla\bm{v})_{\Omega_f} - (p_0, \nabla\cdot\bm{v})_{\Omega_f} + \dual{g, v_3}_{\Omega_p} = (\bm{f}_f, \bm{v})_{\Omega_f} && \forall \bm{v}\in\bm{U},\\
    &(q, \nabla\cdot \bm{u})_{\Omega_f} = 0 &&\forall q\in Q_0, \label{eq:continuity_weak} \\
    & (\partial_{tt}w, \eta)_{\Omega_p} + \rho(\partial_{tt}z, \eta)_{\Omega_p} + (\nabla z, \nabla \eta)_{\Omega_p} - \dual{g, \eta}_{\Omega_p} - (s, \eta)_{\Omega_p} = 0 &&\forall\eta\in W, \label{eq:plate_weak}\\
    &(\nabla w, \nabla \varphi)_{\Omega_p} - (z, \varphi)_{\Omega_p} = 0 && \forall\varphi\in W, \label{eq:poisson_weak}\\
    & \dual{u_3 -\partial_t w, \lambda}_{\Omega_p} = 0 &&\forall\lambda\in G. \label{eq:interface_weak}
\end{align}
%The well-posedness of the 3D fluid 2D plate interaction problem has been extensively studied in \cite{Avalos2014, Avalos2025}.

% Lagrange multiplier, inner product, outward unit normal n

\section{Semi-discrete problem} \label{sec:semi-disc}

In this section, we consider the semi-discrete problem  for the fluid-plate system 
\eqref{eq:momentum_weak}-\eqref{eq:interface_weak} and 
study its well-posedness and 
stability. The well-posedness of the steady state case of the 3D fluid 2D plate interaction problem has been extensively studied in \cite{Avalos2014, Avalos2025}, where the plate dynamics are modeled using the original fourth-order plate equation with a clamped boundary condition.

We begin by dividing the interval $[0,T]$ into $N\in\mathbb{Z}^+$ subintervals of length $\delta t = T/N$ and denote $t^n = n\delta t$ for $n = 0,\dots,N$. For any given quantity $f$, we denote its approximation at time $t^n$ by $f^n$.
We discretize the variational problem \eqref{eq:momentum_weak}-\eqref{eq:interface_weak} using the first-order backward Euler approximation given by
\begin{equation} \label{eq:back_euler}
    \partial_{t}\bm{u}\approx \dot{\bm{u}}^{n+1}\coloneqq\frac{\bm{u}^{n+1}-\bm{u}^{n}}{\delta t}, \quad \partial_{tt}w\approx \ddot{w}^{n+1}\coloneqq\frac{\dot{w}^{n+1}-\dot{w}^n}{\delta t} = \frac{w^{n+1} - 2w^n + w^{n-1}}{(\delta t)^2}
\end{equation}
Hence, the semi-discrete weak form of the 3D fluid - 2D plate interaction system is 
\begin{align} \label{eq:momentum_disc_weak}
    &\begin{aligned}
    &\rho_f(\bm{u}^{n+1}, \bm{v})_{\Omega_f} + \nu_f\delta t(\nabla \bm{u}^{n+1}, \nabla\bm{v})_{\Omega_f} - \delta t(p^{n+1}_0, \nabla\cdot \bm{v})_{\Omega_f} + \delta t\dual{g^{n+1}, v_3}_{\Omega_p}\\
    &\qquad = \delta t(\bm{f}^{n+1}, \bm{v})_{\Omega_f} + \rho_f (\bm{u}^n, \bm{v})_{\Omega_f}
    \end{aligned}
     &&\forall \bm{v}\in \bm{U},\\
    &(\nabla\cdot \bm{u}^{n+1}, q)_{\Omega_f} = 0 &&\forall q\in Q_0,\label{eq:continuity_disc_weak} \\
    &\begin{aligned}
    &\frac{1}{\delta t}(w^{n+1}, \eta)_{\Omega_p} + \frac{\rho}{\delta t}(z^{n+1}, \eta)_{\Omega_p} +\delta t (\nabla z^{n+1}, \nabla \eta)_{\Omega_p} - \delta t\dual{g^{n+1}, \eta}_{\Omega_p} \\
    &- \delta t(s^{n+1}, \eta)_{\Omega_p} = \frac{2}{\delta t}(w^n, \eta)_{\Omega_p} - \frac{1}{\delta t}(w^{n-1}, \eta)_{\Omega_p} + \frac{2\rho}{\delta t} (z^n, \eta)_{\Omega_p} - \frac{\rho}{\delta t} (z^{n-1}, \eta)_{\Omega_p} 
    \end{aligned}
    &&\forall \eta\in W, \label{eq:plate_disc_weak}\\
    & (z^{n+1}, \varphi)_{\Omega_p} = (\nabla w^{n+1}, \nabla \varphi)_{\Omega_p} &&\forall \varphi\in W, \label{eq:poisson_semi_disc_weak} \\
    & \dual{u_3^{n+1}, \lambda}_{\Omega_p} - \frac{1}{\delta t}\dual{w^{n+1}, \lambda}_{\Omega_p} = \frac{1}{\delta t}\dual{w^{n}, \lambda}_{\Omega_p} &&\forall \lambda \in G. \label{eq:interface_disc_weak}
\end{align}
From \eqref{eq:wdot_zeromean} we note that $(\partial_t w,c)_{\Omega_p} = 0$ for any $c \in \mathbb{R}$ and
\eqref{eq:interface_weak} can be written as 
$\dual{u_3 - \partial_t w , \lambda}_{\Omega_p} - 
(\partial_t w, c)_{\Omega_p}= 0$. Thus, \eqref{eq:interface_disc_weak}
may be replaced by 
$$
\dual{u_3^{n+1}, \lambda}_{\Omega_p} -\frac{1}{\delta t}\dual{w^{n+1}, \lambda+c}_{\Omega_p}  = \frac{1}{\delta t}\dual{w^{n}, \lambda+c}_{\Omega_p} \quad \forall 
(\lambda,c) \in G \times \mathbb{R}.
$$
Therefore, 
this system, \eqref{eq:momentum_disc_weak}-\eqref{eq:interface_disc_weak}, may be stated as the following saddle point system: \\find $\left(\bm{u}^{n+1}, w^{n+1}, z^{n+1}, p^{n+1}, g^{n+1}, s^{n+1}\right)\in \bm{U}\times W\times W\times Q\times G \times \mathbb{R}$ such that
\begin{equation} \label{eq:saddle_point_prob}
    \begin{aligned}
        a_{\delta t}\left((\bm{u}^{n+1}, w^{n+1}, z^{n+1}), (\bm{v}, \varphi, \eta)\right) + b_{\delta t}\left((\bm{v}, \eta), \left(\delta t p_0^{n+1}, \delta t g^{n+1}, \delta t s^{n+1}\right)\right) &= f^1_{\delta t}(\bm{v}, \varphi, \eta), \\
        b_{\delta t}\left(\left(\bm{u}^{n+1}, \frac{1}{\delta t}w^{n+1}\right),(q, \lambda, c)\right) &= f_{\delta t}^2 (q, \lambda),
    \end{aligned}
\end{equation}
where $a_{\delta t}: (\bm{U}\times W\times W) \times (\bm{U}\times W\times W) \to \mathbb{R}$, $b_{\delta t}:(\bm{U}\times W)\times (Q\times G) \to \mathbb{R}$, $f_{\delta t}^1:\bm{U}\times W\times W \to \mathbb{R}$, %{\color{red} W instead of Z?}
and $f_{\delta t}^2:Q\times G \to \mathbb{R}$ are defined by
\begin{align}
    &\begin{aligned}
        a_{\delta t}&\left((\bm{u}, w, z), (\bm{v}, \varphi, \eta) \right) \coloneqq \rho_f(\bm{u}, \bm{v})_{\Omega_f} + \nu_f\delta t(\nabla\bm{u}, \nabla\bm{v})_{\Omega_f} + \frac{1}{\delta t}(w, \eta)_{\Omega_p} + \frac{\rho}{\delta t}(z, \eta)_{\Omega_p} \\
        & + \delta t(\nabla z, \nabla\eta) - \frac{1}{\delta t}(z, \varphi)_{\Omega_p} + \frac{1}{\delta t}(\nabla w, \nabla \varphi)_{\Omega_p}
    \end{aligned}\\
    & b_{\delta t}\left((\bm{v}, \eta), (q, \lambda, c)\right) \coloneqq -(\nabla\cdot \bm{v}, q)_{\Omega_f} + \dual{v_3 - \eta, \lambda}_{\Omega_p} - (c, \eta)_{\Omega_p},\\
    &\begin{aligned}
        f_{\delta t}^1&(\bm{v}, \varphi, \eta) \coloneqq \delta t(\bm{f}^{n+1}, \bm{v})_{\Omega_f} + \rho_f (\bm{u}^n, \bm{v})_{\Omega_f} + \frac{2}{\delta t}(w^n, \eta)_{\Omega_p} - \frac{1}{\delta t}(w^{n-1}, \eta)_{\Omega_p} \\
        & + \frac{2\rho}{\delta t} (\nabla w^n, \nabla\eta)_{\Omega_p} - \frac{\rho}{\delta t} (\nabla w^{n-1}, \nabla\eta)_{\Omega_p},
    \end{aligned}\\
    & f_{\delta t}^2(q, \lambda, c) \coloneqq  \frac{1}{\delta t}\dual{w^{n}, \lambda}_{\Omega_p} + \frac{1}{\delta t}(w^{n}, c)_{\Omega_p}.
\end{align}

Note that $a_{\delta t}(\cdot, \cdot)$ can be easily shown to be coercive. Let $X = \bm{U}\times W\times W$ and  $(\bm{v}, \varphi, \eta)\in X$ with norm
\begin{equation*}
    \norm{(\bm{v}, \varphi, \eta)}_{X}^2 = \norm{\bm{v}}_{1,\Omega_f}^2 + \norm{\varphi}_{1,\Omega_p}^2 + \norm{\eta}_{1,\Omega_p}^2.
\end{equation*}
Then, we see that by the Poincar\'{e}-Friedrich inequality, we have
\begin{equation} \label{eq:semi_disc_coercive}
    \begin{aligned}
        a_{\delta t}\left((\bm{v}, \varphi, \eta), (\bm{v}, \varphi, \eta) \right) &= \rho_f\norm{\bm{v}}_{\Omega_f}^2 + \nu_f\delta t\norm{\nabla\bm{v}}_{\Omega_f}^2 + \frac{\rho}{\delta t}\norm{\eta}_{\Omega_p}^2 + \delta t\norm{\nabla\eta}_{\Omega_p}^2 + \frac{1}{\delta t}\norm{\nabla\varphi}_{\Omega_p}^2 \\
        &\geq \rho_f\norm{\bm{v}}_{\Omega_f}^2 + \nu_f\delta t\norm{\nabla\bm{v}}_{\Omega_f}^2 + \frac{\rho}{\delta t}\norm{\eta}_{\Omega_p}^2 + \delta t\norm{\nabla\eta}_{\Omega_p}^2 + \frac{1}{2\delta t(C_P^2 + 1)}\norm{\varphi}_{1,\Omega_p}^2\\
        &\geq \alpha  \norm{(\bm{v}, \varphi, \eta)}_{X}^2
    \end{aligned}
\end{equation}
 where $\alpha = \min\{\rho_f, \nu_f\delta t, \rho/\delta t, \delta t, 1/(2\delta t (C_P^2 + 1))\}$ and $C_P$ is the Poincar\'{e} constant. \\

Next, we show that the bilinear operator $b_{\delta t}(\cdot, \cdot)$ satisfies the inf-sup condition. 
%First, we recall a result in \cite{de_Castro2025} {\color{red} [20] instead of [28]?} which is necessary for the proof of the inf-sup condition.

\begin{lemma} \label{thm:semi-disc-infsup}
    There exists a positive constant $\beta$ such that
    \begin{equation}
        \sup_{(\bm{v}, \eta)\in \bm{U}\times W}\frac{b_{\delta t}((\bm{v}, \eta), (q, \lambda, c))}{\norm{\bm{v}}_{1,\Omega_f} + \norm{\eta}_{1,\Omega_p}} \geq \beta \left(\norm{q}_{\Omega_f} + \norm{\lambda}_{H^{1/2}(\Omega_p)} + \abs{c}\right)
    \end{equation}
\end{lemma}
\begin{proof}
    Let $(\tilde{q}, \tilde{\lambda}, \tilde{c})\in Q_0\times G \times \mathbb{R}$ be given. We choose a $\tilde{w} = \tilde{w}_1 + \tilde{w}_2\in W$ such that 
    \begin{equation} \label{eq:infsup_poisson_1}
        \begin{aligned}
            (\nabla \tilde{w}_1, \nabla\varphi_1)_{\Omega_p} = -(\tilde{\lambda}, \varphi_1)_{\Omega_p} \quad \forall \varphi_1\in W,\\
            (\nabla \tilde{w}_2, \nabla\varphi_2)_{\Omega_p} = -(\tilde{c}, \varphi_2)_{\Omega_p} \quad \forall \varphi_2\in W.
        \end{aligned}
    \end{equation}
    Note that by setting $\varphi_1 = \tilde{w}_2, \varphi_2 = \tilde{w}_2$ and using Poincar\'{e}-Friedrich inequality, this implies that
    \begin{equation} \label{eq:infsup_poisson_2}
        \begin{aligned}
            \norm{\nabla\tilde{w}_1}_{\Omega_p}^2 = -(\tilde{\lambda}, \tilde{w}_1)_{\Omega_p} &\text{ and } \norm{\nabla\tilde{w}_1}_{\Omega_p} \leq D_1 \norm{\tilde{\lambda}}_{H^{-1/2}(\Omega_p)},\\
            \norm{\nabla\tilde{w}_2}_{\Omega_p}^2 = -(\tilde{c}, \tilde{w}_2)_{\Omega_p} &\text{ and } \norm{\nabla\tilde{w}_2}_{\Omega_p} \leq D_2 \abs{\tilde{c}}.
        \end{aligned}
    \end{equation}
    for some $D_1,D_2>0$. This also means that
    \begin{equation} \label{eq:infsup_poisson_4}
        \left(\nabla (\tilde{w}_1 + \tilde{w}_2), \nabla\varphi\right)_{\Omega_p} = -(\tilde{\lambda} + \tilde{c}, \varphi)_{\Omega_p} \quad \forall \varphi\in W.
    \end{equation}
    From \eqref{eq:infsup_poisson_1}-\eqref{eq:infsup_poisson_2}, we obtain
    \begin{equation}
        \norm{\nabla\tilde{w}_2}^{2}_{\Omega_p} = \abs{\tilde{c}}^2\norm{\nabla\tilde{\phi}}_{\Omega_p}^2
    \end{equation}
    where $\tilde{\phi}\in W$ is the solution to 
    \begin{equation}
        (\nabla\tilde{\phi}, \nabla\varphi)_{\Omega_p} = (1, \varphi)_{\Omega_p}, \quad \forall\varphi\in W.
    \end{equation}
    Hence, by \eqref{eq:infsup_poisson_2} and Poincar\'{e}-Friedrich inequality, we have 
    \begin{equation} \label{eq:infsup_poisson_3}
        \norm{\tilde{w}}_{1,\Omega_p}\leq C_0 \left(\norm{\tilde{\lambda}}_{H^{-1/2}(\Omega_p)} + \abs{\tilde{c}}\right)
    \end{equation}
    for some $C_0>0$. Now, following the strategy in \cite{Avalos2025}, we choose $\tilde{\bm{u}}_0\in \bm{U}$ be the solution to the divergence problem
    \begin{equation} \label{eq:u0_div_prob}
        \begin{cases}
            \nabla\cdot\tilde{\bm{u}}_0 = -\tilde{q} & \text{ in }\Omega_f,\\
            \tilde{\bm{u}}_0 = \bm{0} & \text{ on }\partial\Omega_f.
        \end{cases}
    \end{equation}
    Note that with this, we have
    \begin{equation} \label{eq:u0_bound}
        \norm{\nabla\tilde{\bm{u}}_0}_{\Omega_f}\leq C_1
        \norm{\tilde{q}}_{\Omega_f}.
    \end{equation}
    By the Riesz Representation Theorem, there exists a unique $\xi\in H^{1/2}(\Omega_p)$ such that for any $\theta\in H^{1/2}(\Omega_p)$, 
    \begin{equation}
        \dual{\tilde{\lambda}, \theta}_{\Omega_p} = \dual{\xi, \theta}_{1/2,\Omega_p} \text{ with }\norm{\tilde{\lambda}}_{-1/2,\Omega_p} = \norm{\tilde{\xi}}_{1/2,(\Omega_p)}.
    \end{equation}
     Setting $\tilde{\xi} = \xi/\norm{\xi}_{1/2,\Omega_p}$ gives
    \begin{equation} \label{eq:xi_norm}
        \dual{\tilde{\lambda}, \tilde{\xi}}_{\Omega_p} = \dual{\xi, \frac{\xi}{\norm{\xi}_{H^{1/2}(\Omega_p)}}}_{1/2,\Omega_p} = \norm{\xi}_{1/2,\Omega_p} = \norm{\tilde{\lambda}}_{-1/2,(\Omega_p)}.
    \end{equation}
    Note that $\xi/\norm{\xi}_{1/2,\Omega_p}=1$. 
    Now, we choose $\tilde{\bm{u}}_1\in\bm{U}$ to be the solution to the problem
    \begin{equation} \label{eq:u1_div_prob}
        \begin{cases}
            \nabla\cdot\tilde{\bm{u}}_1 = \frac{\epsilon}{\abs{\Omega_f}}\norm{\tilde{\lambda}}_{-1/2,\Omega_p}\int_{\Omega_p} \tilde{\xi}\,d\Omega_p & \text{ in }\Omega_f,\\
            \tilde{\bm{u}}_0 = \begin{cases}
                \bm{0} & \text{ on } S,\\
                \left[0, 0,\epsilon\norm{\tilde{\lambda}}_{-1/2,\Omega_p}\tilde{\xi}\right] & \text{ on }\Omega_p.
            \end{cases}
        \end{cases}
    \end{equation}
    for some $\epsilon>0$, where $\abs{\Omega_f}$ is the volume of $\Omega_f$.  Note that $( \nabla\cdot\tilde{\bm{u}}_1, q)_{\Omega_f} = 0 \; \forall q \in Q_0$.
      We also obtain from \eqref{eq:u1_div_prob} that 
    \begin{equation} \label{eq:u1_bound}
        \norm{\nabla\tilde{\bm{u}}_1}_{\Omega_f}\leq C_2\norm{\tilde{\lambda}}_{-1/2,\Omega_p}
    \end{equation}
    as $\norm{\tilde{\xi}}_{L^1(\Omega_p)} \le \norm{\tilde{\xi}}_{1/2,\Omega_p}=1$.
    Now, we set $\tilde{\bm{u}} = \tilde{\bm{u}}_0 + \tilde{\bm{u}}_1$ and this implies
    \begin{equation} \label{eq:u_h1_norm}
        \norm{\tilde{\bm{u}}}_{1,\Omega_f} \leq C \left(\norm{\tilde{q}}_{\Omega_f} + \norm{\tilde{\lambda}}_{-1/2,\Omega_p}\right).
    \end{equation}
    by \eqref{eq:u0_bound} and \eqref{eq:u1_bound}
    We also note that by \eqref{eq:xi_norm}-\eqref{eq:u1_div_prob}, we have
    \begin{equation}\label{eq:vel_LM_rel}
        \dual{\tilde{\lambda}, \tilde{u}_3}_{\Omega_p} = \dual{\tilde{\lambda}, (\tilde{\bm{u}}_1)_3}_{\Omega_p} = \epsilon\norm{\tilde{\lambda}}_{-1/2,\Omega_p}\dual{\tilde{\lambda}, \tilde{\xi}}_{\Omega_p} = \epsilon\norm{\tilde{\lambda}}_{-1/2,\Omega_p}^2
    \end{equation}
    Then, by \eqref{eq:u0_div_prob}, \eqref{eq:vel_LM_rel}, \eqref{eq:infsup_poisson_2}-\eqref{eq:infsup_poisson_3}, we obtain
    \begin{equation}
        \begin{aligned}\label{eq:infsup1}
            b_{\delta t}\left((\tilde{\bm{u}}, \tilde{w}), (\tilde{q}, \tilde{\lambda}, \tilde{c})\right) & = - (\nabla \cdot \tilde{\bm{u}}, \tilde{q})_{\Omega_f} + \dual{\tilde{u}_3, \tilde{\lambda}}_{\Omega_p} - \dual{\tilde{\lambda}, \tilde{w}}_{\Omega_p} - (\tilde{c}, \tilde{w})_{\Omega_p} \\
            &= \norm{\tilde{q}}_{\Omega_f}^2 + \epsilon\norm{\tilde{\lambda}}_{-1/2,\Omega_p}^2 + \norm{\nabla\tilde{w}_1+\nabla\tilde{w}_2}_{\Omega_p}^2\\
            & \geq \norm{\tilde{q}}_{\Omega_f}^2 + \epsilon\norm{\tilde{\lambda}}_{-1/2,\Omega_p}^2 + \norm{\nabla\tilde{w}_1}^2_{\Omega_p} + 2(\nabla\tilde{w}_1, \nabla\tilde{w}_2)_{\Omega_p} + \norm{\nabla\tilde{w}_2}^2_{\Omega_p}\\
            & \geq \norm{\tilde{q}}_{\Omega_f}^2 + \epsilon\norm{\tilde{\lambda}}_{-1/2,\Omega_p}^2 - \left(\frac{1}{\epsilon_1}-1\right)\norm{\nabla\tilde{w}_1}^2_{\Omega_p} + \left(1-\epsilon_1\right)\norm{\nabla\tilde{w}_2}^2_{\Omega_p}\\
            & \geq \norm{\tilde{q}}_{\Omega_f}^2 + \left(\epsilon - D_1^2\left(\frac{1}{\epsilon_1}-1\right)\right)\norm{\tilde{\lambda}}_{-1/2,\Omega_p}^2 + \left(1-\epsilon_1\right)\norm{\nabla\tilde{w}_2}^2_{\Omega_p}\\
            & \geq \norm{\tilde{q}}_{\Omega_f}^2 + \left(\epsilon - D_1^2\left(\frac{1}{\epsilon_1}-1\right)\right)\norm{\tilde{\lambda}}_{-1/2,\Omega_p}^2 + C_r \abs{\tilde{c}}^2\\
            & \geq \tilde{D}\left(\norm{\tilde{q}}_{\Omega_f}^2 + \norm{\tilde{\lambda}}_{-1/2,\Omega_p}^2 + \abs{\tilde{c}}^2\right)
        \end{aligned}
    \end{equation}
    where we choose $\epsilon > D_1^2\left(\frac{1}{\epsilon_1}-1\right)$, $\tilde{D} = \min\left\{1, \epsilon - D_1^2\left(\frac{1}{\epsilon_1}-1\right), C_r\right\}$,  $C_r = (1-\epsilon_1)\norm{\nabla\tilde{\phi}}_{\Omega_p}^2$. By \eqref{eq:infsup_poisson_3} and \eqref{eq:u_h1_norm},  we have
    \begin{equation}\label{eq:h1norm_bound}
        \norm{\tilde{\bm{u}}}_{1,\Omega_f} + \norm{\tilde{w}}_{1,\Omega_p}\leq \tilde{C}\left(\norm{\tilde{q}}_{\Omega_f} + \norm{\tilde{\lambda}}_{-1/2,\Omega_p} + \abs{c}\right).
    \end{equation}
    Using the identity $a^2 + b^2 + b^2 \geq (1/3)(a+b +c)^2$, \eqref{eq:infsup1}, and \eqref{eq:h1norm_bound}, we obtain
    \begin{equation}
        \begin{aligned}
            b_{\delta t}\left((\tilde{\bm{u}}, \tilde{w}), (\tilde{q}, \tilde{\lambda})\right) & \geq \frac{\tilde{D}}{3}\left(\norm{\tilde{q}}_{\Omega_f} + \norm{\tilde{\lambda}}_{-1/2,\Omega_p} + \abs{\tilde{c}}\right)^2 \\
            & \geq \frac{\tilde{D}}{3\tilde{C}}\left(\norm{\tilde{q}}_{\Omega_f} + \norm{\tilde{\lambda}}_{-1/2,\Omega_p} + \abs{\tilde{c}}\right)\left(\norm{\tilde{\bm{u}}}_{1,\Omega_f} + \norm{\tilde{w}}_{1,\Omega_p}\right) \\
        \end{aligned}
    \end{equation}
    Hence,
    \begin{equation*}
        \frac{b_{\delta t}\left((\tilde{\bm{u}}, \tilde{w}), (\tilde{q}, \tilde{\lambda})\right)}{\norm{\tilde{\bm{u}}}_{1,\Omega_f} + \norm{\tilde{w}}_{1,\Omega_p}} \geq \beta \left(\norm{\tilde{q}}_{\Omega_f} + \norm{\tilde{\lambda}}_{-1/2,\Omega_p} + \abs{c}\right).
    \end{equation*}
    Since $(\tilde{q}, \tilde{\lambda}, \tilde{c})\in Q_0\times G \times \mathbb{R}$ is arbitrary, taking the supremum over $\bm{0}\neq(\bm{v}, \eta)\in \bm{U}\times W$, shows the inf-sup condition for $b_{\delta t}(\cdot, \cdot)$.
\end{proof}

%{\color{red} Add a theorem for the well-posedness of (26).}

From \eqref{eq:semi_disc_coercive}, Theorem \ref{thm:semi-disc-infsup}, and the theory for the existence and uniqueness of solutions to saddle point systems, we present the well-posedness of the semi-discrete problem \eqref{eq:momentum_disc_weak}-\eqref{eq:interface_disc_weak}.

\begin{theorem}
    The semi-discrete problem \eqref{eq:saddle_point_prob} has a unique solution $\left((\bm{u}^{n+1}, w^{n+1}, z^{n+1}), (p^{n+1}_0, g^{n+1}, s^{n+1})\right)\in \left(\bm{U}\times W\times W\right)\times \left(Q_0\times G\times \mathbb{R}\right)$.
\end{theorem}

Next, we show the stability of \eqref{eq:momentum_disc_weak}-\eqref{eq:interface_disc_weak} in the following theorem.
\begin{theorem}
    Suppose $\bm{f}\in [L^2(\Omega_f)]^3$. Then, we have the following estimates for \eqref{eq:momentum_disc_weak}-\eqref{eq:interface_disc_weak}, 
    \begin{equation} \label{eq:energy_est}
        \begin{aligned}
            \frac{1}{2}&\norm{\bm{u}^N}_{\Omega_f}^2 + \frac{\nu_f\delta t}{2}\sum_{n=0}^{N-1}\norm{\nabla\bm{u}^{n+1}}_{\Omega_f}^2 + \frac{1}{2}\norm{\dot{w}^N}^2_{\Omega_p} + \frac{\rho}{2}\norm{\nabla\dot{w}^N}^2_{\Omega_p} + \frac{1}{2}\norm{z^N}_{\Omega_p}^2 + \sum_{n=0}^{N-1}\norm{g^{n+1}}_{-1/2,\Omega_p}^2 \\
            & + \frac{(\delta t)^2}{2}\sum_{n=0}^{N-1}\left(\norm{\dot{\bm{u}}^{n+1}}_{\Omega_f}^2 + \norm{\ddot{w}^{n+1}}_{\Omega_p}^2 + \rho\norm{\nabla\ddot{w}^{n+1}}_{\Omega_p}^2\right) \leq \left(\frac{(\delta t)^2C^2}{2\nu_f} + C_1\right)\sum_{n=0}^{N-1}\norm{\bm{f}^{n+1}}_{\Omega_f}^2,
        \end{aligned}
    \end{equation}
    \begin{align}
        &\norm{\nabla w^{n+1}}_{\Omega_p}^2\leq C_P^2 \norm{z^{n+1}}^2_{\Omega_p}, \label{eq:plate_estimate}
        %&\norm{\tilde{p}^{n+1}}_{\Omega_f}^2\leq C_\beta\left(\norm{\tilde{\bm{f}}^{n+1}}_{\Omega_f}^2 + \norm{\tilde{g}^{n+1}}_{H^{-1/2}(\Omega_p)}\right) \label{eq:pressure_estimate}
    \end{align}
    for some constants $C, C_1, C_P > 0$ independent of the time step $\delta t$.
\end{theorem}

\begin{proof}
    For simplicity, we assume that the systems starts from rest, i.e., $\bm{u} = \bm{0}$, and $w = w_{t0} z=z_{t0}= 0$. Setting $\bm{v} = \bm{u}^{n+1}$ and $q = p_0^{n+1}$ in \eqref{eq:momentum_disc_weak}-\eqref{eq:continuity_disc_weak} and using the modified polarization identity $(a-b, a)_{\gamma} = \frac{1}{2}(\norm{a}_\gamma^2-\norm{b}_\gamma^2+\norm{a-b}_\gamma^2)$, we have
    \begin{equation}\label{eq:fluid_ineq}
        \begin{aligned}
            \frac{\rho_f}{2}\left(\norm{\bm{u}^{n+1}}_{\Omega_f}^2 - \norm{\bm{u}^{n}}_{\Omega_f}^2\right) + \frac{\rho_f}{2}\norm{\bm{u}^{n+1}-\bm{u}^{n}}_{\Omega_f}^2 + \nu_f\delta t\norm{\nabla\bm{u}^{n+1}}^2_{\Omega_f} + \delta t\dual{g^{n+1}, u_3^{n+1}}_{\Omega_p}\\
            = \delta t(\bm{f}^{n+1}, \bm{u}^{n+1})_{\Omega_f}.
        \end{aligned}
    \end{equation}
    Using Cauchy-Schwarz inequality, Young's inequality, and Poinca\'{e}-Friedrich's inequality, we obtain
    \begin{equation}\label{eq:fluid_ineq_2}
        \begin{aligned}
            \frac{\rho_f}{2}\left(\norm{\bm{u}^{n+1}}_{\Omega_f}^2 - \norm{\bm{u}^{n}}_{\Omega_f}^2\right) + \frac{\rho_f}{2}\norm{\bm{u}^{n+1}-\bm{u}^{n}}_{\Omega_f}^2 &+ \nu_f\delta t\norm{\nabla\bm{u}^{n+1}}^2_{\Omega_f} + \delta t\dual{g^{n+1}, u_3^{n+1}}_{\Omega_p}
            \\
            &\leq \frac{\delta t\epsilon}{2}\norm{\bm{f}^{n+1}}_{\Omega_f}^2 + \frac{\delta t C^2}{2\epsilon}\norm{\nabla\bm{u}^{n+1}}_{\Omega_f}^2.
        \end{aligned}
    \end{equation}
    Setting $\epsilon=\frac{C^2}{\nu_f}$ and using the notation in \eqref{eq:back_euler}, we get
    \begin{equation}\label{eq:fluid_ineq_3}
        \begin{aligned}
            \frac{\rho_f}{2}\left(\norm{\bm{u}^{n+1}}_{\Omega_f}^2 - \norm{\bm{u}^{n}}_{\Omega_f}^2\right) + \frac{\rho_f(\delta t)^2}{2}\norm{\dot{\bm{u}}^{n+1}}_{\Omega_f}^2 + \frac{\nu_f\delta t}{2}\norm{\nabla\bm{u}^{n+1}}^2_{\Omega_f} + \delta t\dual{g^{n+1}, u_3^{n+1}}_{\Omega_p} \\
            \leq \frac{\delta tC^2}{2\nu_f}\norm{\bm{f}^{n+1}}_{\Omega_f}^2.
        \end{aligned}
    \end{equation}
    Note that using \eqref{eq:back_euler}, we can rewrite \eqref{eq:plate_disc_weak} as
    \begin{equation} \label{eq:plate_weak_1st}
        (\dot{w}^{n+1}-\dot{w}^{n}, \eta)_{\Omega_p} + \rho(\dot{z}^{n+1}-\dot{z}^n, \eta)_{\Omega_p} + \delta t(\nabla z^{n+1}, \nabla\eta)_{\Omega_p} - \delta t\dual{g^{n+1}, \eta}_{\Omega_p} - \delta t(s^{n+1}, \eta)_{\Omega_p} = 0.
    \end{equation}
   If we set $\varphi = \eta$ in \eqref{eq:poisson_semi_disc_weak}, and taking the difference between time $t = t^{n+1}$ and $t = t^n$, we obtain
    \begin{equation*}
        (z^{n+1} - z^{n}, \eta)_{\Omega_p} = (\nabla w^{n+1} - \nabla w^n, \nabla \eta)_{\Omega_p}.
    \end{equation*}
    Using \eqref{eq:back_euler} and taking the difference again between time $t = t^{n+1}$ and $t = t^n$, we get
    \begin{equation} \label{eq:rotation_switch}
       (\dot{z}^{n+1} - \dot{z}^n, \eta)_{\Omega_p} = (\nabla \dot{w}^{n+1} - \nabla \dot{w}^{n}, \nabla \eta)_{\Omega_p}
    \end{equation}
    Hence, plugging \eqref{eq:rotation_switch} into the rotation term in \eqref{eq:fluid_ineq_3}, we have
    \begin{equation} \label{eq:plate_weak_1st}
        (\dot{w}^{n+1}-\dot{w}^{n}, \eta)_{\Omega_p} + \rho(\nabla\dot{w}^{n+1}-\nabla\dot{w}^n, \nabla\eta)_{\Omega_p} + \delta t(\nabla z^{n+1}, \nabla\eta)_{\Omega_p} - \delta t\dual{g^{n+1}, \eta}_{\Omega_p} - \delta t(s^{n+1}, \eta)_{\Omega_p} = 0.
    \end{equation}
    We set $\varphi=\dot{w}^{n+1}$ and use the polarization identity again to obtain
    \begin{equation} \label{eq:plate_weak_2nd}
        \begin{aligned}
        \frac{1}{2}&\left(\norm{\dot{w}^{n+1}}^2_{\Omega_p}-\norm{\dot{w}^{n}}^2_{\Omega_p}\right) + \frac{1}{2}\norm{\dot{w}^{n+1}-\dot{w}^n}_{\Omega_p}^2 + \frac{\rho}{2}\left(\norm{\nabla\dot{w}^{n+1}}^2_{\Omega_p}-\norm{\nabla\dot{w}^{n}}^2_{\Omega_p}\right) \\
        &+ \frac{\rho}{2}\norm{\nabla\dot{w}^{n+1}-\nabla\dot{w}^n}_{\Omega_p}^2 + \delta t(\nabla z^{n+1}, \nabla\dot{w}^{n+1})_{\Omega_p} - \delta t\dual{g^{n+1}, \dot{w}^{n+1}}_{\Omega_p} - \delta t(s^{n+1}, \dot{w}^{n+1})_{\Omega_p} = 0 
        \end{aligned}
    \end{equation}
    Now, setting $\eta=z^{n+1}$ in \eqref{eq:poisson_semi_disc_weak} gives
    \begin{equation} \label{eq:poisson_n+1}
        \norm{z^{n+1}}_{\Omega_p}^2=(\nabla z^{n+1}, \nabla w^{n+1})_{\Omega_p}.
    \end{equation}
  Since \eqref{eq:poisson_semi_disc_weak} holds for all time, at $t = t^n$ with $\eta = z^{n+1}$, we get
    \begin{equation} \label{eq:poisson_n}
        (z^n, z^{n+1})_{\Omega_p} = (\nabla w^{n}, \nabla z^{n+1})_{\Omega_p}.
    \end{equation}
    Using \eqref{eq:back_euler} and \eqref{eq:poisson_n+1}-\eqref{eq:poisson_n}, we observe that
    \begin{equation} \label{eq:cross_term}
        \delta t(\nabla z^{n+1}, \nabla\dot{w}^{n+1})_{\Omega_p} = (\nabla z^{n+1}, \nabla w^{n+1})_{\Omega_p} - (\nabla z^{n+1}, \nabla w^{n})_{\Omega_p} = \norm{z^{n+1}}^2_{\Omega_p} - (z^n, z^{n+1})_{\Omega_p}.
    \end{equation}
    Hence, using \eqref{eq:back_euler} and \eqref{eq:cross_term}, and letting $\eta=\dot{w}^{n+1}$ we obtain from \eqref{eq:plate_weak_2nd} that 
    \begin{equation}
        \begin{aligned}
        \frac{1}{2}\left(\norm{\dot{w}^{n+1}}^2_{\Omega_p}-\norm{\dot{w}^{n}}^2_{\Omega_p}\right) &+ \frac{(\delta t)^2}{2}\norm{\ddot{w}^{n+1}}_{\Omega_p}^2 + \frac{\rho}{2}\left(\norm{\nabla\dot{w}^{n+1}}^2_{\Omega_p}-\norm{\nabla\dot{w}^{n}}^2_{\Omega_p}\right) + \frac{\rho(\delta t)^2}{2}\norm{\nabla\ddot{w}^{n+1}}_{\Omega_p}^2 \\
        & + \norm{z^{n+1}}_{\Omega_p}^2 - \delta t\dual{g^{n+1}, \dot{w}^{n+1}}_{\Omega_p} = (z^n, z^{n+1})_{\Omega_p} + \delta t(s^{n+1}, \dot{w}^{n+1})_{\Omega_p}.
        \end{aligned}
    \end{equation}
    By the Cauchy-Schwarz inequality, and Young's inequality, we have
    \begin{equation}
        \begin{aligned} \label{eq:plate_ineq_3}
        \frac{1}{2}\left(\norm{\dot{w}^{n+1}}^2_{\Omega_p}-\norm{\dot{w}^{n}}^2_{\Omega_p}\right) &+ \frac{(\delta t)^2}{2}\norm{\ddot{w}^{n+1}}_{\Omega_p}^2 + \frac{\rho}{2}\left(\norm{\nabla\dot{w}^{n+1}}^2_{\Omega_p}-\norm{\nabla\dot{w}^{n}}^2_{\Omega_p}\right) + \frac{\rho(\delta t)^2}{2}\norm{\nabla\ddot{w}^{n+1}}_{\Omega_p}^2 \\
        & + \frac{1}{2}\left(\norm{z^{n+1}}_{\Omega_p}^2 - \norm{z^{n}}_{\Omega_p}^2\right) - \delta t\dual{g^{n+1}, \dot{w}^{n+1}}_{-1/2,\Omega_p} \leq \delta t(s^{n+1}, \dot{w}^{n+1})_{\Omega_p}.
        \end{aligned}
    \end{equation}
    Note that setting $\lambda = g^{n+1}$ in \eqref{eq:interface_disc_weak} gives
    \begin{equation}
        \dual{\dot{w}^{n+1}, g^{n+1}}_{\Omega_p} = \dual{u^{n+1}_3, g^{n+1}}_{\Omega_p}.
    \end{equation}
    With this and adding \eqref{eq:fluid_ineq_3} and \eqref{eq:plate_ineq_3} result to
    \begin{equation*}
        \begin{aligned}
            \frac{\rho_f}{2}&\left(\norm{\bm{u}^{n+1}}_{\Omega_f}^2 - \norm{\bm{u}^{n}}_{\Omega_f}^2\right) + \frac{\rho_f(\delta t)^2}{2}\norm{\dot{\bm{u}}^{n+1}}_{\Omega_f}^2 + \frac{\nu_f\delta t}{2}\norm{\nabla\bm{u}^{n+1}}^2_{\Omega_f} \\
            & + \frac{1}{2}\left(\norm{\dot{w}^{n+1}}^2_{\Omega_p}-\norm{\dot{w}^{n}}^2_{\Omega_p}\right) + \frac{(\delta t)^2}{2}\norm{\ddot{w}^{n+1}}_{\Omega_p}^2 + \frac{\rho}{2}\left(\norm{\nabla\dot{w}^{n+1}}^2_{\Omega_p} -\norm{\nabla\dot{w}^{n}}^2_{\Omega_p}\right) \\
            & + \frac{\rho(\delta t)^2}{2}\norm{\nabla\ddot{w}^{n+1}}_{\Omega_p}^2 + \frac{1}{2}\left(\norm{z^{n+1}}_{\Omega_p}^2 - \norm{z^{n}}_{\Omega_p}^2\right) \leq \frac{\delta tC^2}{2\nu_f}\norm{\bm{f}^{n+1}}_{\Omega_f}^2 + \delta t(s^{n+1}, \dot{w}^{n+1})_{\Omega_p}.
        \end{aligned}
    \end{equation*}
    From \eqref{eq:wdot_zeromean} and the fact that $s^{n+1}\in\mathbb{R}$, we see that
    \begin{equation*}
        (s^{n+1}, \dot{w}^{n+1})_{\Omega_p} = \int_{\Omega_p}s^{n+1}\dot{w}^{n+1}\,d\Omega_p = s^{n+1}\int_{\Omega_p}\dot{w}^{n+1}\,d\Omega_p = 0.
    \end{equation*}
    Then, summing $n=0, \dots, N-1$, we obtain
    \begin{equation}\label{eq:energy_1}
        \begin{aligned}
            \frac{\rho_f}{2}&\norm{\bm{u}^N}_{\Omega_f}^2 + \frac{\nu_f\delta t}{2}\sum_{n=0}^{N-1}\norm{\nabla\bm{u}^{n+1}}_{\Omega_f}^2 + \frac{1}{2}\norm{\dot{w}^N}^2_{\Omega_p} + \frac{\rho}{2}\norm{\nabla\dot{w}^N}^2_{\Omega_p} + \frac{1}{2}\norm{z^N}_{\Omega_p}^2 \\
            & + \frac{(\delta t)^2}{2}\sum_{n=0}^{N-1}\left(\rho_f\norm{\dot{\bm{u}}^{n+1}}_{\Omega_f}^2 + \norm{\ddot{w}^{n+1}}_{\Omega_p}^2 + \rho\norm{\nabla\ddot{w}^{n+1}}_{\Omega_p}^2\right) \leq \frac{(\delta t)^2C^2}{2\nu_f}\sum_{n=0}^{N-1}\norm{\bm{f}^{n+1}}_{\Omega_f}^2.
        \end{aligned}
    \end{equation}
    Now, setting $\bm{v}\in \bm{V}$ and isolating the term with the Lagrange multiplier $g^{n+1}$ in \eqref{eq:momentum_disc_weak} give
    \begin{equation}
        \dual{g^{n+1}, v_3}_{\Omega_p} = \frac{\rho_f}{\delta t}(\bm{u}^{n+1} - \bm{u}^{n}, \bm{v})_{\Omega_f} + \nu_f(\nabla\bm{u}^{n+1}, \nabla\bm{v})_{\Omega_f} - (\bm{f}^{n+1}, v)_{\Omega_f}. \label{eq:g_weak}
    \end{equation}
    %Note that by the Poincar\'{e}-Friedrich's inequality, there exists a constant $C_0>0$ such that
    %\begin{equation}
    %    \abs{(\bm{u}^{n+1} - \bm{u}^n, \bm{v})_{\Omega_f}} \leq \norm{\bm{u}^{n+1} - \bm{u}^n}_{\Omega_f}\norm{\bm{v}}_{\Omega_f} \leq C_0\norm{\bm{u}^{n+1} - \bm{u}^n}_{\Omega_f}\norm{\bm{v}}_{H^{1}(\Omega_f)}
    %\end{equation}
    %for any $v\in\bm{U}$. %Dividing both sides by $\norm{\bm{v}}_{H^{1}(\Omega_f)}$ and taking the supremum over all $v\in\bm{U}$, we see that $\bm{u}^{n+1} - \bm{u}^n \in H^{-1}(\Omega_f)$. Using this, we know that
    Now, taking the absolute value on both sides of \eqref{eq:g_weak}, then using triangle inequality, Cauchy-Schwarz, and Poincar\'{e}-Friedrich's inequality, we obtain
    \begin{equation*}
        \abs{\dual{g^{n+1}, v_3}_{\Omega_p}} \leq \left(\frac{\rho_f}{\delta t}\norm{\bm{u}^{n+1} - \bm{u}^{n}}_{\Omega_f} + \nu_f\norm{\nabla\bm{u}^{n+1}}_{\Omega_f} + \norm{\bm{f}^{n+1}}_{\Omega_f}\right)\norm{\bm{v}}_{1,\Omega_f}.
    \end{equation*}
    %where $C_{v}>0$ is the Poincar\'{e} constant.
    From \eqref{eq:back_euler}, we have
    \begin{equation*}
        \abs{\dual{g^{n+1}, v_3}_{\Omega_p})} \leq \left(\rho_f\norm{\dot{\bm{u}}^{n+1}}_{\Omega_f} + \nu_f\norm{\nabla\bm{u}^{n+1}}_{\Omega_f} + \norm{\bm{f}^{n+1}}_{\Omega_f}\right)\norm{\bm{v}}_{1,\Omega_f}.
    \end{equation*}
    Using the extension operator $E:H^{1/2}(\Omega_p)\to H^1(\Omega_f)$ defined in \cite{Gagliardo1957, Schaftingen2025}, we restrict the choice to $\bm{v}\in \mathcal{R}(E)$, where $\mathcal{R}(E)$ is the range of $E$, $\lambda\in H^{1/2}(\Omega_p)$, i.e., we choose $\bm{v}\in \bm{V}$ such that $v_3 = \lambda$ on $\Omega_p$. Since $E$ is linear and continuous, we have $\norm{\bm{v}}_{1,\Omega_f}\leq C_E\norm{\lambda}_{1/2,\Omega_p}$. With this, we obtain
    \begin{equation*}
        \frac{\abs{\dual{g^{n+1}, \lambda}_{\Omega_p}}}{\norm{\lambda}_{1/2,\Omega_p}} \leq C_E\left(\frac{\rho_f}{\delta t}\norm{\dot{\bm{u}}^{n+1}}_{\Omega_f} + \nu_f\norm{\nabla\bm{u}^{n+1}}_{\Omega_f} + \norm{\bm{f}^{n+1}}_{\Omega_f}\right)
    \end{equation*}
    and taking supremum over all $\lambda\in H^{1/2}(\Omega_p)$ results to
    \begin{equation}
        \norm{g^{n+1}}_{-1/2,\Omega_p} \leq C_E\left(\frac{\rho_f}{\delta t}\norm{\dot{\bm{u}}^{n+1}}_{\Omega_f} + \nu_f\norm{\nabla\bm{u}^{n+1}}_{\Omega_f} + \norm{\bm{f}^{n+1}}_{\Omega_f}\right). \label{eq:g_bound1}
    \end{equation}
    Squaring both sides of \eqref{eq:g_bound1}, then using the identity $(a + b + c)^2 \leq 3(a^2 + b^2 + c^2)$, and summing from $n=0$ to $N-1$, we have
    \begin{equation}
        \norm{g^{n+1}}_{-1/2,\Omega_p}^2 \leq 3C_E^2\left(\rho_f^2\norm{\dot{\bm{u}}^{n+1}}_{\Omega_f}^2 + \nu_f^2\norm{\nabla\bm{u}^{n+1}}_{\Omega_f}^2 + \norm{\bm{f}^{n+1}}_{\Omega_f}^2\right). \label{eq:g_bound2}
    \end{equation}
    Summing up $n=0, \dots, N-1$ and using \eqref{eq:energy_1} give
    \begin{equation}
        \sum_{n=0}^{N-1}\norm{g^{n+1}}_{-1/2,\Omega_p}^2 \leq C_1 \sum_{n=0}^{N-1}\norm{\bm{f}^{n+1}}^2_{\Omega_f}. \label{eq:g_bound4}
    \end{equation}
    where $C_1 = 3C_E^2\left(\frac{C^2\rho_f}{\nu_f} + 
    C^2 \delta t
    + 1\right)$. 
    %{\color{red} $C_1 = 3C_e^2\left(\frac{C^2\rho_f}{\nu_f} + C_v^2 
    %C^2 \delta t
    %+ 1\right)$? Please check.}
    Adding  \eqref{eq:g_bound4} and \eqref{eq:energy_1}, we obtain \eqref{eq:energy_est}.
    
    Now, if we set $\varphi = w^{n+1}$ in \eqref{eq:poisson_semi_disc_weak} and by the Cauchy-Schwarz inequality, Young's inequality, and Poincar\'{e}-Friedrich inequality , we obtain
    \begin{align}
        \norm{\nabla w^{n+1}}_{\Omega_p}^2 &= (z^{n+1}, w^{n+1})_{\Omega_p} \\
        &\leq \frac{1}{2\epsilon}\norm{z^{n+1}}_{\Omega_p}^2 + \frac{C_P^2\epsilon}{2}\norm{\nabla w^{n+1}}_{\Omega_p}^2
    \end{align}
    and choosing $\epsilon = 1/C_P^2$ gives
    \begin{equation}
        \norm{\nabla w^{n+1}}_{\Omega_p}^2\leq C_P^2 \norm{z^{n+1}}^2_{\Omega_p}.
    \end{equation}
\end{proof}

In the next section, we describe the fully discrete problem developed using conforming finite element method and present error estimates for the problem.

\section{Fully-discrete problem} \label{sec:fully-disc}
Let $\mathcal{T}_{h_f}$ and $\mathcal{T}_{h_p}$ be shape regular triangulations of $\overline{\Omega}_f$ and $\overline{\Omega}_p$, respectively, into tetrahedra where $h_f$ and $h_p$ are the corresponding mesh sizes, defined by
\begin{equation}
    h_f = \max_{E\in \mathcal{T}_{h_f}}\mathrm{diam}(E) \quad\text{ and }\quad h_p = \max_{F\in \mathcal{T}_{h_p}}\mathrm{diam}(F).
\end{equation}
We define the following conforming discrete finite element spaces:
\begin{equation*}
    \begin{aligned}
        &\bm{U}_h = \{\bm{v}_h\in \bm{U}: \bm{v}_h|_{E}\in \mathbb{P}_{k+1}(E)^3, \quad\forall E\in \mathcal{T}_{h_f}\},\\
        &Q_{h} = \{q_h\in Q: q_h|_{E}\in \mathbb{P}_{k}(E), \quad \forall E\in \mathcal{T}_{h_f}\}, \\
        &W_h = \{\varphi_h\in W: \varphi_h|_{F}\in \mathbb{P}_{k+1}(F), \quad \forall F\in \mathcal{T}_{h_p}\}, \\
        &G_{h} = \{\lambda_h\in G: \lambda_h|_{F}\in \mathbb{P}_{k}(F), \quad \forall F\in \mathcal{T}_{h_p}\},
    \end{aligned}
\end{equation*}
for a fixed $1\leq k\in \mathbb{Z}$ and $\mathbb{P}_{k}$ is the space of polynomials with degree $\leq k$. Note that the finite element spaces $\bm{U}_h$ and $Q_{h}$ satisfy the discrete inf-sup condition:
\begin{equation} \label{eq:disc_infsup_Stokes}
    \sup_{\bm{v}_h\in \bm{U}_h} \frac{(\nabla \cdot \bm{v}_h, q_h)_{\Omega_f}}{\norm{\bm{v}_h}_{1,\Omega_f}} \geq \beta_p \norm{q_h}_{\Omega_f}, \quad \forall q_h \in Q_{h}.
\end{equation}
for some $\beta_p>0$. 

The fully-discrete problem reads: find $\left(\bm{u}_h^{n+1}, w_h^{n+1}, z_h^{n+1}, p_h^{n+1}, g_h^{n+1}\right)\in \bm{U}_h\times W_h\times W_h\times Q_h\times G_h$ such that
\begin{align} \label{eq:momentum_full_disc_weak}
    &\begin{aligned}
    &\rho_f(\bm{u}_h^{n+1}, \bm{v}_h)_{\Omega_f} + \nu_f\delta t(\nabla \bm{u}_h^{n+1}, \nabla\bm{v}_h)_{\Omega_f} - \delta t(p^{n+1}_{h}, \nabla\cdot \bm{v}_h)_{\Omega_f} \\
    &\qquad + \delta t\dual{g_h^{n+1}, (\bm{v}_h)_3}_{\Omega_p} = \delta t(\bm{f}^{n+1}, \bm{v}_h)_{\Omega_f} + \rho_f (\bm{u}_h^n, \bm{v}_h)_{\Omega_f}
    \end{aligned}
     &&\forall \bm{v}_h\in \bm{U}_h,\\
    &(\nabla\cdot \bm{u}_h^{n+1}, q_h)_{\Omega_f} = 0 &&\forall q_h\in Q_{h},\label{eq:continuity_full_disc_weak} \\
    &\begin{aligned}
    &\frac{1}{\delta t}(w_h^{n+1}, \eta)_{\Omega_p} + \frac{\rho}{\delta t}(z_h^{n+1}, \eta)_{\Omega_p} +\delta t (\nabla z_h^{n+1}, \nabla \eta_h)_{\Omega_p} - \delta t\dual{g_h^{n+1}, \eta_h}_{\Omega_p} \\
    &\qquad = \frac{2}{\delta t}(w^n_h, \eta)_{\Omega_p} - \frac{1}{\delta t}(w_h^{n-1}, \eta)_{\Omega_p} + \frac{2\rho}{\delta t} (z_h^n, \eta)_{\Omega_p} - \frac{\rho}{\delta t} (z_h^{n-1}, \eta)_{\Omega_p} 
    \end{aligned}
    &&\forall \eta_h\in W_h, \label{eq:plate_full_disc_weak}\\
    & (z_h^{n+1}, \varphi_h)_{\Omega_p} = (\nabla w_h^{n+1}, \nabla \varphi_h)_{\Omega_p} &&\forall \varphi_h\in W_h, \label{eq:poisson_disc_weak} \\
    &\dual{(\bm{u}_h)_3^{n+1}, \lambda_h}_{\Omega_p}- \frac{1}{\delta t}\dual{w_h^{n+1}, \lambda_h}_{\Omega_p}  = -\frac{1}{\delta t}\dual{w_h^{n}, \lambda_h}_{\Omega_p} &&\forall \lambda_h \in G_h. \label{eq:interface_full_disc_weak}
\end{align}
For the well-posedness of the fully discrete problem \eqref{eq:momentum_full_disc_weak}-\eqref{eq:interface_full_disc_weak}, we consider the pressure $p_h$ in  $Q_h\subset Q$ instead in the discrete mean zero space $Q_{0,h} \subset Q_0$ 
%where
%\begin{equation*}
%    Q_{h} = \{q_h\in Q: q_h|_{E}\in \mathbb{P}_{k}(E), %\quad\forall E\in \mathcal{T}_{h_f}\},
%\end{equation*}
and set the Lagrange multiplier to be $g_h = p_h|_{\Omega_p}\in G$. Hence, we rewrite \eqref{eq:momentum_full_disc_weak}-\eqref{eq:interface_full_disc_weak} as the saddle point problem which reads: find $\left(\bm{u}_h^{n+1}, w_h^{n+1}, z_h^{n+1}, p_h^{n+1}, g_h^{n+1}\right)\in \bm{U}_h\times W_h\times W_h\times Q_h\times G_h$ such that
\begin{equation} \label{eq:saddle_point_prob_full_disc}
    \begin{aligned}
        a_{\delta t}\left((\bm{u}_h^{n+1}, w_h^{n+1}, z_h^{n+1}), (\bm{v}_h, \varphi_h, \eta_h)\right) + b_{\delta t}^h\left((\bm{v}_h, \eta_h), \left(\delta t p_h^{n+1}, \delta t g_h^{n+1}\right)\right) &= f^1_{\delta t}(\bm{v}_h, \varphi_h, \eta_h), \\
        b_{\delta t}^h\left(\left(\bm{u}_h^{n+1}, \frac{1}{\delta t}w_h^{n+1}\right),(q_h, \lambda_h)\right) &= f_{\delta t}^2 (q_h, \lambda_h),
    \end{aligned}
\end{equation}
where $b_{\delta t}^h: \left(\bm{U}_h\times W_h\right) \rightarrow (Q_h\times G_h)$ is defined as
\begin{equation*}
    b_{\delta t}^h\left((\bm{v}_h, \eta_h), (q_h, \lambda_h)\right) \coloneqq -(\nabla\cdot \bm{v}_h, q_h)_{\Omega_f} + \dual{(\bm{v}_h)_3, \lambda_h}_{\Omega_p} - \dual{\eta_h, \lambda_h}_{\Omega_p}.
\end{equation*}
Note that we reformulate the system by allowing the pressure variable to have a non-zero mean. This modification enables a more straightforward establishment of the inf-sup condition. We emphasize that this reformulation is equivalent to the original problem.

The coercivity of $a_{\delta t}$ over the discrete spaces can easily be adopted from \eqref{eq:semi_disc_coercive}. Therefore, 
the well-posedness of \eqref{eq:momentum_full_disc_weak}-\eqref{eq:interface_full_disc_weak} follows from the inf-sup condition for $b_{\delta t}^h$ established in the following theorem.

\begin{theorem}
    There exists a positive constant $\beta_h$ such that 
    \begin{equation}\label{eq:full_disc_infsup}
        \sup_{(\bm{v}_h, \eta_h)\in(\bm{U}_h, W_h)} \frac{b_{\delta t}^h((\bm{v}_h, \eta_h), (q_h, g_h))}{\left(\norm{\bm{v}_h}_{1, \Omega_f}^2 + \norm{\eta_h}_{1, \Omega_p}^2\right)^{1/2}} \geq \beta_h \left(\norm{q_h}_{\Omega_f}^2 + \norm{g_h}_{-1/2, \overline{\Omega}_p}\right)^{1/2},
    \end{equation}
    for any $(q_h, g_h)\in (Q_h, G_h)$.
\end{theorem}

\begin{proof}
Let $(q_h, g_h)\in (Q_h, G_h)$ be given. For ${k} \in H^{1/2}_{\overline{\Omega}_p}$, we may find $\bm{v}_k \in \bm{U}$ such that  \\
\begin{align}\label{vkConds}
\begin{split}
   (\bm{v}_k)_3 \big|_{\overline{\Omega}_p} &= {k}\\
\|\bm{v}_k\|_{1, \Omega_f} &\leq C_1 \| {k}\|_{1/2,\overline{\Omega}_p}.
\end{split}
\end{align}
By \eqref{vkConds} and the definition of the dual norm,
\begin{align}\label{sh_bound}
\begin{split}
    \|g_h\|_{-1/2,\overline{\Omega}_p} &= \underset{k \in H^{1/2} (\overline{\Omega}_p) }{\sup}     \frac{\langle g_h,k\rangle_{\overline{\Omega}_p}} {\|k\|_{1/2,\overline{\Omega}_p}} \\
    &\leq \underset{k \in H^{1/2} (\overline{\Omega}_p)  }{\sup} \frac{C_1 \langle g_h,(\bm{v}_k)_3\rangle_{\overline{\Omega}_p}} {\|\bm{v}_k\|_{1, \Omega_f}} \\
    &\leq  \underset{\bm{v_k} \in \bm{U}}{\sup} \frac{C_1 \langle g_h,(\bm{v}_k)_3\rangle_{\overline{\Omega}_p}} {\|\bm{v}_k\|_{1, \Omega_f}} .
    \end{split}
\end{align}
There exists the projection operator $I_f^h : \bm{U} \rightarrow \bm{U}_h$ satisfying the following two conditions for $\bm{v_k} \in \bm{U}$ \cite{LionsBooks1972}
\begin{align}\label{ProjectorConds}
\langle g_h, (\bm{v}_k)_3 \rangle_{\overline{\Omega}_p} &= \langle g_h, (I^h_f\bm{v}_k)_3\rangle_{\overline{\Omega}_p}, \\ %
\|I^h_f \bm{v}_k\|_{1, \Omega_f} &\leq C_2 \| \bm{v}_k\|_{1, \Omega_f}.
\end{align}
With these properties, we have, for $\bm{v}_k \in \bm{U}$ and $g_h \in G_h$,
\begin{align*}
 \frac{C_1 \langle g_h,(\bm{v}_k)_3\rangle_{\overline{\Omega}_p}}{\|\bm{v}_k\|_{1, \Omega_f}}  &=   \frac{C_1 \langle g_h,(I^h_f \bm{v}_k)_3\rangle_{\overline{\Omega}_p}}{\|\bm{v}_k\|_{1, \Omega_f}}  \\
 &\leq  \frac{C_1 C_2 \langle g_h, (I^h_f \bm{v}_k)_3\rangle_{\overline{\Omega}_p}}{\|I^h_f\bm{v}_k\|_{1, \Omega_f}} \\
 &\leq \underset{\bm{v}_{k}^h \in \bm{U}_h}{\sup}\frac{C_1 C_2 \langle g_h, (\bm{v}_{k}^h)_3\rangle_{\overline{\Omega}_p}}{\|\bm{v}_{k}^h\|_{1, \Omega_f}}.
\end{align*}
Taking the supremum over $\bm{v}_k \in \bm{U}$ of this inequality yields:
\begin{align*}
\underset{\bm{v_k} \in \bm{U}}{\sup}  \frac{C_1 \langle g_h,(\bm{v}_k)_3\rangle_{\overline{\Omega}_p}}{\|\bm{v}_k\|_{1, \Omega_f}} \leq \underset{\bm{v}_k \in \bm{U}}{\sup} \underset{\bm{v}_k^h \in \bm{U}_h}{\sup}
\frac{C_1 C_2 \langle g_h,(\bm{v}_k^h)_3\rangle_{\overline{\Omega}_p}}{\|\bm{v}_k^h\|_{1, \Omega_f}} = \underset{\bm{v}_k^h \in \bm{U}_h}{\sup}\frac{C_1 C_2 \langle g_h,(\bm{v}_k^h)_3\rangle_{\overline{\Omega}_p}}{\|\bm{v}_k^h\|_{1, \Omega_f}}.
\end{align*}
Thus, combining \eqref{sh_bound} with the above, we have the desired inf-sup condition:
\begin{align*}
\|g_h\|_{-1/2,\overline{\Omega}_p} \leq  \underset{\bm{v_k} \in  \bm{U}}{\sup} \frac{C_1 \langle g_h,(\bm{v}_k)_3\rangle_{\overline{\Omega}_p}}{\|\bm{v}_k\|_{1}} \leq C_1 C_2\underset{\bm{\bm{v}_k}^h \in \bm{U}_h}{\sup}\frac{ \langle g_h,(\bm{v}_k^h)_3\rangle_{\overline{\Omega}_p}}{\|\bm{v}_k^h\|_{1}}.
\end{align*}
Rewritten,
\begin{align}\label{infsup_LM_Stokes}
\underset{\bm{v}_h \in \bm{U}_h}{\sup}\frac{ \langle g_h,(\bm{v}_h)_3\rangle_{\overline{\Omega}_p}}{\|\bm{v}_h\|_{1}} \geq \beta_f \|g_h\|_{-1/2,\overline{\Omega}_p} \quad \forall g_h \in G_h.
\end{align}

For $k \in H^1(\Omega_p) $  there exists the $L^2$-projector $P_s : H^1(\Omega_p) \rightarrow W_h$ satisfying
\begin{align}\label{ProjectorConds_phi}
( g_h, k )_{\Omega_p} &= ( g_h, P_s k )_{\Omega_p}, \\ 
\|P_s k\|_1 &\leq C_3 \| k\|_{1,\Omega_p} .
\end{align}
Since $g_h \in L^2(\Omega_p) \subset H^{-1/2}(\overline{\Omega}_p)$, 
\begin{equation} \label{equi_products}
\langle g_h,k\rangle_{\Omega_p}  = (g_h,k )_{\Omega_p}. % \quad \mbox{(dual product = $L^2$ inner product $= \int g^h \, k \; d \Omega_p$)}
\end{equation}
Thus, from \cite{LionsBooks1972}, we have
\begin{align}\label{sh_bound_phi}
\begin{split}
    \|g_h\|_{-1/2,\overline{\Omega}_p} &\le  C\underset{k \in H^1_0(\Omega_p) } {\sup} \frac{\langle g_h,k\rangle_{\Omega_p} } {\|k\|_{1, \Omega_p} } \\%  \quad  [J.-L. Lions, E. Magenes]
    &=  C \underset{k \in W } {\sup} \frac{( g_h,k)_{\Omega_p} } {\|k\|_{1, \Omega_p} }  \\
    &\leq C C_3 \underset{k \in W }{\sup} \frac{ ( g_h, P_s {k} ) }{ \| P_s k\|_{1, \Omega_p}} \\
    &\leq  C C_3 \underset{\varphi_k^h \in W_h}{\sup} \frac{\langle g_h,{\varphi_k^h}\rangle_{\Omega_p}}{\|{\varphi_k^h}\|_{1, \Omega_p}}.
    \end{split}
\end{align}
Rewritten, 
\begin{align}\label{infsup_LM_Structure}
 \underset{\varphi_h \in W_h}{\sup}\frac{ \langle g_h,\varphi_h\rangle_{{\Omega}_p}}{\|\varphi^h\|_{1, \Omega_p}} \geq  \beta_s\|g_h\|_{-1/2,\overline{\Omega}_p} \quad \forall g_h \in G_h.
\end{align}
By \eqref{infsup_LM_Stokes}, there exists $\hat{\bm{v}}_h \in \bm{U}_h$ such that 
\begin{align}\label{V2_Prop}
\|\hat{\bm{v}}_h\|_{1, \Omega_p} = 1 \quad \text{ and } \quad \langle (\hat{\bm{v}}_h)_3,g_h\rangle_{\overline{\Omega}_p} \geq \beta_f \|g_h\|_{-1/2,\overline{\Omega}_p} .
\end{align}
By the Cauchy-Schwarz inequality 
 \begin{align}\label{V21_Prop}
(\nabla \cdot \hat{\bm{v}}_h,q_h) \le   \|\nabla \cdot  \hat{\bm{v}}_h\|_{\Omega_f}  \|q^h\|_{\Omega_f} \   \le \sqrt{3} \|\nabla  \hat{\bm{v}}\|_{\Omega_f}   \|q_h\|_{\Omega_f} \le \sqrt{3}\|q_h\|_{\Omega_f}.
\end{align}

\noindent Likewise, by \eqref{eq:disc_infsup_Stokes}, there exists $\tilde{\bm{v}}_h \in \bm{U}^h_0$ with
\begin{align}\label{V1_Prop}
\|\tilde{\bm{v}}^h\|_{1, \Omega_f} = 1 \quad \text{ and } \quad (\nabla \cdot \tilde{\bm{v}}_h,q_h) \geq \beta^* \|q^h\|_{\Omega_f},
\end{align}
where $\bm{U}^h_0 \subset \bm{U}_h \subset \bm{U}$.
\noindent By \eqref{infsup_LM_Structure}, there exists $\hat{\varphi}_h \in W_h$ such that 
\begin{align}\label{Varphi1_Prop}
\|\hat{\varphi}_h\|_{1, \Omega_p} = 1 \quad \text{ and } \quad -\langle \hat{\varphi}_h,g_h\rangle_{\overline{\Omega}_p} \geq \beta_s \|g_h\|_{-1/2,\overline{\Omega}_p} .
\end{align}

\noindent\textit{Remark: If $\bm{\zeta}_h$ satisfies $\dfrac{\langle g_h,\bm{\zeta}_h \rangle_{\overline{\Omega}_p}}{\|\bm{\zeta}_h\|_{1, \Omega_p}} \geq \beta_s \|g_h\|_{-1/2,\overline{\Omega}_p}$, take $\bm{\varphi}_h := -\dfrac{\bm{\zeta}_h}{\|\bm{\zeta}^h\|_{1, \Omega_p}}.$ }

\noindent Let $\bm{v}_h := \hat{\bm{v}}_h + (1 + \frac{\sqrt{3} }{\beta^*}) \tilde{\bm{v}}_h \in \bm{U}_h$ and $\varphi_h = \hat{\varphi}_h\in W_h$. Then applying \eqref{V2_Prop}-\eqref{Varphi1_Prop} in the following yields
\begin{align*}
&-(\nabla \cdot \bm{v}_h, q_h) + \langle (\bm{v}_h)_3,g_h\rangle_{\overline{\Omega}_p} - \langle {\varphi}_h,g_h\rangle_{\overline{\Omega}_p} \\
&= -(\nabla \cdot \hat{\bm{v}}_h, q_h) + (1 + \frac{\sqrt{3}}{\beta^*}) (\nabla \cdot \tilde{\bm{v}}_h, q_h) + \langle (\hat{\bm{v}}_h)_3,g_h\rangle_{\overline{\Omega}_p} + (1 + \frac{\sqrt{3}}{\beta^*}) \langle  (\tilde{\bm{v}}_h)_3,g_h\rangle_{\overline{\Omega}_p} \\
& \hspace{.5in} 
-  \langle  \hat{\varphi}_h ,g_h\rangle_{\overline{\Omega}_p}  
\\
&\geq - \sqrt{3} \| q_h\|_{\Omega_f}+ \beta^* (1 + \frac{\sqrt{3}}{\beta^*}) \|q_h\|_{\Omega_f}  + {\beta_f}\|g_h\|_{-1/2,\overline{\Omega}_p} + 0 
+ \beta_s \|g_h\|_{-1/2,\overline{\Omega}_p}   \\
&= \beta^* \|q_h\|_{\Omega_f}   +  \left( {\beta_f}  + \beta_s \right) \|g_h\|_{-1/2,\overline{\Omega}_p} .
\end{align*}
\noindent Note that 
\begin{align*}
\|\bm{v}_h\|_{1, \Omega_f} + \|\varphi_h\|_{1, \Omega_p} \leq \|\hat{\bm{v}}_h\|_{1, \Omega_f} + (1 + \frac{\sqrt{3}}{\beta^*}) \| \tilde{\bm{v}}_h\|_{1, \Omega_f} +  \| \hat{\varphi}_h\|_{1, \Omega_p} = 3+ \frac{\sqrt{3}}{\beta^*},
\end{align*}
which implies that $$ \frac{\|\bm{v}_h\|_{1, \Omega_f} + \|{\varphi}_h\|_{1, \Omega_p}}{3 + \frac{\sqrt{3}}{\beta^*}} \leq 1.$$
Thus, 
\begin{align*}
&-(\nabla \cdot \bm{v}_h, q_h) + \langle (\bm{v}_h)_3,g_h\rangle_{\overline{\Omega}_p} - \langle {\varphi}_h,g_h\rangle_{\overline{\Omega}_p} 
\\
&\geq \min\{\beta^*,\beta_f+\beta_s\} \left( \|q_h\|_{\Omega_f} + \|g_h\|_{-1/2,\overline{\Omega}_p} \right)\\
&\geq   \frac{\min\{\beta^*, \beta_f+\beta_s\}} {3 + \frac{\sqrt{3}}{\beta^*}} \left(\|\bm{v}_h\|_{1, \Omega_p} + \|{\varphi}_h\|_{1, \Omega_p} \right) \left( \|q_h\|_{\Omega_f} + \|g_h\|_{-1/2,\overline{\Omega}_p}  \right)\\
&\geq  \frac{\min\{\beta^*, \beta_f+\beta_s\}} {3 + \frac{\sqrt{3}}{\beta^*}} \left(\|\bm{v}_h\|_{1, \Omega_f}^2 + \|{\varphi}_h\|_{1, \Omega_p}^2 \right)^{1/2} \left( \|q_h\|_{\Omega_f}^2 + \|g_h\|_{-1/2,\overline{\Omega}_p}^2 \right)^{1/2}\\
&= \frac{\min\{\beta^*, \beta_f+\beta_s\}} {3+ \frac{\sqrt{3}}{\beta^*}} \|(\bm{v}_h,{\varphi}_h)\|_{\bm{U}_h\times W_h}\ \|(q_h,g_h)\|_{Q^h \times G^h },
\end{align*}
and we are done since $(q_h, g_h)\in (Q_h, G_h)$ is arbitrary.

\end{proof}

%\begin{rem}
%    We note that the inf-sup condition for the saddle-point problem \eqref{eq:saddle_point_prob_full_disc} is still open.
%\end{rem}

For the error estimates of \eqref{eq:momentum_full_disc_weak}-\eqref{eq:interface_full_disc_weak}, we assume that there exists a unique solution $(\bm{u}, p_0, w, z, g, s)\in \bm{U}\times Q_0\times W\times W\times G\times \mathbb{R}$ to \eqref{eq:momentum_weak}-\eqref{eq:interface_weak} which is sufficiently smooth. We define the truncation errors $\tau_u^{n+1}$, $\tau_w^{n+1}$, $\tau_z^{n+1}$ in time by
\begin{equation} \label{eq:time_trunc_err}
    \begin{aligned}
        \tau_u^{n+1} \coloneqq \partial_{t}\bm{u}(t^{n+1}) -& \frac{\bm{u}(t^{n+1})-\bm{u}^n}{\delta t}, \quad \tau_w^{n+1} \coloneqq \partial_{tt}w(t^{n+1}) - \frac{w^{n+1}-2w^n+w^{n-1}}{(\delta t)^2},\\
        &\tau_z^{n+1} \coloneqq \partial_{tt}z(t^{n+1}) - \frac{z^{n+1}-2z^n+z^{n-1}}{(\delta t)^2}.
    \end{aligned}
\end{equation}
Then, by using Taylor series approximation, we have the following truncation error estimates
\begin{equation} \label{eq:taylor_bound}
    \begin{aligned}
        \norm{\tau_u^{n+1}}_{\Omega_f} \leq C_u\delta t&\norm{\partial_{tt}\bm{u}(\theta_u^{n+1})}_{\Omega_f},\quad \norm{\tau_w^{n+1}}_{\Omega_p} \leq C_w\delta t\norm{\partial_{ttt}w(\theta_w)^{n+1}}_{\Omega_p},\\
        &\text{ }\norm{\tau_z^{n+1}}_{\Omega_p} \leq C_z\delta t\norm{\partial_{ttt}z(\theta_z^{n+1})}_{\Omega_p}.
    \end{aligned}
\end{equation}
for some $t^n\leq\theta_u^{n+1}, \theta_w^{n+1}, \theta_z^{n+1}\leq t^{n+1}$. We also define the Stokes projection \cite{Layton2008} operator $\Pi_u: \bm{U} \to \bm{U}_h$ which satisfies
\begin{equation} \label{eq:stokes-proj}
    \begin{aligned}
        &\nu_f\left(\left(\nabla\bm{u} - \nabla \Pi_u\bm{u}\right), \nabla\bm{v}_h\right)_{\Omega_f} - \left(p_{0,h}, \nabla\cdot\bm{v}_h\right)_{\Omega_f} = 0 && \forall \bm{v}_h\in\bm{U}_h,\\
        &\left(\left(\nabla\cdot\bm{u} - \nabla\cdot \Pi_u\bm{u}\right), q_h\right)_{\Omega_f} = 0 &&\forall q_h\in Q_{h},
    \end{aligned}
\end{equation}
with the following approximation error
\begin{equation} \label{eq:stokes_proj_bound}
    \norm{\bm{u} - \Pi_u \bm{u}}_{1,\Omega_f}\leq Ch_f^{k}\norm{\bm{u}}_{k+1, \Omega_f}.
\end{equation}
Analogously, let $\Pi_g: G \to G_h$, $\Pi_p: Q\to Q_{h}$, and $\Pi_z: W\to W_h$ be the $L^2$ projection operators, and $\Pi_w: W\to W_h$ be the $H^1$ projection operator \cite{Ern_Guermond2004} onto the discrete FE spaces which satisfy
\begin{equation} \label{eq:l2-proj}
    \begin{aligned}
        &\left(\left(g-\Pi_g g\right), \lambda_h\right)_{\Omega_p} = 0 &&\forall \lambda_h \in G_h, \\
        &\left(\left(p-\Pi_p p\right), q_h\right)_{\Omega_f} = 0 &&\forall q_h \in Q_{h}, \\
        &\left(\left(z-\Pi_z z\right), \eta_h\right)_{\Omega_p} = 0 &&\forall \eta_h \in W_h, \\
        &\left(\left(w-\Pi_w w\right), \varphi_h\right)_{\Omega_p} + \rho\left(\nabla\left(w-\Pi_w w\right), \nabla\varphi_h\right)_{\Omega_p} = 0 &&\forall \varphi_h \in W_{h},
    \end{aligned}
\end{equation}
with approximation properties
\begin{equation} \label{eq:l2_proj_bound}
    \begin{aligned}
        \norm{p - \Pi_p p}_{\Omega_f} \leq Ch_f^{k}\norm{p}_{k, \Omega_f}, &\quad \norm{g - \Pi_g g}_{\Omega_p} \leq Ch_p^{k}\norm{g}_{k, \Omega_p}, \\
        \norm{z - \Pi_z z}_{\Omega_p} \leq Ch_f^{k+1}\norm{z}_{k+1, \Omega_f}, &\quad \norm{z - \Pi_z z}_{1,\Omega_p} \leq Ch_p^{k}\norm{z}_{k+1, \Omega_p}, \\
        \norm{w - \Pi_w w}_{1,\Omega_p} &\leq Ch_p^{k}\norm{w}_{k+1, \Omega_p}.
    \end{aligned}
\end{equation}

\begin{theorem}
    Suppose that the exact solution 
    $(\bm u,p,w,z,g)$ of \eqref{eq:momentum_weak}--\eqref{eq:interface_weak} is sufficiently smooth. 
    Let $(\bm{u}_h^{n+1},p_h^{n+1},w_h^{n+1},z_h^{n+1},g_h^{n+1})$ be the fully discrete finite element solution obtained using $(\mathbb{P}_{k+1})^3$--$\mathbb{P}_k$ elements for the fluid velocity and pressure, $\mathbb{P}_{k+1}$ elements for $w$ and $z$, and $\mathbb{P}_k$ elements for the Lagrange multiplier $g$. Then, there exists a positive constant $C$, independent of the mesh sizes $h_f,h_p$ and the time step $\delta t$, such that the following a priori error estimate holds:
    \begin{equation}
        \begin{aligned}
            &\norm{\partial_t w(t^{N}) - \dot{w}_h^N}_{1, \Omega_p} + \norm{z(t^{N}) - z_h^N}_{\Omega_p} + \norm{\bm{u}(t^N) - \bm{u}_h^N}_{\Omega_f} \\
            &+ \sum_{n=0}^{N-1}\left( \sqrt{\delta t}\norm{\bm{u}(t^{n+1}) - \bm{u}_h^{n+1}}_{1,\Omega_f} + \delta t\norm{p(t^{n+1}) - p_h^{n+1}}_{\Omega_f} + \delta t \norm{g(t^{n+1}) - g^{n+1}}_{-1/2, \Omega_p}\right) \\
            &\leq C(\delta t + h_f^k + h_p^k).
        \end{aligned}
    \end{equation}
\end{theorem}

\begin{proof}

We define the error function for the variable $\Phi\in\{\bm{u}, p, w, z, g, s\}$  at $t=t^{n+1}$
\begin{equation*}
    e_\Phi^{n+1} \coloneqq\Phi (t^{n+1})- \Phi_h^{n+1},
\end{equation*}
%where $\Phi^{n+1}$ is the time discrete approximation of $\Phi(t^{n+1})$, 
which can be decomposed into the discretization error (in space) $\chi_\Phi^{n+1}$ and the approximation error $\xi_\Phi^{n+1}$:
\begin{equation}\label{eq:trunc-approx}
    e_{\Phi}^{n+1} = \underbrace{(\Phi(t^{n+1}) - \Pi_\Phi \Phi(t^{n+1})}_{\coloneqq \xi_{\Phi}^{n+1}} + \underbrace{(\Pi_\Phi \Phi(t^{n+1}) - \Phi_{h}^{n+1}}_{\coloneqq \chi_\Phi^{n+1}}).
\end{equation}
In the remainder of this section, for any variable $\Phi\in\{\bm{u}, p, w, z, g\}$, we use $\Phi^{n+1}$ to denote  $\Phi(t^{n+1})$ for notational convenience; thus, $\Phi^{n+1}$ is not a variable in the semi-discrete system \eqref{eq:momentum_disc_weak}-\eqref{eq:interface_disc_weak} in the following error estimation. 
Define $\dot{\Phi}^{n+1}$, $\dot{\Phi}^{n+1}_h $similarly to \eqref{eq:back_euler} for 
$\Phi^{n+1}$ and $\Phi^{n+1}_h$, respectively, to obtain 
\begin{equation} \label{eq:backward_diff_error}
    \begin{aligned}
        \dot{\Phi}^{n+1} - \dot{\Phi}^{n+1}_h &= \frac{1}{\delta t}(\Phi^{n+1} - \Phi^{n}) - \frac{1}{\delta t}(\Phi^{n+1}_h - \Phi^{n}_h) \\
        &= \frac{1}{\delta t}(\Phi^{n+1}- \Phi^{n+1}_h) - \frac{1}{\delta t}(\Phi^{n}- \Phi^{n}_h)  \\
        &= \frac{1}{\delta t}(e_\Phi^{n+1} - e_\Phi^{n})\\
        &= \dot{e}_\Phi^{n+1}.
    \end{aligned}
\end{equation}
Subtracting \eqref{eq:momentum_full_disc_weak}-\eqref{eq:interface_full_disc_weak} from \eqref{eq:momentum_weak}-\eqref{eq:interface_weak} and restricting the test functions to come from the discrete FE spaces, we obtain
\begin{align} \label{eq:momentum_error}
    &\begin{aligned}
    &\rho_f\left(\partial_{t}\bm{u}(t^{n+1}) - \frac{\bm{u}^{n+1}_h-\bm{u}^{n}_h}{\delta t}, \bm{v}_h\right)_{\Omega_f} + \nu_f(\nabla \bm{e}_u^{n+1}, \nabla\bm{v}_h)_{\Omega_f} - (e_p^{n+1}, \nabla\cdot \bm{v}_h)_{\Omega_f} \\
    &\qquad + \dual{e_g^{n+1}, (\bm{v}_h)_3}_{\Omega_p} = 0 
    \end{aligned}
     &&\forall \bm{v}_h\in \bm{U}_h,\\
    &(\nabla\cdot \bm{e}_u^{n+1}, q_h)_{\Omega_f} = 0 &&\forall q_h\in Q_{h},\label{eq:continuity_error} \\
    &\begin{aligned}
    &\left(\partial_{tt}w(t^{n+1}) - \frac{w^{n+1}_h - 2w^{n}_h + w^{n-1}_h}{(\delta t)^2}, \eta_h\right)_{\Omega_p} \\
    &\qquad + \rho\left(\partial_{tt}z(t^{n+1}) - \frac{z^{n+1}_h - 2z^{n}_h + z^{n-1}_h}{(\delta t)^2}, \eta_h\right)_{\Omega_p} + (\nabla e_z^{n+1}, \nabla \eta_h)_{\Omega_p} \\
    &\qquad - \dual{e_g^{n+1}, \eta_h}_{\Omega_p} - (e_s^{n+1}, \eta_h)_{\Omega_p} = 0
    \end{aligned}
    &&\forall \eta_h\in W_h, \label{eq:plate_error}\\
    & (\nabla e_w^{n+1}, \nabla \varphi_h)_{\Omega_p} - (e_z^{n+1}, \varphi_h)_{\Omega_p} = 0 &&\forall \varphi_h\in W_h, \label{eq:poisson_error} \\
    & \dual{(\bm{e}_u^{n+1})_3, \lambda_h}_{\Omega_p} - \dual{\partial_t w(t^{n+1}) - \frac{w^{n+1}_h-w^n_h}{\delta t}, \lambda_h}_{\Omega_p} = 0 &&\forall \lambda_h \in G_h. \label{eq:interface_error}
\end{align}
Following \cite{Wang2025}, we can rewrite the errors involving the time derivatives as
\begin{equation} \label{eq:ddt_error}
    \begin{aligned}
        \partial_{t}\bm{u}(t^{n+1}) - \frac{\bm{u}^{n+1}_h-\bm{u}^{n}_h}{\delta t} &= \left(\partial_{t}\bm{u}(t^{n+1}) - \frac{\bm{u}^{n+1}-\bm{u}^{n}}{\delta t} \right) + \left(\frac{\bm{u}^{n+1}-\bm{u}^{n}}{\delta t} - \frac{\bm{u}^{n+1}_h-\bm{u}^{n}_h}{\delta t} \right)\\
        &= \tau_u^{n+1} + \frac{e_u^{n+1} - e_u^{n}}{\delta t},\\
        \partial_{tt}w(t^{n+1}) - \frac{w^{n+1}_h - 2w^{n}_h + w^{n-1}}{(\delta t)^2} &= \left(\partial_{tt}w(t^{n+1}) - \frac{w^{n+1} - 2w^{n} + w^{n-1}}{(\delta t)^2}\right)\\
        &\quad+ \left(\frac{w^{n+1} - 2w^{n} + w^{n-1}}{(\delta t)^2} - \frac{w^{n+1}_h - 2w^{n}_h + w^{n-1}_h}{(\delta t)^2}\right)\\
        &= \tau_w^{n+1} + \frac{\dot{e}_w^{n+1}-\dot{e}_w^n}{\delta t}.
    \end{aligned}
\end{equation}
This means we can rewrite \eqref{eq:momentum_error}-\eqref{eq:interface_error} to
\begin{align} \label{eq:momentum_error_2}
    &\begin{aligned}
    &\rho_f\left(\delta t\tau_u^{n+1} + \bm{e}_u^{n+1} - \bm{e}_u^{n}, \bm{v}_h\right)_{\Omega_f} + \nu_f\delta t(\nabla \bm{e}_u^{n+1}, \nabla\bm{v}_h)_{\Omega_f} - \delta t(e_p^{n+1}, \nabla\cdot \bm{v}_h)_{\Omega_f} \\
    &\qquad + \delta t\dual{e_g^{n+1}, (\bm{v}_h)_3}_{\Omega_p} = 0 
    \end{aligned}
     &&\forall \bm{v}_h\in \bm{U}_h,\\
    &(\nabla\cdot \bm{e}_u^{n+1}, q_h)_{\Omega_f} = 0 &&\forall q_h\in Q_{h},\label{eq:continuity_error_2} \\
    &\begin{aligned}
    &\left(\delta t\tau_w^{n+1} + \dot{e}_w^{n+1} - \dot{e}_w^{n}, \eta_h\right)_{\Omega_p} + \rho\left(\delta t\tau_z^{n+1} + \dot{e}_z^{n+1} - \dot{e}_z^{n}, \eta_h\right)_{\Omega_p} \\
    &\qquad + \delta t(\nabla e_z^{n+1}, \nabla \eta_h)_{\Omega_p} - \delta t\dual{e_g^{n+1}, \eta_h}_{\Omega_p} = 0
    \end{aligned}
    &&\forall \eta_h\in W_h, \label{eq:plate_error_2}\\
    & (\nabla e_w^{n+1}, \nabla \varphi_h)_{\Omega_p} - (e_z^{n+1}, \varphi_h)_{\Omega_p} = 0 &&\forall \varphi_h\in W_h, \label{eq:poisson_error_2} \\
    & \dual{(\bm{e}_u^{n+1})_3, \lambda_h}_{\Omega_p} - \dual{\tau_w^{n+1} + \dot{e}_w^{n+1}, \lambda_h}_{\Omega_p} = 0 &&\forall \lambda_h \in G_h. \label{eq:interface_error_2}
\end{align}
Note that by setting $\varphi_h = \eta_h$ in \eqref{eq:poisson_error}, then subtracting the equation with time $t^{n+1}$ to time $t^n$, and using \eqref{eq:backward_diff_error}, we can rewrite $\eqref{eq:plate_error}$ to
\begin{equation*}
    \begin{aligned}
        &(\dot{e}_w^{n+1}-\dot{e}_w^{n}, \eta_h)_{\Omega_p} + \rho(\nabla\dot{e}_w^{n+1}-\nabla\dot{e}_w^{n}, \nabla\eta_h)_{\Omega_p} +\delta t (\nabla e_z^{n+1}, \nabla \eta_h)_{\Omega_p} - \delta t\dual{e_g^{n+1}, \eta_h}_{\Omega_p} \\
        &\quad = -\delta t (\tau_w^{n+1}, \eta_h)_{\Omega_p} - \rho\delta t (\tau_z^{n+1}, \eta_h)_{\Omega_p},
    \end{aligned} \quad \forall \eta_h \in W_h.
\end{equation*}
Keeping the truncation error $\chi_\Phi$ terms defined in \eqref{eq:trunc-approx} on the left hand side, and transposing the approximation error $\xi_\Phi$ and the time discretization error $\tau_\Phi$ to the right hand side, we obtain
\begin{align} \label{eq:momentum_error_decomp}
    &\begin{aligned}
    &\rho_f(\bm{\chi}_u^{n+1}-\bm{\chi}_u^{n}, \bm{v}_h)_{\Omega_f} + \nu_f\delta t(\nabla \bm{\chi}_u^{n+1}, \nabla\bm{v}_h)_{\Omega_f} - \delta t(\chi_p^{n+1}, \nabla\cdot \bm{v}_h)_{\Omega_f} \\
    &\quad + \delta t\dual{\chi_g^{n+1}, (\bm{v}_h)_3}_{\Omega_p} = -\rho_f\delta t (\tau_u^{n+1}, \bm{v}_h)_{\Omega_f} - \rho_f(\bm{\xi}_u^{n+1}-\bm{\xi}_u^{n}, \bm{v}_h)_{\Omega_f} \\
    &\quad - \nu_f\delta t(\nabla \bm{\xi}_u^{n+1}, \nabla\bm{v}_h)_{\Omega_f} + \delta t(\xi_p^{n+1}, \nabla\cdot \bm{v}_h)_{\Omega_f} - \delta t\dual{\xi_g^{n+1}, (\bm{v}_h)_3}_{\Omega_p}
    \end{aligned}
     &&\forall \bm{v}_h\in \bm{U}_h,\\
    &(\nabla\cdot \bm{\chi}_u^{n+1}, q_h)_{\Omega_f} = -(\nabla\cdot \bm{\xi}_u^{n+1}, q_h)_{\Omega_f} &&\forall q_h\in Q_{h},\label{eq:continuity_error_decomp} \\
    &\begin{aligned}
    &(\dot{\chi}_w^{n+1}-\dot{\chi}_w^{n}, \eta_h)_{\Omega_p} + \rho(\nabla\dot{\chi}_w^{n+1}-\nabla\dot{\chi}_w^{n}, \nabla\eta_h)_{\Omega_p} +\delta t (\nabla \chi_z^{n+1}, \nabla \eta_h)_{\Omega_p} \\
    &\quad - \delta t\dual{\chi_g^{n+1}, \eta_h}_{\Omega_p} = -\delta t (\tau_w^{n+1}, \eta_h)_{\Omega_p} - \rho\delta t (\tau_z^{n+1}, \eta_h)_{\Omega_p} \\
    &\quad - (\dot{\xi}_w^{n+1}-\dot{\xi}_w^{n}, \eta_h)_{\Omega_p} - \rho(\nabla\dot{\xi}_w^{n+1}-\nabla\dot{\xi}_w^{n}, \nabla\eta_h)_{\Omega_p} - \delta t (\nabla \xi_z^{n+1}, \nabla \eta_h)_{\Omega_p} \\
    &\quad + \delta t\dual{\xi_g^{n+1}, \eta_h}_{\Omega_p}
    \end{aligned}
    &&\forall \eta_h\in W_h, \label{eq:plate_error_decomp}\\
    & (\nabla \chi_w^{n+1}, \nabla \varphi_h)_{\Omega_p} - (\chi_z^{n+1}, \varphi_h)_{\Omega_p} = -(\nabla \xi_w^{n+1}, \nabla \varphi_h)_{\Omega_p} + (\xi_z^{n+1}, \varphi_h)_{\Omega_p} &&\forall \varphi_h\in W_h, \label{eq:poisson_error_decomp} \\
    &\begin{aligned}
    & \dual{(\bm{\chi}_u^{n+1})_3, \lambda_h}_{\Omega_p} - \dual{\dot{\chi}_w^{n+1}, \lambda_h}_{\Omega_p} = \dual{\tau_w^{n+1}, \lambda_h}_{\Omega_p} + \dual{\dot{\xi}_w^{n+1}, \lambda_h}_{\Omega_p} \\
    &\quad - \dual{(\bm{\xi}_u^{n+1})_3, \lambda_h}_{\Omega_p}
    \end{aligned} &&\forall \lambda_h \in G_h. \label{eq:interface_error_decomp}
\end{align}
Setting $\bm{v}_h = \bm{\chi}_u^{n+1}$ in \eqref{eq:momentum_error_decomp} and $q_h = \chi_{p}^{n+1}$ in \eqref{eq:continuity_error_decomp}, we first see that by \eqref{eq:stokes-proj}, we have $ (\nabla\cdot \bm{\xi}_u^{n+1}, \chi_p^{n+1})_{\Omega_f} = 0$ and so, $(\nabla\cdot \bm{\chi}_u^{n+1}, \chi_p^{n+1})_{\Omega_f} = 0$ from \eqref{eq:continuity_error_decomp}. By using the identity $(a - b,a)_{\gamma} = \frac{1}{2}(\norm{a}_\gamma^2 - \norm{b}_\gamma^2 + \norm{a-b}_\gamma^2)$, we have
\begin{equation} \label{eq:fluid_error}
    \begin{aligned}
        &\frac{\rho_f}{2}\left(\norm{\bm{\chi}_u^{n+1}}_{\Omega_f}^2-\norm{\bm{\chi}_u^{n}}^2_{\Omega_f} + \norm{\bm{\chi}_u^{n+1} - \bm{\chi}_u^{n}}_{\Omega_f}^2 \right)  + \nu_f\delta t\norm{\nabla \bm{\chi}_u^{n+1}}_{\Omega_f}^2 + \delta t\dual{\chi_g^{n+1}, (\bm{\chi}_u^{n+1})_3}_{\Omega_p} \\
        &\quad = - \rho_f\delta t (\tau_u^{n+1}, \bm{\chi}_u^{n+1})_{\Omega_f} - \rho_f(\bm{\xi}_u^{n+1}-\bm{\xi}_u^{n}, \bm{\chi}_u^{n+1})_{\Omega_f} - \nu_f\delta t(\nabla \bm{\xi}_u^{n+1}, \nabla\bm{\chi}_u^{n+1})_{\Omega_f}\\
        &\quad + \delta t(\xi_p^{n+1}, \nabla\cdot \bm{\chi}_u^{n+1})_{\Omega_f} - \delta t\dual{\xi_g^{n+1}, (\bm{\chi}_u^{n+1})_3}_{\Omega_p}.
    \end{aligned}
\end{equation}
Now, setting $\eta_h = \dot{\chi}_w^{n+1}$ in \eqref{eq:plate_error_decomp}, we obtain
\begin{equation} \label{eq:plate_error_3}
    \begin{aligned}
    &\frac{1}{2}\left(\norm{\dot{\chi}_w^{n+1}}_{\Omega_p}^2-\norm{\dot{\chi}_w^{n}}_{\Omega_p}^2 + \norm{\dot{\chi}_w^{n+1} - \dot{\chi}_w^{n}}_{\Omega_p}^2\right) + \frac{\rho}{2}\left(\norm{\nabla\dot{\chi}_w^{n+1}}_{\Omega_p}^2-\norm{\nabla\dot{\chi}_w^{n}}_{\Omega_p}^2 + \norm{\nabla\dot{\chi}_w^{n+1} - \nabla\dot{\chi}_w^{n}}_{\Omega_p}^2\right) \\
    &\quad +\delta t (\nabla \chi_z^{n+1}, \nabla \dot{\chi}_w^{n+1})_{\Omega_p} - \delta t\dual{\chi_g^{n+1}, \dot{\chi}_w^{n+1}}_{\Omega_p} = -\delta t (\tau_w^{n+1}, \dot{\chi}_w^{n+1})_{\Omega_p} - \rho\delta t (\tau_z^{n+1}, \dot{\chi}_w^{n+1})_{\Omega_p} \\
    &\quad - (\dot{\xi}_w^{n+1}-\dot{\xi}_w^{n}, \dot{\chi}_w^{n+1})_{\Omega_p} - \rho(\nabla\dot{\xi}_w^{n+1}-\nabla\dot{\xi}_w^{n}, \nabla\dot{\chi}_w^{n+1})_{\Omega_p} - \delta t (\nabla \xi_z^{n+1}, \nabla \dot{\chi}_w^{n+1})_{\Omega_p} + \delta t\dual{\xi_g^{n+1}, \dot{\chi}_w^{n+1}}_{\Omega_p}.
    \end{aligned}
\end{equation}
Note that if we set $\varphi_h = \chi_z^{n+1}$ for both time instants $t^{n+1}$ and $t^{n}$, we have
\begin{equation*}
    \begin{aligned}
        (\nabla\chi_w^{n+1}, \nabla\chi_z^{n+1})_{\Omega_p} - \norm{\chi_z^{n+1}}^2_{\Omega_p} &= -(\nabla\xi_w^{n+1}, \nabla\chi_z^{n+1})_{\Omega_p} + (\xi_z^{n+1}, \chi_z^{n+1})_{\Omega_p},\\
        (\nabla\chi_w^{n}, \nabla\chi_z^{n+1})_{\Omega_p} - (\chi_z^{n}, \chi_z^{n+1})_{\Omega_p} &= -(\nabla\xi_w^{n}, \nabla\chi_z^{n+1})_{\Omega_p} + (\xi_z^{n}, \chi_z^{n+1})_{\Omega_p}.
    \end{aligned}
\end{equation*}
Subtracting these two equations and using \eqref{eq:backward_diff_error} give
\begin{equation*}
    \delta t (\nabla\dot{\chi}_w^{n+1}, \nabla\chi_z^{n+1})_{\Omega_p} = \norm{\chi_z^{n+1}}^2_{\Omega_p} - (\chi_z^{n}, \chi_z^{n+1})_{\Omega_p} - \delta t(\nabla\dot{\xi}_w^{n+1}, \nabla\chi_z^{n+1})_{\Omega_p} + \delta t(\dot{\xi}_z^{n+1}, \chi_z^{n+1})_{\Omega_p}.
\end{equation*}
Hence, we have
\begin{equation} \label{eq:plate_error_4}
    \begin{aligned}
    &\frac{1}{2}\left(\norm{\dot{\chi}_w^{n+1}}_{\Omega_p}^2-\norm{\dot{\chi}_w^{n}}_{\Omega_p}^2 + \norm{\dot{\chi}_w^{n+1} - \dot{\chi}_w^{n}}_{\Omega_p}^2\right) + \frac{\rho}{2}\left(\norm{\nabla\dot{\chi}_w^{n+1}}_{\Omega_p}^2-\norm{\nabla\dot{\chi}_w^{n}}_{\Omega_p}^2 + \norm{\nabla\dot{\chi}_w^{n+1} - \nabla\dot{\chi}_w^{n}}_{\Omega_p}^2\right) \\
    &\quad + \norm{\chi_z^{n+1}}_{\Omega_p}^2 - \delta t\dual{\chi_g^{n+1}, \dot{\chi}_w^{n+1}}_{\Omega_p} = -\delta t (\tau_w^{n+1}, \dot{\chi}_w^{n+1})_{\Omega_p} - \rho\delta t (\tau_z^{n+1}, \dot{\chi}_w^{n+1})_{\Omega_p} \\
    &\quad - (\dot{\xi}_w^{n+1}-\dot{\xi}_w^{n}, \dot{\chi}_w^{n+1})_{\Omega_p} - \rho(\nabla\dot{\xi}_w^{n+1}-\nabla\dot{\xi}_w^{n}, \nabla \dot{\chi}_w^{n+1})_{\Omega_p} - \delta t (\nabla \xi_z^{n+1}, \nabla \dot{\chi}_w^{n+1})_{\Omega_p} \\
    &\quad + \delta t\dual{\xi_g^{n+1}, \dot{\chi}_w^{n+1}}_{\Omega_p} + (\chi_z^{n}, \chi_z^{n+1})_{\Omega_p} + \delta t(\nabla\dot{\xi}_w^{n+1}, \nabla\chi_z^{n+1})_{\Omega_p} - \delta t(\dot{\xi}_z^{n+1}, \chi_z^{n+1})_{\Omega_p}.
    \end{aligned}
\end{equation}
Setting $\lambda_h = \chi_g^{n+1}$ in \eqref{eq:interface_error_decomp} and multiplying by -$\delta t$ give
\begin{equation} \label{eq:interface_error_2}
    \begin{aligned}
        \delta t\dual{\dot{\chi}_w^{n+1}, \chi_g^{n+1}}_{\Omega_p} - \delta t\dual{(\bm{\chi}_u^{n+1})_3, \chi_g^{n+1}}_{\Omega_p} &= -\delta t\dual{\tau_w^{n+1}, \chi_g^{n+1}}_{\Omega_p} - \delta t\dual{\dot{\xi}_w^{n+1}, \chi_g^{n+1}}_{\Omega_p} \\
        & \quad + \delta t\dual{(\bm{\xi}_u^{n+1})_3, \chi_g^{n+1}}_{\Omega_p}.
    \end{aligned}
\end{equation}
Hence, adding \eqref{eq:fluid_error}, \eqref{eq:plate_error_4}, and \eqref{eq:interface_error_2} results to
\begin{equation}\label{eq:disc_total_err-0}
    \begin{aligned} 
        &\frac{\rho_f}{2}\left(\norm{\bm{\chi}_u^{n+1}}_{\Omega_f}^2-\norm{\bm{\chi}_u^{n}}^2_{\Omega_f} + \norm{\bm{\chi}_u^{n+1} - \bm{\chi}_u^{n}}_{\Omega_f}^2 \right)  + \nu_f\delta t\norm{\nabla \bm{\chi}_u^{n+1}}_{\Omega_f}^2\\
        & + \frac{1}{2}\left(\norm{\dot{\chi}_w^{n+1}}_{\Omega_p}^2-\norm{\dot{\chi}_w^{n}}_{\Omega_p}^2 + \norm{\dot{\chi}_w^{n+1} - \dot{\chi}_w^{n}}_{\Omega_p}^2\right) + \frac{\rho}{2}\left(\norm{\nabla\dot{\chi}_w^{n+1}}_{\Omega_p}^2-\norm{\nabla\dot{\chi}_w^{n}}_{\Omega_p}^2 + \norm{\nabla\dot{\chi}_w^{n+1} - \nabla\dot{\chi}_w^{n}}_{\Omega_p}^2\right)\\
        & + \norm{\chi_z^{n+1}}_{\Omega_p}^2  = -\rho_f\delta t (\tau_u^{n+1}, \bm{\chi}_u^{n+1})_{\Omega_f} - \delta t (\tau_w^{n+1}, \dot{\chi}_w^{n+1})_{\Omega_p} - \rho\delta t (\tau_z^{n+1},\dot{\chi}_w^{n+1}) -\delta t\dual{\tau_w^{n+1}, \chi_g^{n+1}}_{\Omega_p} \\
        & - \rho_f(\bm{\xi}_u^{n+1}-\bm{\xi}_u^{n}, \bm{\chi}_u^{n+1})_{\Omega_f} - \nu_f\delta t(\nabla \bm{\xi}_u^{n+1}, \nabla\bm{\chi}_u^{n+1})_{\Omega_f} + \delta t(\xi_p^{n+1}, \nabla\cdot \bm{\chi}_u^{n+1})_{\Omega_f} - \delta t\dual{\xi_g^{n+1}, (\bm{\chi}_u^{n+1})_3}_{\Omega_p} \\
        & - (\dot{\xi}_w^{n+1}-\dot{\xi}_w^{n}, \dot{\chi}_w^{n+1})_{\Omega_p} - \rho(\nabla\dot{\xi}_w^{n+1}-\nabla\dot{\xi}_w^{n}, \nabla\dot{\chi}_w^{n+1})_{\Omega_p} - \delta t (\nabla \xi_z^{n+1}, \nabla \dot{\chi}_w^{n+1})_{\Omega_p} + \delta t\dual{\xi_g^{n+1}, \dot{\chi}_w^{n+1}}_{\Omega_p} \\
        & + (\chi_z^{n}, \chi_z^{n+1})_{\Omega_p} + \delta t(\nabla\dot{\xi}_w^{n+1}, \nabla\chi_z^{n+1})_{\Omega_p} - \delta t(\dot{\xi}_z^{n+1}, \chi_z^{n+1})_{\Omega_p} - \delta t\dual{\dot{\xi}_w^{n+1}, \chi_g^{n+1}}_{\Omega_p} \\
        & + \delta t\dual{(\bm{\xi}_u^{n+1})_3, \chi_g^{n+1}}_{\Omega_p}.
    \end{aligned}
\end{equation}
Due to \eqref{eq:l2-proj}, $(\dot{\xi}_w^{n+1}-\dot{\xi}_w^{n}, \dot{\chi}_w^{n+1})_{\Omega_p} + \rho(\nabla\dot{\xi}_w^{n+1}-\nabla\dot{\xi}_w^{n}, \nabla\dot{\chi}_w^{n+1})_{\Omega_p} = 0$, $(\nabla\dot{\xi}_w^{n+1}, \nabla\chi_z^{n+1})_{\Omega_p}=-\frac{1}{\rho}(\dot{\xi}_w, \chi_z^{n+1})_{\Omega_p}$, and $\delta t(\dot{\xi}_z^{n+1}, \chi_z^{n+1})_{\Omega_p} = 0$. Hence, we can rewrite \eqref{eq:disc_total_err-0} as 
\begin{equation}\label{eq:disc_total_err}
    \begin{aligned} 
        &\frac{\rho_f}{2}\left(\norm{\bm{\chi}_u^{n+1}}_{\Omega_f}^2-\norm{\bm{\chi}_u^{n}}^2_{\Omega_f} + \norm{\bm{\chi}_u^{n+1} - \bm{\chi}_u^{n}}_{\Omega_f}^2 \right)  + \nu_f\delta t\norm{\nabla \bm{\chi}_u^{n+1}}_{\Omega_f}^2\\
        & + \frac{1}{2}\left(\norm{\dot{\chi}_w^{n+1}}_{\Omega_p}^2-\norm{\dot{\chi}_w^{n}}_{\Omega_p}^2 + \norm{\dot{\chi}_w^{n+1} - \dot{\chi}_w^{n}}_{\Omega_p}^2\right) + \frac{\rho}{2}\left(\norm{\nabla\dot{\chi}_w^{n+1}}_{\Omega_p}^2-\norm{\nabla\dot{\chi}_w^{n}}_{\Omega_p}^2 + \norm{\nabla\dot{\chi}_w^{n+1} - \nabla\dot{\chi}_w^{n}}_{\Omega_p}^2\right)\\
        & + \norm{\chi_z^{n+1}}_{\Omega_p}^2 = \sum_{k = 1}^{14} R_{k}^{n+1}.
    \end{aligned}
\end{equation}

Using \eqref{eq:back_euler}, Cauchy-Schwarz inequality, Young's inequality, Trace inequality, Poincar\'{e}-Friedrich's inequality, and Sobolev inequalities, we have the following:
\begin{equation*}
    \begin{aligned}
        &R_1^{n+1} = -\rho_f\delta t (\tau_u^{n+1}, \bm{\chi}_u^{n+1})_{\Omega_f} \leq C\delta t\norm{\tau_u^{n+1}}_{\Omega_f}^2 + \frac{\nu_f\delta t}{8}\norm{\nabla\bm{\chi}_u^{n+1}}_{\Omega_f}^2, \\
        &R_2^{n+1} = -\delta t(\tau_w^{n+1}, \dot{\chi}_w^{n+1})_{1, \Omega_p} \leq C\delta t\norm{\tau_w^{n+1}}^2_{\Omega_p} + \frac{\delta t}{8}\norm{\dot{\chi}_w^{n+1}}_{\Omega_p}^2, \\
        &R_3^{n+1} = - \rho\delta t (\tau_z^{n+1},\dot{\chi}_w^{n+1})_{\Omega_p} \leq C\delta t\norm{\tau_z^{n+1}}_{\Omega_p}^2 + \frac{\delta t}{4}\norm{\dot{\chi}_w^{n+1}}_{\Omega_p}^2, \\
        &R_4^{n+1} = -\delta t\dual{\tau_w^{n+1}, \chi_g^{n+1}}_{\Omega_p} \leq \frac{\delta t}{2\epsilon_1}\norm{\tau_w^{n+1}}_{1,\Omega_p}^2 + \frac{\delta t\epsilon_1}{2}\norm{\chi_g^{n+1}}_{-1/2\Omega_p}^2, \\
        &R_5^{n+1} = -\rho_f(\bm{\xi}_u^{n+1}-\bm{\xi}_u^{n}, \bm{\chi}_u^{n+1})_{\Omega_f} \leq \frac{C}{\delta t}\norm{\bm{\xi}_u^{n+1}-\bm{\xi}_u^{n}}_{\Omega_f}^2 + \frac{\nu_f\delta t}{32}\norm{\nabla\bm{\chi}_u^{n+1}}_{\Omega_f}^2,\\
        &R_6^{n+1} = -\nu_f\delta t(\nabla \bm{\xi}_u^{n+1}, \nabla\bm{\chi}_u^{n+1})_{\Omega_f} \leq C\delta t\norm{\nabla\bm{\xi}_u^{n+1}}_{\Omega_f}^2 + \frac{\nu_f\delta t}{32}\norm{\nabla\bm{\chi}_u^{n+1}}_{\Omega_f}^2,\\
        &R_7^{n+1} = \delta t(\xi_p^{n+1}, \nabla\cdot \bm{\chi}_u^{n+1})_{\Omega_f} \leq C\delta t\norm{\xi_p^{n+1}}_{\Omega_f}^2 + \frac{\nu_f\delta t}{32}\norm{\nabla\bm{\chi}_u^{n+1}}_{\Omega_f}^2,\\
        &R_8^{n+1} = - \delta t\dual{\xi_g^{n+1}, (\bm{\chi}_u^{n+1})_3}_{\Omega_p} \leq C\delta t\norm{\xi_g^{n+1}}_{\Omega_p}^2 + \frac{\nu_f\delta t}{32}\norm{\nabla\bm{\chi}_u^{n+1}}_{\Omega_f}^2,\\
        &R_{9}^{n+1} = - \delta t (\nabla \xi_z^{n+1}, \nabla \dot{\chi}_w^{n+1})_{\Omega_p} \leq C\delta t\norm{\nabla \xi_z^{n+1}}_{\Omega_p}^2 + \frac{\rho\delta t}{4}\norm{\nabla \dot{\chi}_w^{n+1}}_{\Omega_p}^2,\\
        &R_{10}^{n+1} = - \delta t\dual{\xi_g^{n+1}, \dot{\chi}_w^{n+1}}_{\Omega_p}  \leq C\delta t\norm{\xi_g^{n+1}}_{\Omega_p}^2 + \frac{\rho\delta t}{4}\norm{\nabla\dot{\chi}_w^{n+1}}_{\Omega_p}^2,\\
        &R_{11}^{n+1} = (\chi_z^{n}, \chi_z^{n+1})_{\Omega_p} \leq \frac{1}{2}\norm{\chi_z^{n}}_{\Omega_p}^2 + \frac{1}{2}\norm{\chi_z^{n+1}}_{\Omega_p}^2, \\
        &R_{12}^{n+1} = \delta t(\nabla\dot{\xi}_w^{n+1}, \nabla\chi_z^{n+1})_{\Omega_p} = -\frac{\delta t}{\rho}(\dot{\xi}_w^{n+1}, \chi_z^{n+1})_{\Omega_p} \leq \frac{C}{\delta t}\norm{\xi_w^{n+1} - \xi_w^{n}}_{\Omega_p}^2 + \frac{\delta t}{2}\norm{\chi_z^{n+1}}_{\Omega_p}^2 \\
        &R_{13}^{n+1} = - \delta t\dual{\dot{\xi}_w^{n+1}, \chi_g^{n+1}}_{\Omega_p} = - \dual{\xi_w^{n+1} - \xi_w^{n}, \chi_g^{n+1}}_{\Omega_p} \leq \frac{1}{2\epsilon_2}\norm{\xi_w^{n+1} - \xi_w^{n}}_{1,\Omega_p}^2 + \frac{\epsilon_2}{2}\norm{\chi_g^{n+1}}_{-1/2,\Omega_p}^2,\\
        &R_{14}^{n+1} = \delta t\dual{(\bm{\xi}_u^{n+1})_3, \chi_g^{n+1}}_{\Omega_p} \leq \frac{\delta t}{2\epsilon_3}\norm{\nabla\bm{\xi}_{u}^{n+1}}_{\Omega_f}^2 + \frac{\epsilon_3\delta t}{2}\norm{\chi_g^{n+1}}_{-1/2,\Omega_p}^2.
    \end{aligned}
\end{equation*}
where we use $C>0$ as a generic constant independent of the mesh sizes $h_f$ and $h_p$, and the time step $\delta t$. Before continuing with the result above, we develop estimates for $\norm{\chi_p^{n+1}}_{\Omega_f}$ and $\norm{\chi_g^{n+1}}_{-1/2, \Omega_p}$. From the discrete inf-sup conditions \eqref{eq:disc_infsup_Stokes} and \eqref{infsup_LM_Stokes}, and inequalities \eqref{V2_Prop}-\eqref{V21_Prop}, we have
%{\color{red} TTHIS IS NOT TRUE. The inf-sup conditioin for (p,g) and v follows from the argument on p17 after (95), ignoring the $\varphi$ part. I think it should be
\begin{equation*}
\beta_{pf} \left( \norm{\chi_p^{n+1}}_{\Omega_f} +  \norm{\chi_g^{n+1}}_{-1/2, \Omega_p} \right) \leq \sup_{\bm{v}_h\in \bm{U}_h} \frac{-(\nabla\cdot\bm{v}_h, \chi_p^{n+1})_{\Omega_f} + \dual{\chi_g^{n+1}, (\bm{v_h})_3}_{\Omega_p}}{\norm{\bm{v}_h}_{1,\Omega_f}},
\end{equation*}
where $\beta_{pf} = \frac{\min\{\beta^*, \beta_f\}}{2+ \frac{\sqrt{3}}{\beta^*}}$. %Please check. You can say that this inf-sup is easily obtain from the step (add the equation number there) of the proof.

Using \eqref{eq:momentum_error_decomp}, Cauchy-Schwarz inequality and the fact that $\norm{\cdot}_{\gamma} \leq \norm{\cdot}_{1, \gamma}$, we have
\begin{equation*} 
    \begin{aligned}
        \beta_{pf}\left( \norm{\chi_p^{n+1}}_{\Omega_f} + \norm{\chi_g^{n+1}}_{-1/2, \Omega_p} \right) &\leq \frac{\rho_f}{\delta t}\norm{\bm{\chi}_u^{n+1} - \bm{\chi}_u^{n}}_{\Omega_f} + \nu_f\norm{\nabla\bm{\chi}_u^{n+1
        }}_{\Omega_f} +  \rho_f\norm{\tau_u^{n+1}}_{\Omega_f} \\
        & + \frac{\rho_f}{\delta t}\norm{\bm{\xi}_u^{n+1} - \bm{\xi}_u^{n}}_{\Omega_f} + \nu_f\norm{\nabla\bm{\xi}_u^{n+1}}_{\Omega_f} + \norm{\xi_p^{n+1}}_{\Omega_f} + \norm{\xi_g^{n+1}}_{\Omega_p}.
    \end{aligned}
\end{equation*}
Hence, using the fact that if $\abs{a_1 + a_2} \leq \abs{b_1 + b_2 + \cdots + b_n}$ and $a_1, a_2>0$ , then $a_1^2 + a_2^2 \leq (a_1 + a_2)^2 \leq (b_1 + b_2 + \cdots + b_n)^2 \leq n(b_1^2 + b_2^2 + \cdots + b_n^2)$, we obtain
\begin{equation} \label{eq:LM_pressure_bound}
    \begin{aligned}
        \beta_{pf}^2\left( \norm{\chi_p^{n+1}}_{\Omega_f}^2 + \norm{\chi_g^{n+1}}_{-1/2, \Omega_p}^2\right) &\leq \frac{7\rho_f^2}{(\delta t)^2}\norm{\bm{\chi}_u^{n+1} - \bm{\chi}_u^{n}}_{\Omega_f}^2 + 7\nu_f^2\norm{\nabla\bm{\chi}_u^{n+1
        }}_{\Omega_f}^2 \\
        &\quad + 7\rho_f^2\norm{\tau_u^{n+1}}_{\Omega_f}^2 + \frac{7\rho_f^2}{(\delta t)^2}\norm{\bm{\xi}_u^{n+1} - \bm{\xi}_u^{n}}_{\Omega_f}^2 \\
        &\quad + 7\nu_f^2\norm{\nabla\bm{\xi}_u^{n+1}}_{\Omega_f}^2 + 7\norm{\xi_p^{n+1}}_{\Omega_f}^2 + 7\norm{\xi_g^{n+1}}_{\Omega_p}^2.
    \end{aligned}
\end{equation}

We multiply \eqref{eq:LM_pressure_bound} by $\alpha>0$ which will be defined later, and add it to \eqref{eq:disc_total_err}. 
\begin{equation}\label{eq:disc_total_err_2}
    \begin{aligned} 
        &\frac{\rho_f}{2}\left(\norm{\bm{\chi}_u^{n+1}}_{\Omega_f}^2-\norm{\bm{\chi}_u^{n}}^2_{\Omega_f} + \norm{\bm{\chi}_u^{n+1} - \bm{\chi}_u^{n}}_{\Omega_f}^2 \right)  + \nu_f\delta t\norm{\nabla \bm{\chi}_u^{n+1}}_{\Omega_f}^2\\
        & + \frac{1}{2}\left(\norm{\dot{\chi}_w^{n+1}}_{\Omega_p}^2-\norm{\dot{\chi}_w^{n}}_{\Omega_p}^2 + \norm{\dot{\chi}_w^{n+1} - \dot{\chi}_w^{n}}_{\Omega_p}^2\right) + \frac{\rho}{2}\left(\norm{\nabla\dot{\chi}_w^{n+1}}_{\Omega_p}^2-\norm{\nabla\dot{\chi}_w^{n}}_{\Omega_p}^2 + \norm{\nabla\dot{\chi}_w^{n+1} - \nabla\dot{\chi}_w^{n}}_{\Omega_p}^2\right)\\
        & + \frac{1}{2}\left(\norm{\chi_z^{n+1}}_{\Omega_p}^2-\norm{\chi_z^{n}}_{\Omega_p}^2\right) + \alpha\beta_{pf}^2 \norm{\chi_p^{n+1}}_{\Omega_f}^2 + \alpha\beta_{pf}^2 \norm{\chi_g^{n+1}}_{-1/2, \Omega_p}^2 \\
        &\quad \leq \frac{7\rho_f^2\alpha}{(\delta t)^2}\norm{\bm{\chi}_u^{n+1} - \bm{\chi}_u^{n}}_{\Omega_f}^2 + \left(7\alpha\nu_f^2 + \frac{\nu_f\delta t}{4}\right)\norm{\nabla\bm{\chi}_u^{n+1
        }}_{\Omega_f}^2 + \frac{\delta t}{2}\norm{\dot{\chi}_w^{n+1}}_{\Omega_p}^2 + \frac{\rho\delta t}{2}\norm{\nabla\dot{\chi}_w^{n+1}}_{\Omega_p}^2 \\
        & \qquad + \frac{\delta t}{2}\norm{\chi_z^{n+1}}_{\Omega_p}^2 + \left(\frac{\epsilon_1\delta t+\epsilon_2+\epsilon_3\delta t}{2}\right)\norm{\chi_g^{n+1}}_{-1/2,\Omega_p}^2 + \left(C\delta t +  7\alpha\rho_f^2\right)\norm{\tau_u^{n+1}}_{\Omega_f}^2 \\
        &\qquad  + \left(C\delta t + \frac{\delta t}{2\epsilon_1}\right)\norm{\tau_w^{n+1}}_{1, \Omega_p}^2 + C\delta t\norm{\tau_z^{n+1}}_{\Omega_p}^2 + \left(\frac{C}{\delta t} + \frac{7\rho_f^2\alpha}{(\delta t)^2}\right)\norm{\bm{\xi}_u^{n+1} - \bm{\xi}_u^{n}}_{\Omega_f}^2 \\
        &\qquad + \left(C\delta t + \frac{\delta t}{2\epsilon_3} + 7\alpha\nu_f^2\right)\norm{\nabla\bm{\xi}_{u}^{n+1}}_{\Omega_f}^2  + \left(C\delta t + 7\alpha\right)\norm{\xi_p^{n+1}}_{\Omega_f}^2 + \left(2C\delta t + 7\alpha\right)\norm{\xi_g^{n+1}}_{\Omega_p}^2 \\
        &\qquad + C\delta t\norm{\nabla\xi_z^{n+1}}_{\Omega_p}^2 + \left(\frac{C}{\delta t} + \frac{1}{2\epsilon_2}\right)\norm{\xi_w^{n+1} - \xi_w^{n}}_{1,\Omega_p}^2.
    \end{aligned}
\end{equation}
We define the following positive constants:
\begin{align*}
    C_1 = \frac{\rho_f}{2} - \frac{7\rho_f^2\alpha}{(\delta t)^2} > 0, \quad C_2 = \frac{3\nu_f}{4} - \frac{7\alpha\nu_f^2}{\delta t} > 0, \quad \overline{\alpha} = \alpha\beta_{pf}^2 - \left(\frac{\epsilon_1\delta t + \epsilon_2 + \epsilon_3\delta t}{3}\right) > 0.
\end{align*}
To ensure their positivity, we set $\alpha = \frac{1}{2}\min\left\{\frac{(\delta t)^2}{14\rho_f}, \frac{3\delta t}{28\nu_f}\right\} \sim (\delta t)^2$ if $\delta t \ll 1$, $\epsilon_1 = \epsilon_3 = \frac{\alpha\beta_{pf}^2}{2\delta t} \sim \delta t$, and $\epsilon_2 = \delta t\epsilon_1 \sim (\delta t)^2$. 
Note that $\overline{\alpha}\sim (\delta t)^2$. The constants $\alpha$ and $\overline{\alpha}$ will appear in the final error estimate and affect the convergence rate in time. 

From \eqref{eq:stokes_proj_bound}, we have
\begin{equation*}
    \begin{aligned}
        \norm{\bm{\xi}_u^{n+1} - \bm{\xi}_u^{n}}_{\Omega_f} &= \norm{\left(\bm{u}(t^{n+1}) - \Pi_u \bm{u}(t^{n+1})\right) - \left(\bm{u}(t^{n}) - \Pi_u \bm{u}(t^{n})\right)}_{\Omega_f} \\
        &= \norm{\left(\bm{u}(t^{n+1}) - \bm{u}(t^{n})\right) - \Pi_u\left(\bm{u}(t^{n+1}) - \bm{u}(t^{n})\right)}_{\Omega_f} \\
        &\leq h_f^{k}\norm{\bm{u}(t^{n+1}) - \bm{u}(t^{n})}_{k+1, \Omega_f}.
    \end{aligned}
\end{equation*}
Moreover, from the Fundamental Theorem of Calculus, we obtain
\begin{equation*}
    \bm{u}(t^{n+1}) - \bm{u}(t^{n}) = \int_{t^n}^{t^{n+1}}\partial_t\bm{u}(s)\, ds.
\end{equation*}
Using Bochner's theorem (see e.g., \cite[Appendix~E]{Evans2010}) gives
\begin{equation}\label{eq:bochner_u}
    \begin{aligned}
        \norm{\bm{\xi}_u^{n+1} - \bm{\xi}_u^{n}}_{\Omega_f} &\leq h_f^{k}\norm{\bm{u}(t^{n+1}) - \bm{u}(t^{n})}_{k+1, \Omega_f} \\
        & \leq h_f^{k}\int_{t^n}^{t^{n+1}} \norm{\partial_t\bm{u}(s)}_{k+1, \Omega_f}\, ds \\
        &\leq h_f^{k}\delta t \sup_{[t^n, t^{n+1}]}\norm{\partial_t\bm{u}}_{k+1, \Omega_f}.
    \end{aligned}
\end{equation}
We apply the same steps to obtain the inequality
\begin{equation} \label{eq:bochner_w}
    \norm{\xi_w^{n+1} - \xi_w^n}_{1, \Omega_p} \leq h_p^k~\delta t\sup_{[t^n, t^{n+1}]}\norm{\partial_t w}_{k+1, \Omega_p}.
\end{equation}
From \eqref{eq:taylor_bound}, \eqref{eq:stokes_proj_bound}, \eqref{eq:l2_proj_bound}, and \eqref{eq:disc_total_err_2}-\eqref{eq:bochner_w}, we obtain
\begin{equation}\label{eq:disc_total_err_3}
    \begin{aligned} 
        &\frac{\rho_f}{2}\left(\norm{\bm{\chi}_u^{n+1}}_{\Omega_f}^2-\norm{\bm{\chi}_u^{n}}^2_{\Omega_f}\right) + C_1\norm{\bm{\chi}_u^{n+1} - \bm{\chi}_u^{n}}_{\Omega_f}^2   + C_2\delta t\norm{\nabla \bm{\chi}_u^{n+1}}_{\Omega_f}^2\\
        & + \frac{1}{2}\left(\norm{\dot{\chi}_w^{n+1}}_{\Omega_p}^2-\norm{\dot{\chi}_w^{n}}_{\Omega_p}^2 + \norm{\dot{\chi}_w^{n+1} - \dot{\chi}_w^{n}}_{\Omega_p}^2\right) + \frac{\rho}{2}\left(\norm{\nabla\dot{\chi}_w^{n+1}}_{\Omega_p}^2-\norm{\nabla\dot{\chi}_w^{n}}_{\Omega_p}^2 + \norm{\nabla\dot{\chi}_w^{n+1} - \nabla\dot{\chi}_w^{n}}_{\Omega_p}^2\right)\\
        & + \frac{1}{2}\left(\norm{\chi_z^{n+1}}_{\Omega_p}^2-\norm{\chi_z^{n}}_{\Omega_p}^2\right) + \alpha\beta_{pf}^2 \norm{\chi_p^{n+1}}_{\Omega_f}^2 + \overline{\alpha} \norm{\chi_g^{n+1}}_{-1/2, \Omega_p}^2 \\
        &\quad \leq \frac{\delta t}{2}\norm{\dot{\chi}_w^{n+1}}_{\Omega_p}^2 + \frac{\rho\delta t}{2}\norm{\nabla\dot{\chi}_w^{n+1}}_{\Omega_p}^2 + \frac{\delta t}{2}\norm{\chi_z^{n+1}}_{\Omega_p}^2 + C\delta t\left((\delta t)^2\norm{\partial_{tt}\bm{u}(\theta_u^{n+1})}_{\Omega_f}^2 \right. \\
        & \qquad + \delta t\norm{\partial_{ttt}w(\theta_w^{n+1})}_{1,\Omega_p}^2 + (\delta t)^2\norm{\partial_{ttt}z(\theta_z^{n+1})}_{\Omega_p}^2 + h_f^{2k}\sup_{[t^n, t^{n+1}]}\norm{\partial_t\bm{u}}_{k+1, \Omega_f}^2 \\
        &\qquad + \frac{h_f^{2k}}{\delta t}\norm{\bm{u}(t^{n+1})}_{k+1, \Omega_f}^2 + h_f^{2k}\norm{p(t^{n+1})}_{k, \Omega_f)}^2 + h_p^{2k}\norm{g(t^{n+1})}_{k, \Omega_p}^2 + h_p^{2k}\norm{z(t^{n+1})}_{k+1, \Omega_p}^2 \\
        & \qquad \left. + \frac{h_p^{2k}}{\delta t}\sup_{[t^n, t^{n+1}]}\norm{\partial_t w}_{k+1, \Omega_p}^2 \right).
    \end{aligned}
\end{equation}
where $t^{n}\leq \theta_u^{n+1}, \theta_w^{n+1}, \theta_z^{n+1} \leq t^{n+1}$ and $C>0$ absorbs all approximation error constants.

We drop the positive terms involving $\norm{\bm{\chi}_u^{n+1} - \bm{\chi}_u^{n}}_{\Omega_f}^2$, $\norm{\dot{\chi}_w^{n+1} - \dot{\chi}_w^{n}}_{\Omega_p}^2$, and $\norm{\nabla\dot{\chi}_w^{n+1} - \nabla\dot{\chi}_w^{n}}_{\Omega_p}^2$. Further, we also take $(\bm{u}_0)_h = \Pi_u \bm{u}_0$, $(w_0)_h = \Pi_w w_0$, and $(z_0)_h = \Pi_z z_0$ so that the terms for the initial conditions are all zero. Summing from $n=0$ to $N-1$, we obtain
\begin{equation}\label{eq:disc_total_err_4}
    \begin{aligned} 
        & C_4\left(\norm{\dot{\chi}_w^{N}}_{\Omega_p}^2 + \norm{\nabla\dot{\chi}_w^{N}}_{\Omega_p}^2 + \norm{\chi_z^{N}}_{\Omega_p}^2\right) + \delta t\sum_{n=0}^{N-1}\left(\frac{\rho_f}{2N\delta t}\norm{\bm{\chi}_u^{N}}_{\Omega_f}^2 + C_2\norm{\nabla \bm{\chi}_u^{n+1}}_{\Omega_f}^2 \right. \\
        & \left. + \frac{\alpha\beta_{pf}^2}{\delta t} \norm{\chi_p^{n+1}}_{\Omega_f}^2 + \frac{\overline{\alpha}}{\delta t} \norm{\chi_g^{n+1}}_{-1/2, \Omega_p}^2  \right) \leq \delta t\sum_{n=0}^{N-2}C_5\left(\norm{\dot{\chi}_w^{n+1}}_{\Omega_p}^2 + \norm{\nabla\dot{\chi}_w^{n+1}}_{\Omega_p}^2 + \norm{\chi_z^{n+1}}_{\Omega_p}^2 \right) \\
        &\qquad + \delta t\sum_{n=0}^{N-1}C\left((\delta t)^2\norm{\partial_{tt}\bm{u}(\theta_u^{n+1})}_{\Omega_f}^2 + \delta t\norm{\partial_{ttt}w(\theta_w^{n+1})}_{1,\Omega_p}^2\right. \\
        & \qquad + (\delta t)^2\norm{\partial_{ttt}z(\theta_z^{n+1})}_{\Omega_p}^2 + h_f^{2k}\sup_{[t^n, t^{n+1}]}\norm{\partial_t\bm{u}}_{k+1, \Omega_f}^2 + \frac{h_f^{2k}}{\delta t}\norm{\bm{u}(t^{n+1})}_{k+1, \Omega_f}^2\\
        &\qquad + h_f^{2k}\norm{p(t^{n+1})}_{k, \Omega_f)}^2 + h_p^{2k}\norm{g(t^{n+1})}_{k, \Omega_p}^2 + h_p^{2k}\norm{z(t^{n+1})}_{k+1, \Omega_p}^2 \\
        & \qquad \left. + \frac{h_p^{2k}}{\delta t}\sup_{[t^n, t^{n+1}]}\norm{\partial_t w}_{k+1, \Omega_p}^2 \right).
    \end{aligned}
\end{equation}
where $C_4 = \min\{\frac{1}{2} - \frac{\delta t}{2}, \frac{\rho}{2} - \frac{\rho\delta t}{2}\}$ and $C_5 = \max\{\frac{1}{2}, \frac{\rho}{2}\}$. For the next step, we use the discrete Gronwall's inequality \cite{Quarteroni1994, Ambartsumyan2018} with
\begin{equation}
    \begin{aligned}
        a^{n+1} &= C_4\left(\norm{\dot{\chi}_w^{n+1}}_{\Omega_p}^2 + \norm{\nabla\dot{\chi}_w^{n+1}}_{\Omega_p}^2 + \norm{\chi_z^{n+1}}_{\Omega_p}^2\right) ,\\
        b^{n+1} &= \frac{\rho_f}{2N\delta t}\norm{\bm{\chi}_u^{N}}_{\Omega_f}^2 + C_2\norm{\nabla \bm{\chi}_u^{n+1}}_{\Omega_f}^2 + \frac{\alpha\beta_{pf}^2}{\delta t} \norm{\chi_p^{n+1}}_{\Omega_f}^2 + \frac{\overline{\alpha}}{\delta t} \norm{\chi_g^{n+1}}_{-1/2, \Omega_p}^2, \\
        c^{n+1} &= C\left((\delta t)^2\norm{\partial_{tt}\bm{u}(\theta_u^{n+1})}_{\Omega_f}^2 + \delta t\norm{\partial_{ttt}w(\theta_w^{n+1})}_{1,\Omega_p}^2 +  (\delta t)^2\norm{\partial_{ttt}z(\theta_z^{n+1})}_{\Omega_p}^2 \right. \\
        & \qquad + h_f^{2k}\sup_{[t^n, t^{n+1}]}\norm{\partial_t\bm{u}}_{k+1, \Omega_f}^2 + \frac{h_f^{2k}}{\delta t}\norm{\bm{u}(t^{n+1})}_{k+1, \Omega_f}^2 + h_f^{2k}\norm{p(t^{n+1})}_{k, \Omega_f)}^2 \\
        &\qquad \left. + h_p^{2k}\norm{g(t^{n+1})}_{k, \Omega_p}^2 + h_p^{2k}\norm{z(t^{n+1})}_{k+1, \Omega_p}^2 + \frac{h_p^{2k}}{\delta t}\sup_{[t^n, t^{n+1}]}\norm{\partial_t w}_{k+1, \Omega_p}^2 \right), \\
        d^{n+1} &= \frac{C_5}{C_4}, \text{ and} \\
        B &= 0.
    \end{aligned}
\end{equation}
Then, applying the discrete Gronwall lemma, we obtain 
\begin{equation}\label{eq:disc_total_err_gronwall}
    \begin{aligned} 
        &C_4\left(\norm{\dot{\chi}_w^{N}}_{\Omega_p}^2 + \norm{\nabla\dot{\chi}_w^{N}}_{\Omega_p}^2 + \norm{\chi_z^{N}}_{\Omega_p}^2 \right) + \frac{\rho_f}{2}\norm{\bm{\chi}_u^{N}}_{\Omega_f}^2 \\
        & + \sum_{n=0}^{N-1}\left(C_2\delta t\norm{\nabla \bm{\chi}_u^{n+1}}_{\Omega_f}^2 +  \alpha\beta_{pf}^2\norm{\chi_p^{n+1}}_{\Omega_f}^2 + \overline{\alpha} \norm{\chi_g^{n+1}}_{-1/2, \Omega_p}^2 \right) \\
        &\quad \leq \exp\left(\delta t (N-1)\frac{C_5}{C_4}\right)\sum_{n=0}^{N-1}C\left((\delta t)^3\norm{\partial_{tt}\bm{u}(\theta_u^{n+1})}_{\Omega_f}^2 + (\delta t)^2\norm{\partial_{ttt}w(\theta_w^{n+1})}_{1,\Omega_p}^2\right. \\
        & \qquad + (\delta t)^3\norm{\partial_{ttt}z(\theta_z^{n+1})}_{\Omega_p}^2 + h_f^{2k}~\delta t\sup_{[t^n, t^{n+1}]}\norm{\partial_t\bm{u}}_{k+1, \Omega_f}^2 + h_f^{2k}\norm{\bm{u}(t^{n+1})}_{k+1, \Omega_f}^2\\
        &\qquad + h_f^{2k}~\delta t\norm{p(t^{n+1})}_{k, \Omega_f}^2 + h_p^{2k}~\delta t\norm{g(t^{n+1})}_{k, \Omega_p}^2 + h_p^{2k}~\delta t\norm{z(t^{n+1})}_{k+1, \Omega_p}^2 \\
        & \qquad \left. + h_p^{2k}\sup_{[t^n, t^{n+1}]}\norm{\partial_t w}_{k+1, \Omega_p}^2 \right).
    \end{aligned}
\end{equation}
Note that from \eqref{eq:back_euler} and \eqref{eq:bochner_w}, we have
\begin{align*}
    \norm{\dot{\xi}_w^{N}}_{1, \Omega_p} = \frac{1}{\delta t} \norm{\xi_w^{n+1} - \xi_w^{n+1}}_{1, \Omega_p} \leq Ch_p^k \sup_{[t^{N-1}, t^N]} \norm{\partial_t w}_{k+1, \Omega_p}.
\end{align*}
We also note that if $a_1^2+a_2^2+\cdots+a_n^2 \leq b_1^2+b_2^2+\cdots+b_m^2$. Then
\begin{equation*}
    \abs{a_1} + \abs{a_2} + \cdots + \abs{a_n} \leq \sqrt{n}\sqrt{b_1^2+b_2^2+\cdots+b_m^2} \leq \sqrt{n}\left(\abs{b_1}+\abs{b_2}+\cdots+\abs{b_m}\right).
\end{equation*}
Hence, using the triangle inequality, \eqref{eq:trunc-approx}, \eqref{eq:ddt_error} and  \eqref{eq:disc_total_err_gronwall}, together with 
$\alpha \sim \overline{\alpha}\sim (\delta t)^2$, we can compute the final error bound
\begin{equation}\label{eq:total_error}
    \begin{aligned}
        &\norm{\partial_t w(t^{N}) - \dot{w}_h^N}_{1, \Omega_p} + \norm{z(t^{N}) - z_h^N}_{\Omega_p} + \norm{\bm{u}(t^N) - \bm{u}_h^N}_{\Omega_f} + \sum_{n=0}^{N-1}\left(\sqrt{\delta t}\norm{\bm{u}(t^{n+1}) - \bm{u}_h^{n+1}}_{1,\Omega_f} \right. \\
        &\left. +~ \delta t \norm{p(t^{n+1}) - p_h^{n+1}}_{\Omega_f} + \delta t \norm{g(t^{n+1}) - g^{n+1}}_{-1/2, \Omega_p}\right) \leq \norm{\partial_t w(t^N) - \dot{w}^N}_{1,\Omega_p} + \norm{\dot{\chi}_w^{N}}_{1,\Omega_p} \\
        &\quad + \norm{\chi_z^{N}}_{\Omega_p} + \norm{\bm{\chi}_u^{N}}_{\Omega_f} + \sum_{n=0}^{N-1}\left(\sqrt{\delta t}\norm{\bm{\chi}_u^{n+1}}_{1,\Omega_f} + \delta t \norm{\chi_p^{n+1}}_{\Omega_f} + \delta t \norm{\chi_g^{n+1}}_{-1/2, \Omega_p}\right) \\
        &\quad + \norm{\dot{\xi}_w^{N}}_{1,\Omega_p} + \norm{\xi_z^{N}}_{\Omega_p} + \norm{\bm{\xi}_u^{N}}_{\Omega_f} + \sum_{n=0}^{N-1}\left(\sqrt{\delta t}\norm{\bm{\xi}_u^{n+1}}_{H^1(\Omega_f1)} + 
        \delta t \norm{\xi_p^{n+1}}_{\Omega_f} + \delta t\norm{\xi_g^{n+1}}_{\Omega_p}\right) \\
        & \leq C\left(\delta t\norm{\partial_{tt}w(\theta_w^{N})}_{1, \Omega_p} + h_p^k\sup_{[t^{N-1}, t^N]} \norm{\partial_t w}_{k+1, \Omega_p} + h_p^{k+1}\norm{z(t^N)}_{k+1, \Omega_p}\right.\\
        &\qquad \left. +~ h_f^{k}\norm{\bm{u}(t^N)}_{k+1, \Omega_f} \right) + C\sum_{n=0}^{N-1}\left(h_f^k~\sqrt{\delta t}\norm{\bm{u}(t^{n+1})}_{k+1, \Omega_f} + h_f^k~\delta t \norm{p(t^{n+1})}_{k+1, \Omega_f} \right. \\
        &\qquad \left. +~ h_p^k~\delta t\norm{g(t^{n+1})}_{k+1, \Omega_p}\right) + C^*\exp(T)\sum_{n=0}^{N-1}\left((\delta t)^{3/2}\norm{\partial_{tt}\bm{u}(\theta_u^{n+1})}_{\Omega_f} + \delta t\norm{\partial_{ttt}w(\theta_w^{n+1})}_{1,\Omega_p} \right. \\
        & \qquad + (\delta t)^{3/2}\norm{\partial_{ttt}z(\theta_z^{n+1})}_{\Omega_p} + h_f^{k}\sqrt{\delta t}\sup_{[t^n, t^{n+1}]}\norm{\partial_t\bm{u}}_{k+1, \Omega_f} + h_f^{k}\norm{\bm{u}(t^{n+1})}_{k+1, \Omega_f} \\
        &\qquad + h_f^{k}\sqrt{\delta t}\norm{p(t^{n+1})}_{k, \Omega_f} + h_p^{k}\sqrt{\delta t}\norm{g(t^{n+1})}_{k, \Omega_p} + h_p^{k}\sqrt{\delta t}\norm{z(t^{n+1})}_{k+1, \Omega_p} \\
        &\qquad \left. + h_p^{k}\sup_{[t^n, t^{n+1}]}\norm{\partial_t w}_{k+1, \Omega_p} \right).
    \end{aligned}
\end{equation}
where $C$ absorbs all approximation error constants and $C^*$ absorbs the approximation error constants and other constants in \eqref{eq:disc_total_err_gronwall}.

\end{proof}

\section{Numerical Results} \label{sec:num_res}

In this section, we present results for two numerical experiments to illustrate the validity of the partitioned algorithm for the fully discrete problem \eqref{eq:momentum_full_disc_weak}-\eqref{eq:interface_full_disc_weak}. 
For the numerical implementation of the problem \eqref{eq:momentum_full_disc_weak}-\eqref{eq:interface_full_disc_weak}, we adopt a partitioned approach based on a fixed-point iteration similar to the one implemented in \cite{Geredeli_Kunwar_Lee2024}. At each time step $t^{n+1}$, the algorithms begins with solving the fluid subproblem using $\dot{w}_h^{n}$ as an initial guess for $\dot{w}$ on the interface. The plate subproblem is then solved using the pressure trace on $\Omega_p$ computed from the fluid subproblem. Following \cite{Geredeli_Kunwar_Lee2024}, we impose the coupling condition $u_3 = \dot{w}$ on $\Omega_p$ strongly as a Dirichlet boundary condition for the fluid subproblem. All numerical experiments are carried out using FreeFem++ \cite{MR3043640}, which is an open-source software platform for numerically solving PDEs using finite element methods. 

As noted in \cite{Geredeli_Kunwar_Lee2024}, solving the 3D fluid subproblem serially and directly in FreeFem++ is possible only up to a mesh size of $h = 1/13$, which may limit the accuracy of the results we obtain. To overcome this limitation, we employ MUltifrontal Massively Parallel sparse direct Solver (MUMPS) \cite{MUMPS:1, MUMPS:2}, which solves a linear system $Ax = b$ using a parallel sparse LU decomposition, for the fluid subproblem. This allows us to refine the mesh size we use for the fluid domain more efficiently.
 In Sec.~\ref{sec:num_test1}, we present a convergence test in both space and time. In Sec.~\ref{sec:num_test2}, we simulate a free vibrating plate interacting with a viscous fluid beneath it.

\subsection{Convergence test} \label{sec:num_test1}
To validate our solver, we verify that it satisfies the optimal rates of convergence by using manufactured solutions. We consider the computational domains 
\begin{equation*}
    \Omega_f=[0,1]\times[0,1]\times[-1,0] \quad\text{and} \quad \Omega_p=[0,1]\times[0,1]\times\{0\},
\end{equation*}
and the following prescribed fluid velocity $\bm{u}=(u_1, u_2, u_3)^T$, fluid pressure $p$, plate displacement $w$, and auxialiary variable $z$: 
\begin{align*}
    u_1 &= \zeta\left(-\frac{\cos(2\pi x)}{4\pi}+\frac{\cos(4\pi x)}{16\pi}+\frac{3}{16\pi}\right)\sin^2(\pi y)\sin(2\pi y)\left(\frac{\pi}{2}\sin(\pi(z+1))\right)e^{-t}, \\
    u_2 &= 0, \\
    u_3 &= -\zeta\sin^2(\pi x)\sin(2\pi x)\sin^2(\pi y)\sin(2\pi y)\sin^2\left(\frac{\pi}{2}(z+1)\right)e^{-t},\\
    p &= 0, \\
    w &= \zeta\sin^2(\pi x)\sin(2\pi x)\sin^2(\pi y)\sin(2\pi y)e^{-t},\\
    z & = - \zeta\left(3\pi^2\sin(4\pi x) - 4\pi^2\sin^2(\pi x)\sin(2\pi x)\right)\sin^2(\pi y)\sin(2\pi y)e^{-t} \\
    &\quad -\zeta\sin^2(\pi x)\sin(2\pi x)\left(3\pi^2\sin(4\pi y) - 4\pi^2\sin^2(\pi y)\sin(2\pi y)\right)e^{-t}.
\end{align*}
For the discrete spaces, we use inf-sup stable $\mathbb{P}_2-\mathbb{P}_1$ pair for the velocity and the pressure, $\mathbb{P}_2$ for both $w$ and $z$, and $\mathbb{P}_1$ for $g$. To study spatial convergence, we fix the time step $\delta t = 1e-04$ and final time $T = 1e-03$, corresponding to $N = 10$ time steps, and refine $h_f$ and $h_p$. For simplicity, the mesh sizes for the fluid and plate domains are taken to be equal, i.e, $h_f = h_p = h$. To control the amplitude of the manufactured solutions, particularly the term involving $\Delta z_h$, we set $\zeta = (60\pi^4)^{-1}$. %{\color{red} Is this $\rho$?} 
All other parameters are set to 1.

Tabs.~\ref{tab:err_rates_fluid}-\ref{tab:err_rates_plate} report the computed errors and corresponding convergence rates for all the variables of interest $(\bm{u}, p, w, z)$ which match their theoretical convergence rates. Fig.~\ref{fig:wdot_u3_comp} presents the snapshots of the vertical component $u_{3,h}^N$ of the fluid velocity restricted on the interface/plate, i.e., at $z=0$, the plate velocity $\dot{w}_h^N$, and their absolute difference at $t = T = 1e-03$. This shows the accuracy of the enforcement of the interface coupling condition on the system.

In addition, Fig.~\ref{fig:int_wdot} shows the numerical values of $\int_{\Omega_p}\dot{w}\,d\Omega_p$ to show that this quantity remains close to machine precision for all mesh sizes.

\begin{table}[htb!]
    \centering
    \begin{tabular}{|c|c|c|c|c|c|c|}
       \hline
       \multirow{2}*{$h$}  & \multicolumn{2}{c|}{$\norm{\bm{u}(t^N)-\bm{u}_h^{N}}_{\Omega_f}$} & \multicolumn{2}{c|}{$\norm{\bm{u}(t^N)-\bm{u}_h^{N}}_{1,\Omega_f}$} & \multicolumn{2}{c|}{$\norm{p(t^N) - p_h^{N}}_{\Omega_f}$} \\ \cline{2-7} 
         & error & rate & error & rate & error & rate \\ \hline
        $1/2$ & $1.07e-05$ & & $1.70e-04$ & & $8.90e-05$ & \\  \hline
        $1/4$ & $2.22e-06$ & $2.27$ & $7.39e-05$ & $1.20$ & $1.39e-05$ & $2.67$ \\ \hline
        $1/8$ & $3.41e-07$ & $2.70$ & $2.28e-05$ & $1.69$ & $1.35e-06$ & $3.37$ \\ \hline
        $1/16$ & $4.68e-08$ & $2.87$ & $6.20e-06$ & $1.88$ & $1.27e-07$ & $3.41$ \\ \hline
        $1/32$ & $6.15e-09$ & $2.93$ & $1.59e-06$ & $1.96$ & $1.03e-08$ & $3.63$ \\ \hline
    \end{tabular}
    \caption{Errors for velocity $\bm{u}$ and pressure $p$, and corresponding rates of convergence for different spatial mesh resolution with $\delta t = 1e-04$.}
    \label{tab:err_rates_fluid}
\end{table}

\begin{table}[htb!]
    \centering
    \begin{tabular}{|c|c|c|c|c|c|c|c|c|}
       \hline
       \multirow{2}*{$h$}  & \multicolumn{2}{c|}{$\norm{w(t^N)-w_h^{N}}_{\Omega_p}$} & \multicolumn{2}{c|}{$\norm{w(t^N)-w_h^{N}}_{1,\Omega_p}$} & \multicolumn{2}{c|}{$\norm{z(t^N)-z_h^{N}}_{\Omega_p}$} & \multicolumn{2}{c|}{$\norm{z(t^N)-z_h^{N}}_{1,\Omega_p}$} \\ \cline{2-9} 
        & error & rate & error & rate & error & rate & error & rate \\ \hline
        $1/2$ & $2.25e-05$ & & $2.60e-04$ & & $3.87e-03$ & & $5.24e-02$ & \\  \hline
        $1/4$ & $5.01e-06$ & $2.17$ & $1.15e-04$ & $1.18$ & $6.56e-04$ & $2.56$ & $2.41e-02$  & $1.12$ \\ \hline
        $1/8$ & $5.58e-07$ & $3.17$ & $2.80e-05$ & $2.04$ & $1.07e-04$ & $2.62$ & $6.88e-03$ & $1.81$ \\ \hline
        $1/16$ & $6.38e-08$ & $3.13$ & $7.19e-06$ & $1.96$ & $1.40e-05$ & $2.93$ & $1.72e-03$ & $2.00$ \\ \hline
        $1/32$ & $8.07e-09$ & $2.98$ & $1.75E-06$ & $2.04$ & $2.00E-06$ & $2.81$ & $4.17e-04$ & $2.05$ \\ \hline
    \end{tabular}
    \caption{Error for $w$ and $z$, and corresponding rates of convergence for different spatial mesh resolution with $\delta t = 1e-04$.}
    \label{tab:err_rates_plate}
\end{table}

\begin{comment}
\begin{figure}[h!]
    \centering
    \begin{subfigure}{0.45\textwidth}
        \includegraphics[width = \textwidth]{figs/err_space_L2_plot.pdf}
    \end{subfigure}
    \begin{subfigure}{0.45\textwidth}
        \includegraphics[width = \textwidth]{figs/err_space_H1_plot.pdf}
    \end{subfigure}
    \begin{subfigure}{\textwidth}
        \centering
        \includegraphics[width = 0.8\textwidth]{figs/err_space_legend.pdf}
    \end{subfigure}
    \caption{Errors for all the variables of interest and their corresponding rates of convergence, with $\delta t = 1e-04$, $T=1e-03$. \lander{should I remove this figure?}}
    \label{fig:placeholder}
\end{figure}
% paraview settings: 0.34, 0.04, 32 font
\end{comment}

\begin{figure}
    \centering
    \begin{subfigure}{0.28\textwidth}
        \includegraphics[width = \textwidth]{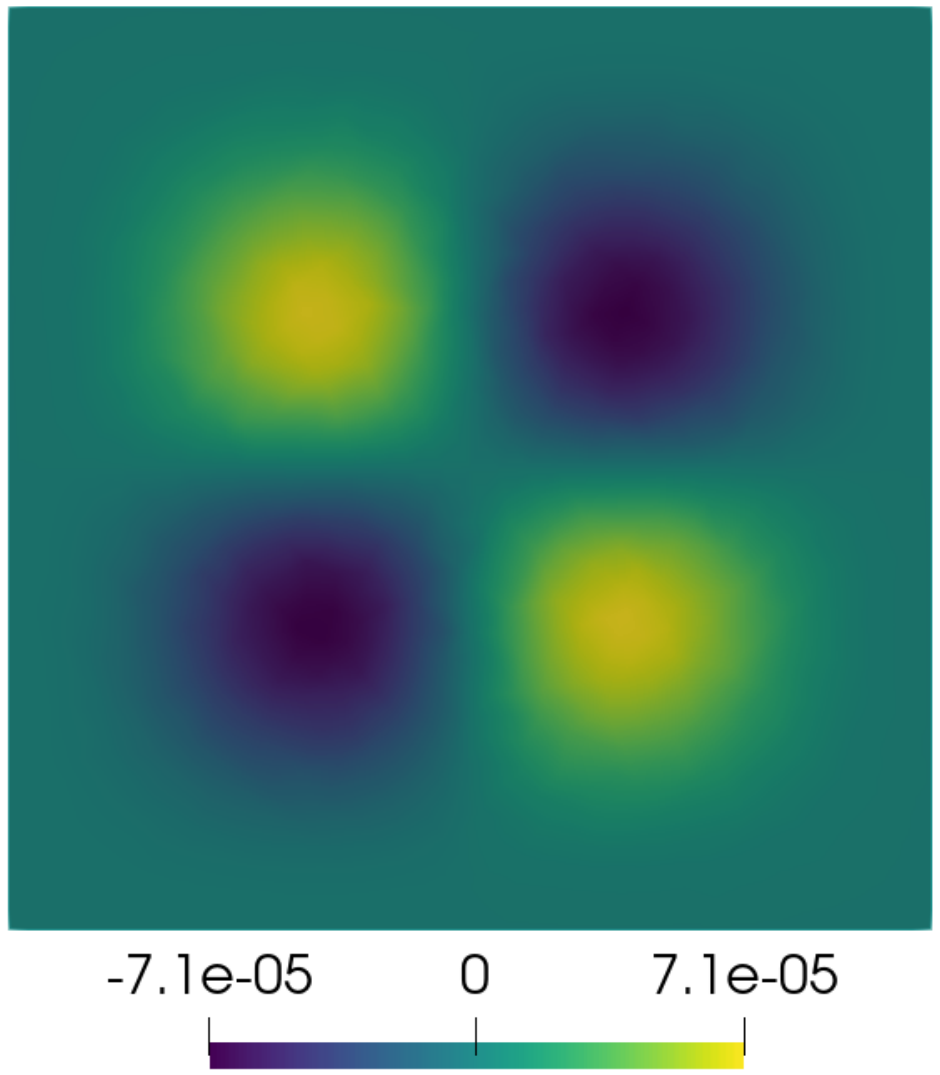}
        %\caption{$u_3(x, y, 0, T)$}
    \end{subfigure}
    \begin{subfigure}{0.28\textwidth}
        \includegraphics[width = \textwidth]{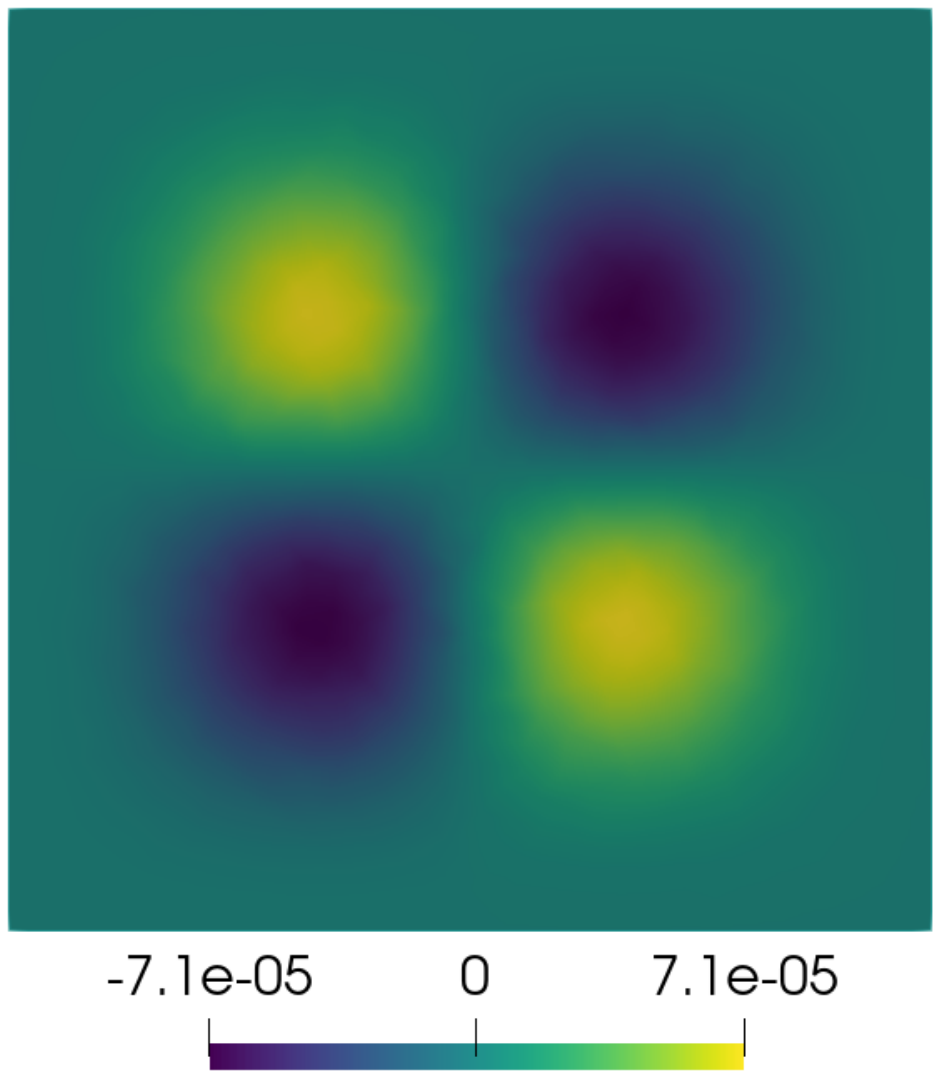}
        %\caption{$\dot{w}(x, y, T)$}
    \end{subfigure}
    \begin{subfigure}{0.28\textwidth}
        \includegraphics[width = \textwidth]{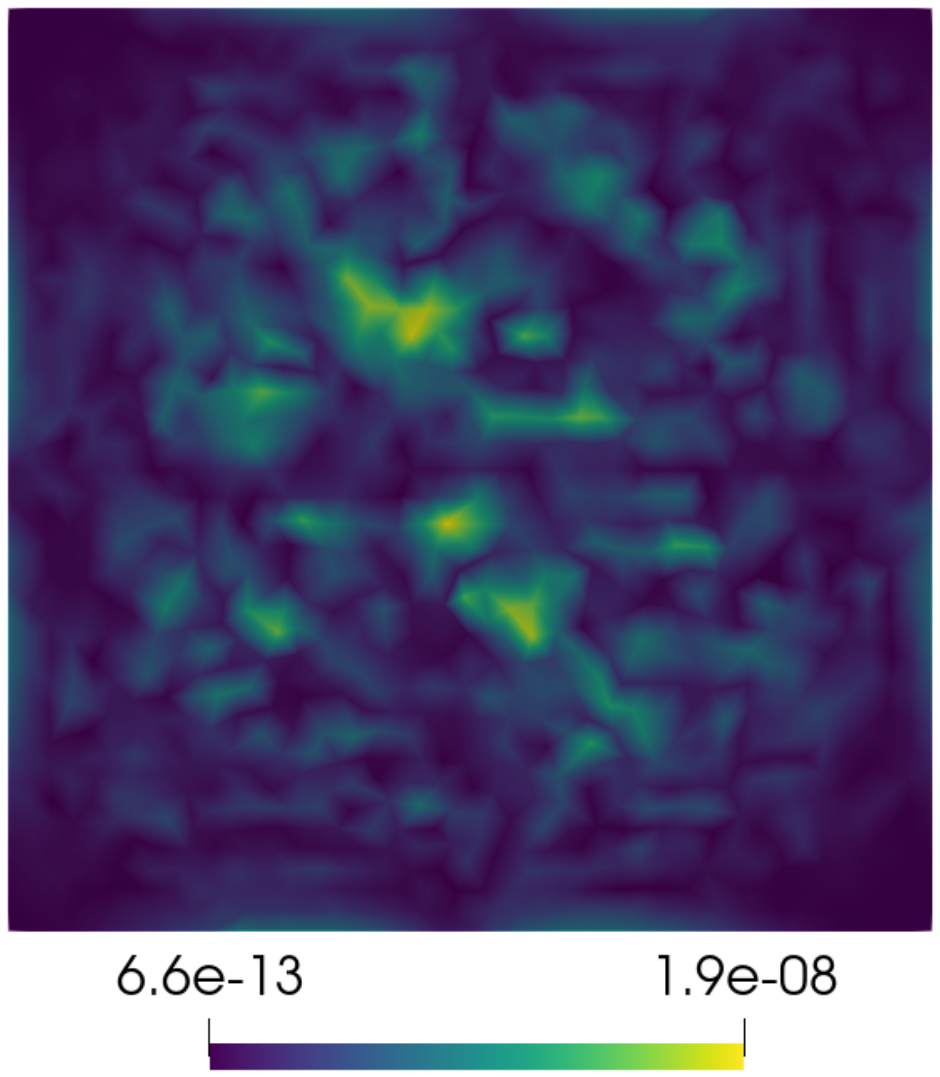}
        %\caption{$\abs{u_3(x, y, 0, T) - w(x, y, T)}$}
    \end{subfigure}
    \caption{Convergence test: vertical component $u_3$ of the velocity at $z = 0$ (left), velocity of the plate $\dot{w}$ (center), and their absolute difference (right) at $t = 1e-03$.}
    \label{fig:wdot_u3_comp}
\end{figure}

To study the convergence in time, we fix the spatial mesh to $h_f = h_p = 1/32$, the final time to $T = 1$, and vary the time step $\delta t$. As also reported in \cite{Geredeli_Kunwar_Lee2024}, the temporal errors remain essential constant as $\delta t$ decreases. Following the strategy in \cite{Geredeli_Kunwar_Lee2024}, we introduce a scaling factor for the terms with time derivatives in the plate equation:
\begin{equation*}
    \omega\partial_{tt}w + \omega\rho \partial_{tt}z - \Delta z = f_p,
\end{equation*}
in order to reduce the influence of spatial errors and isolate the temporal discretization error. With this modification, the expected first-order convergence in time becomes apparent. We set the scaling parameter to $1e+05$ and all other parameters to 1.

Tabs.~\ref{tab:err_rates_fluid_time}–\ref{tab:err_rates_plate_time} report the corresponding errors and convergence rates, which are consistent with the theoretical accuracy of the first-order time-stepping scheme for the velocity and pressure variables. We note that we observe superconvergence for the plate displacement $w$ and the auxiliary variable $z$, which may be due to the particular choice of manufactured solutions.

\begin{table}[htb!]
    \centering
    \begin{tabular}{|c|c|c|c|c|c|c|}
       \hline
       \multirow{2}*{$\delta t$}  & \multicolumn{2}{c|}{$\norm{\bm{u}(t^N)-\bm{u}_h^{N}}_{\Omega_f}$} & \multicolumn{2}{c|}{$\norm{\bm{u}(t^N)-\bm{u}_h^{N}}_{1,\Omega_f}$} & \multicolumn{2}{c|}{$\norm{p(t^N) - p_h^{N}}_{\Omega_f}$} \\ \cline{2-7} 
         & error & rate & error & rate & error & rate \\ \hline
        $1/2$ & $2.18E-06$ & & $4.70E-05$ & & $1.01E-02$ & \\  \hline
        $1/4$ & $1.14E-06$ & $0.94$ & $2.48E-05$ & $0.92$ & $4.62E-03$ & $1.13$ \\ \hline
        $1/8$ & $5.82E-07$ & $0.97$ & $1.27E-05$ & $0.96$ & $2.22E-03$ & $1.06$ \\ \hline
        $1/16$ & $2.94E-07$ & $0.99$ & $6.46E-06$ & $0.98$ & $1.09E-03$ & $1.03$ \\ \hline
    \end{tabular}
    \caption{Errors for velocity $\bm{u}$ and pressure $p$, and corresponding rates of convergence for different time steps $\delta t$ with $h = 1/32$ and $\omega=1e+05$.}
    \label{tab:err_rates_fluid_time}
\end{table}

\begin{table}[htb!]
    \centering
    \begin{tabular}{|c|c|c|c|c|c|c|c|c|}
       \hline
       \multirow{2}*{$\delta t$}  & \multicolumn{2}{c|}{$\norm{w(t^N)-w_h^{N}}_{\Omega_p}$} & \multicolumn{2}{c|}{$\norm{w(t^N)-w_h^{N}}_{1,\Omega_p}$} & \multicolumn{2}{c|}{$\norm{z(t^N)-z_h^{N}}_{\Omega_p}$} & \multicolumn{2}{c|}{$\norm{z(t^N)-z_h^{N}}_{1,\Omega_p}$} \\ \cline{2-9} 
        & error & rate & error & rate & error & rate & error & rate \\ \hline
        $1/2$ & $3.13E-04$ & & $4.32E-03$ & & $6.34E-02$ & & $9.81E-01$ & \\  \hline
        $1/4$ & $7.77E-05$ & $2.01$ & $1.11e-03$ & $1.96$ & $1.66e-02$ & $1.93$ & $2.62e-01$  & $1.91$ \\ \hline
        $1/8$ & $1.40e-05$ & $2.47$ & $2.14e-04$ & $2.37$ & $3.34e-03$ & $2.32$ & $5.38e-02$ & $2.28$ \\ \hline
        $1/16$ & $5.57e-06$ & $1.33$ & $5.21e-05$ & $2.03$ & $5.15e-04$ & $2.70$ & $5.79e-03$ & $3.22$ \\ \hline
    \end{tabular}
    \caption{Error for $w$ and $z$, and corresponding rates of convergence for different time steps $\delta t$ with $h = 1/32$ and $\omega=1e+05$.}
    \label{tab:err_rates_plate_time}
\end{table}

\begin{figure}[htb!]
    \centering
    \includegraphics[width=0.48\linewidth]{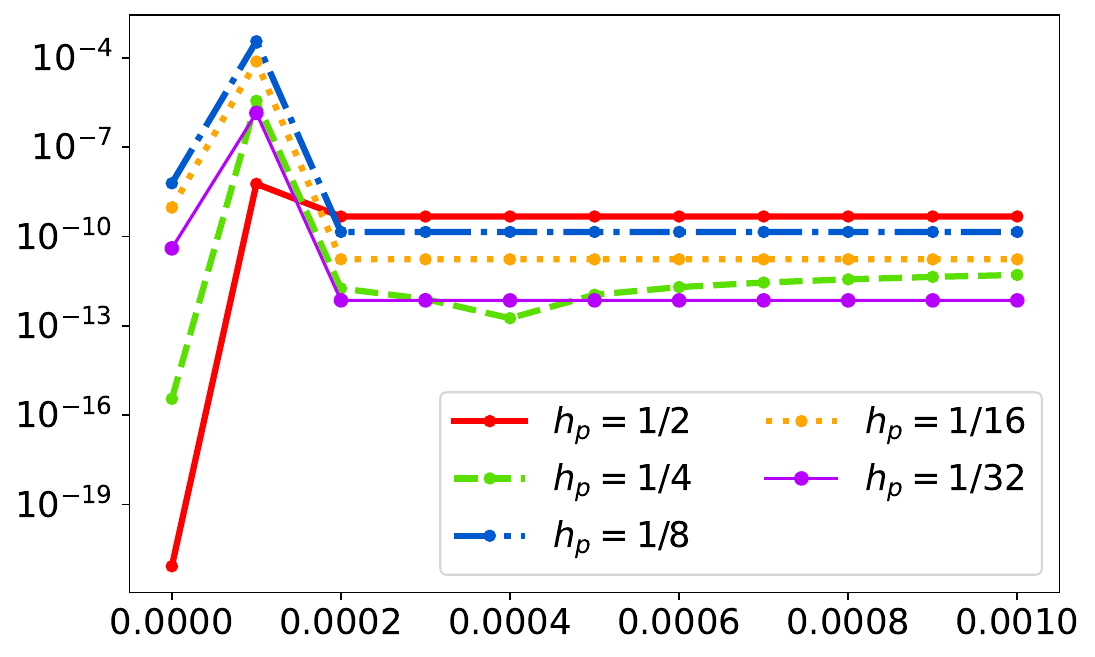}
    \caption{The absolute value of the integral $\int_{\Omega_p}\dot{w}_h^{n}\,d\Omega_p$ for different spatial mesh sizes $h_p$.}
    \label{fig:int_wdot}
\end{figure}

\subsection{Free vibrating plate test} \label{sec:num_test2}

To further assess the performance of our algorithm, we consider a physical experiment inspired by \cite{Nguyen2021, Geredeli_Kunwar_Lee2024}, involving the free vibration of a simply supported plate with a viscous incompressible fluid with computational domains
\begin{equation*}
    \Omega_f = [0,1]\times [0,1] \times [-1/2, 0] \quad \text{and} \quad \Omega_p = [0,1]\times [0,1]\times \{0\}.
\end{equation*} 
In this setting, the body force is set to $\bm{f} = \bm{0}$, and the fluid is initially at rest, i.e., $\bm{u}_0 = \bm{0}$. We prescribe a nonzero initial plate displacement but with zero velocity, i.e.,
\begin{equation}
    w(x, y, 0) = w_0 = A\sin(2\pi x)\sin(2\pi y), \quad \partial_{t}w(x, y, 0) = w_{t0} = 0
\end{equation}
which ensures that the interface condition is satisfied. We choose the amplitude to be $A = 1e-2$.

We also rewrite the plate equation to
\begin{align*}
    \rho_p \partial_{tt}w - \rho \partial_{tt}z + D \Delta z  = 0,
\end{align*}
where $\rho_p$ denotes the plate density and $D$ its flexural rigidity of the plate. Particularly, we consider the case where $\rho_p = 2.7$, $D = 6.4527$, and $\rho = 0$, corresponding to a classical Kirchhoff-Love plate model where only the bending dynamics are retained. 

For the numerical discretization, we set the mesh size to $h_f = h_p = 1/32$, a time step of $\delta t = 0.001$ s, and a final time of $T = 0.1$ s. Fig.~\ref{fig:max_disp_energy} shows the time evolution of the plate displacement and the kinetic energies of the system. The left panel, which presents the maximum displacement of the plate, demonstrates a clear decay in the amplitude as the vibration progresses in time. Meanwhile, the right panel displays the kinetic energy of the fluid, the plate, and the total kinetic energy of the system. Although the energy is initially concentrated in the plate, it is transferred to the fluid through the coupling and is subsequently dissipated. This behavior is expected since the Stokes equations governing the fluid dynamics are inherently dissipative.

Lastly, we report the snapshots of the plate velocity, the vertical component of the fluid velocity restricted on $\Omega_p$, and their absolute difference in Fig.~\ref{fig:exp2_wdot_u3}, at time instants $t = 0.01, 0.05, 0.1$. The plate velocity and the fluid velocity on $\Omega_p$ exhibit nearly identical behavior at all times, while their pointwise difference remains at least three orders of magnitude smaller than the solution amplitude. This confirms the accurate enforcement of the coupling condition throughout the simulation.

\begin{figure}
    \centering
    \begin{subfigure}{0.49\textwidth}
        \includegraphics[width=\linewidth]{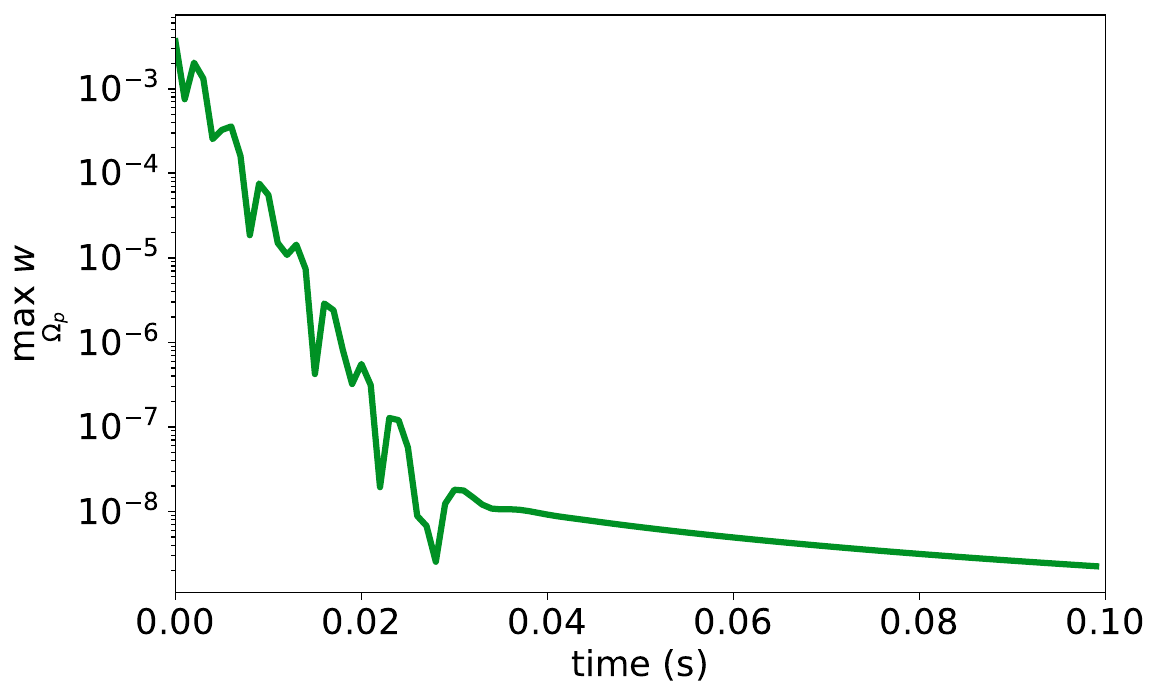}
    \end{subfigure}
    \begin{subfigure}{0.48\textwidth}
        \includegraphics[width=\linewidth]{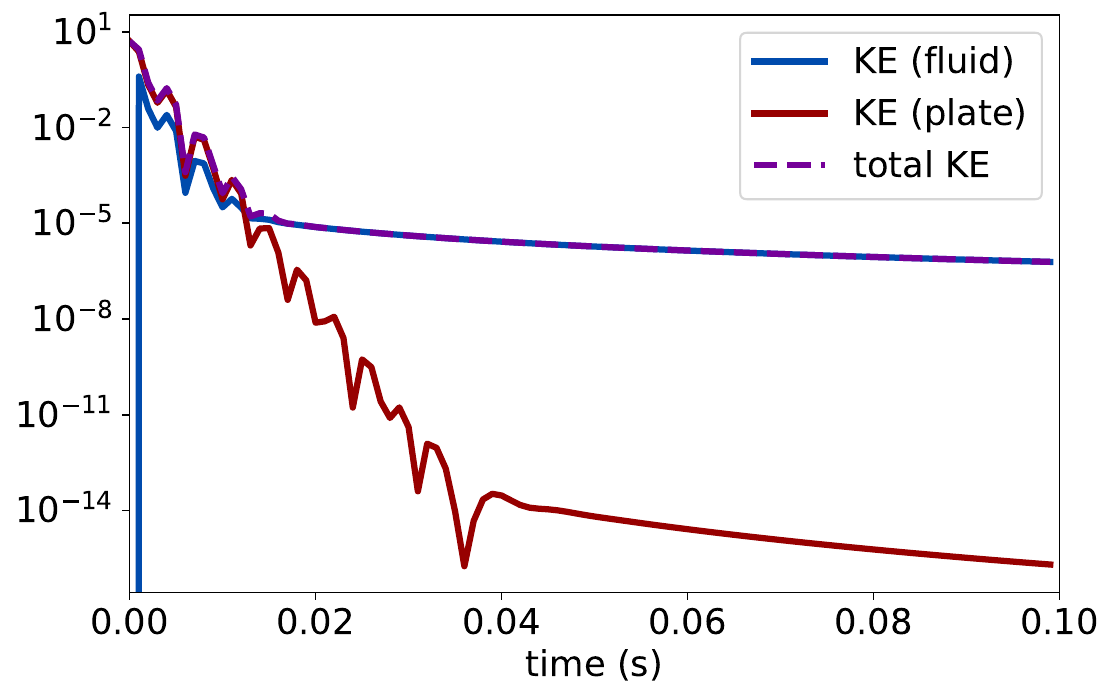}
    \end{subfigure}
    \caption{Free vibrating plate: maximum displacement of the plate (left), and kinetic energies associated to the fluid and the plate, and the corresponding total kinetic energy of the system (right) over time.}
    \label{fig:max_disp_energy}
\end{figure}

\begin{figure}
    \centering
    \begin{tabular}{cccc}
         & $\dot{w}(t)$ & $u_3(t)|_{\Omega_p}$ & $\abs{\dot{w}( t) - u_3(t)|_{\Omega_p}}$ \\
        $t = 0.01$ & \includegraphics[align=c, width = 0.25\textwidth]{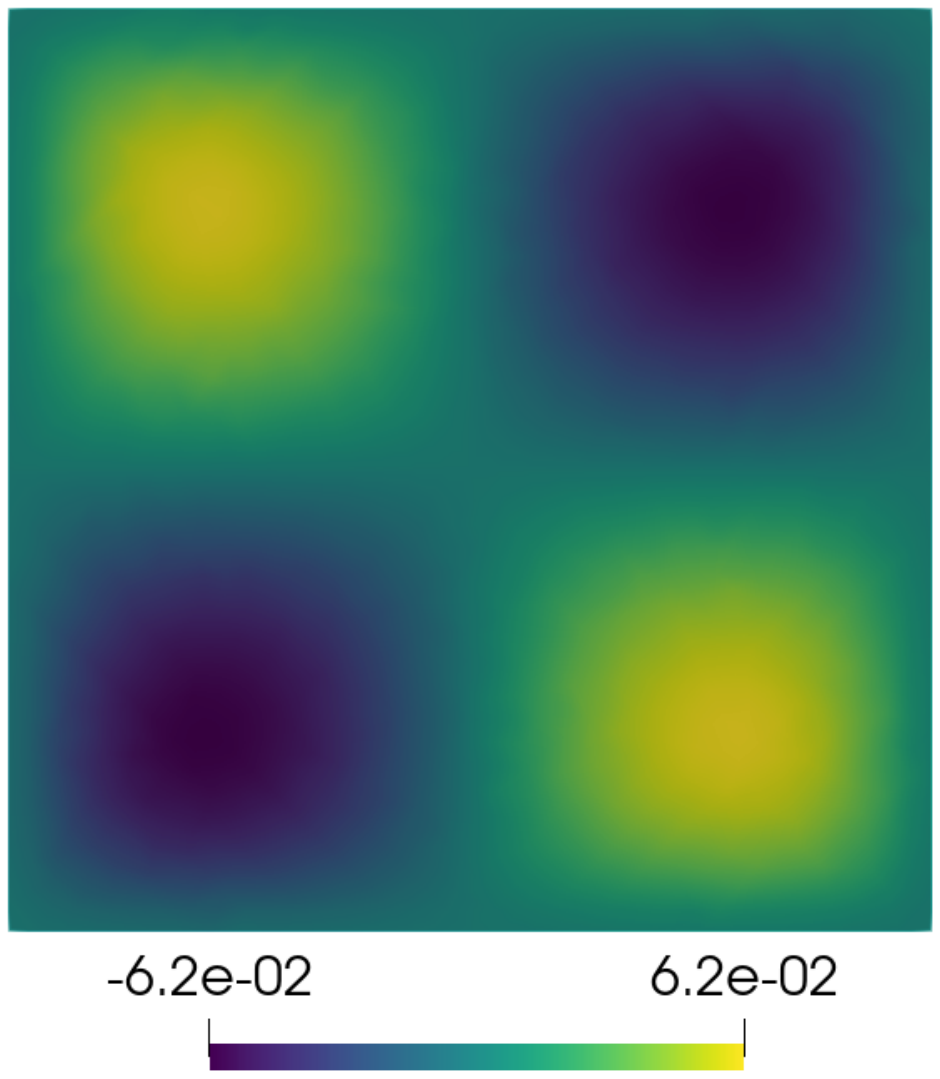} & \includegraphics[align=c, width = 0.25\textwidth]{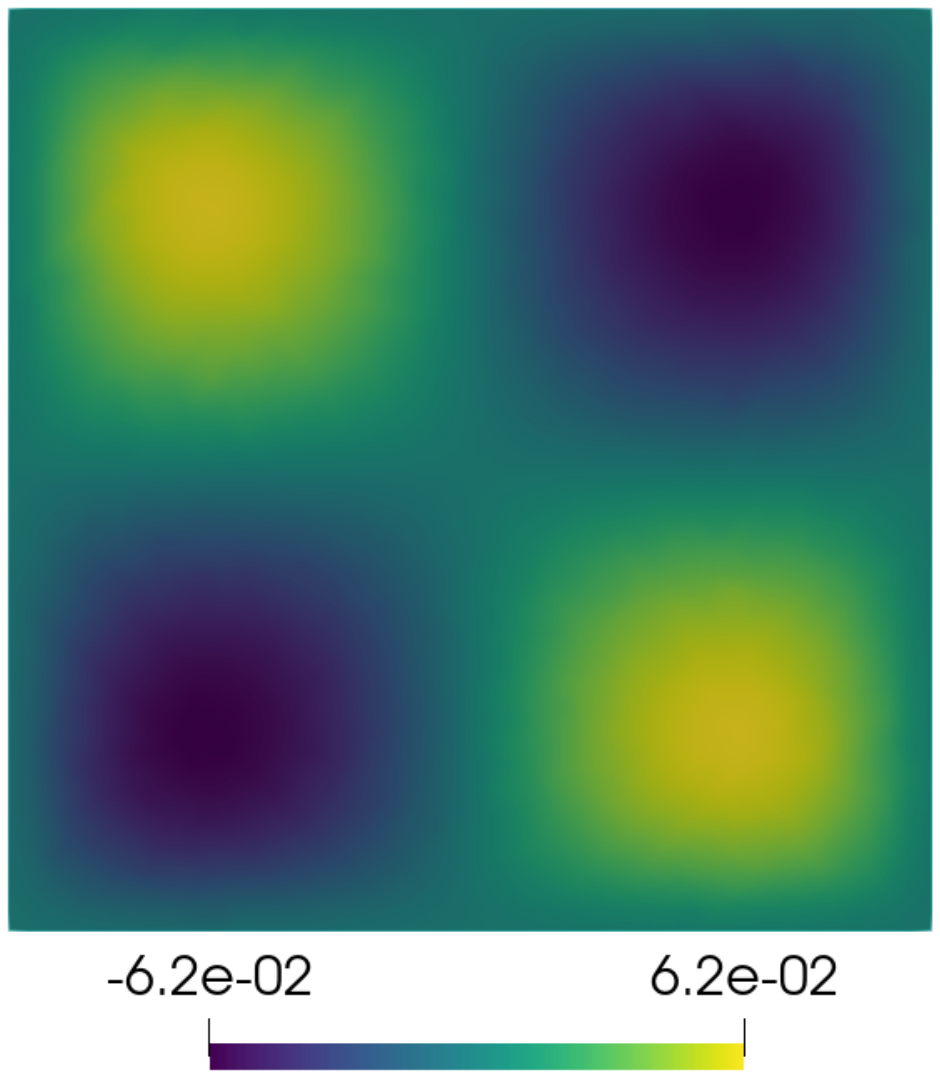} & \includegraphics[align=c, width = 0.25\textwidth]{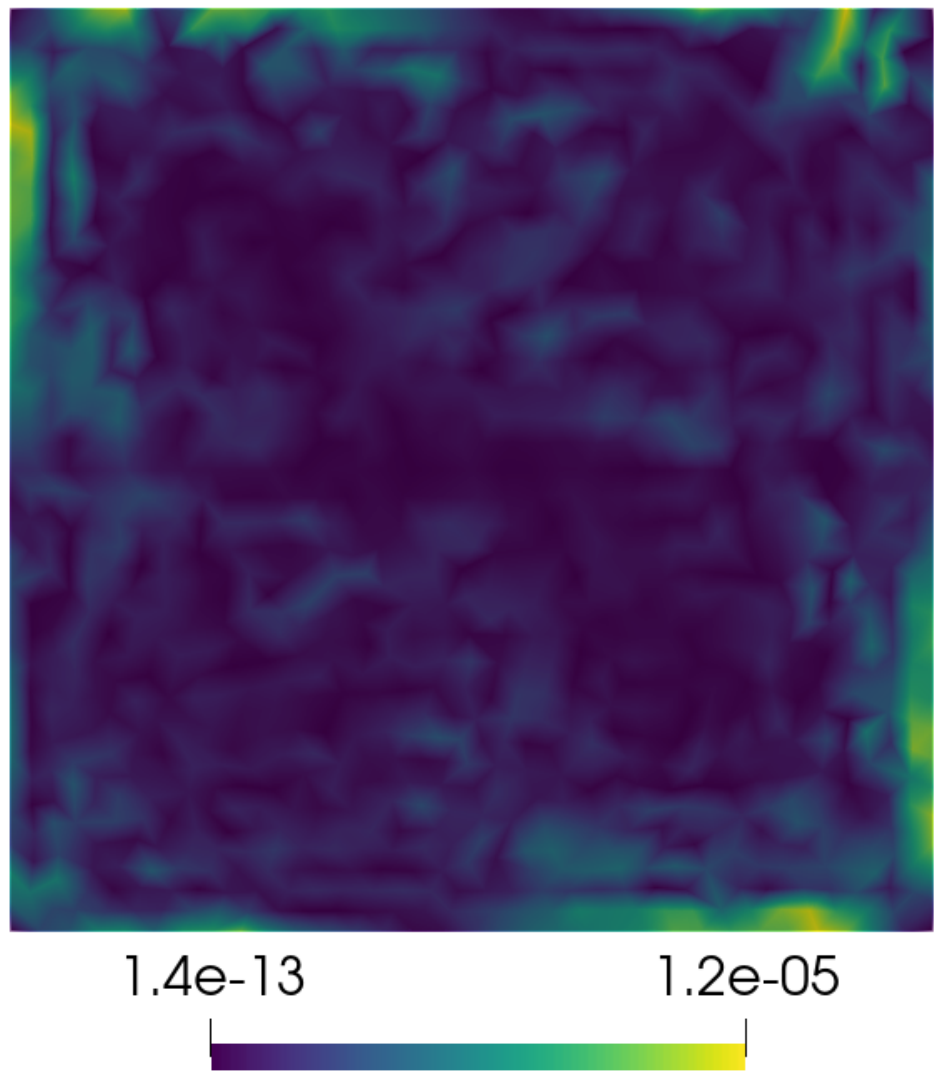} \\
        $t = 0.05$ & \includegraphics[align=c, width = 0.25\textwidth]{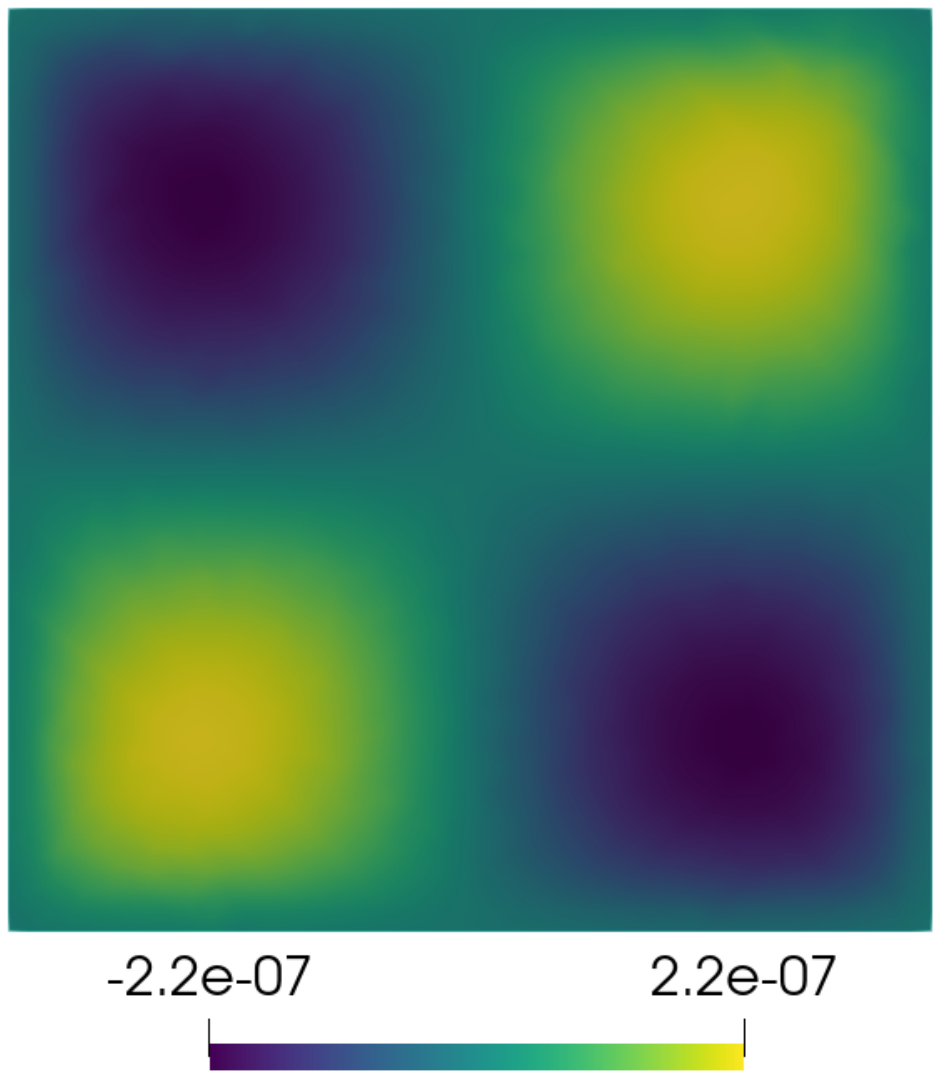} & \includegraphics[align=c, width = 0.25\textwidth]{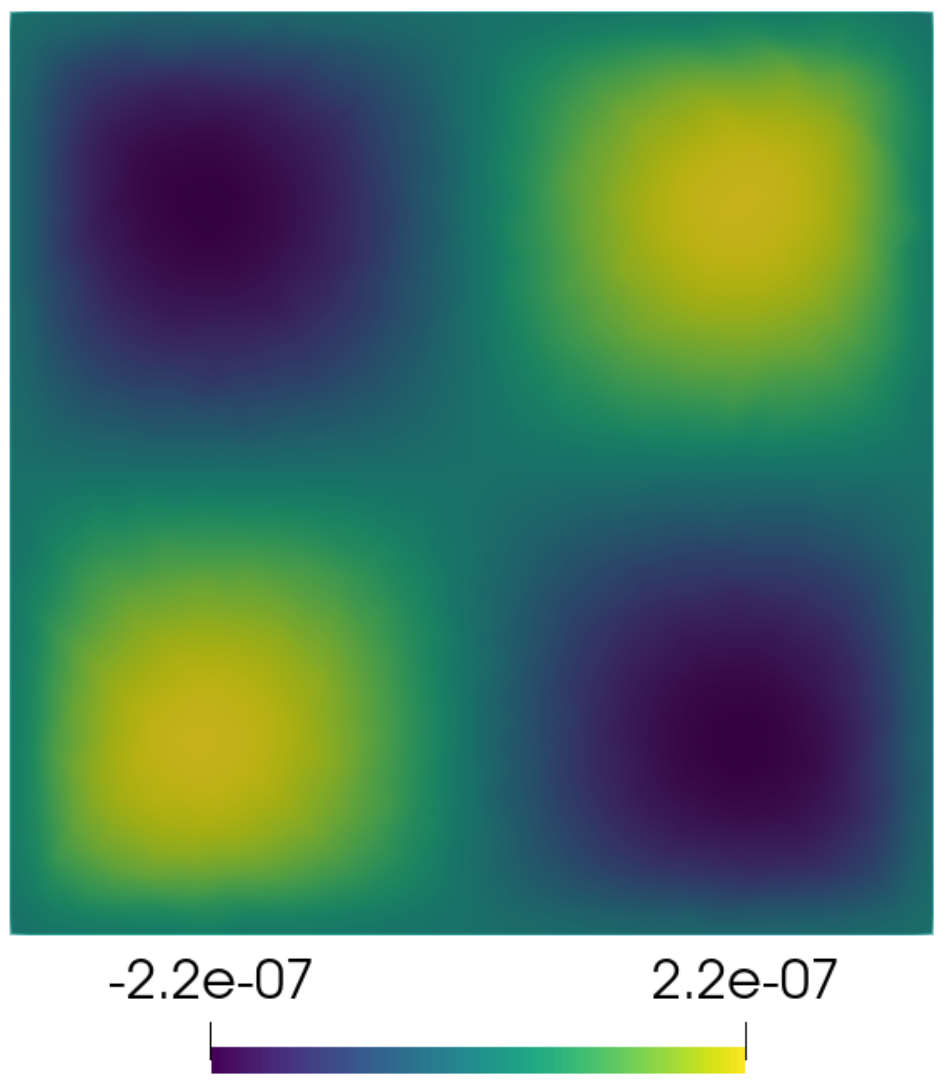} & \includegraphics[align=c, width = 0.25\textwidth]{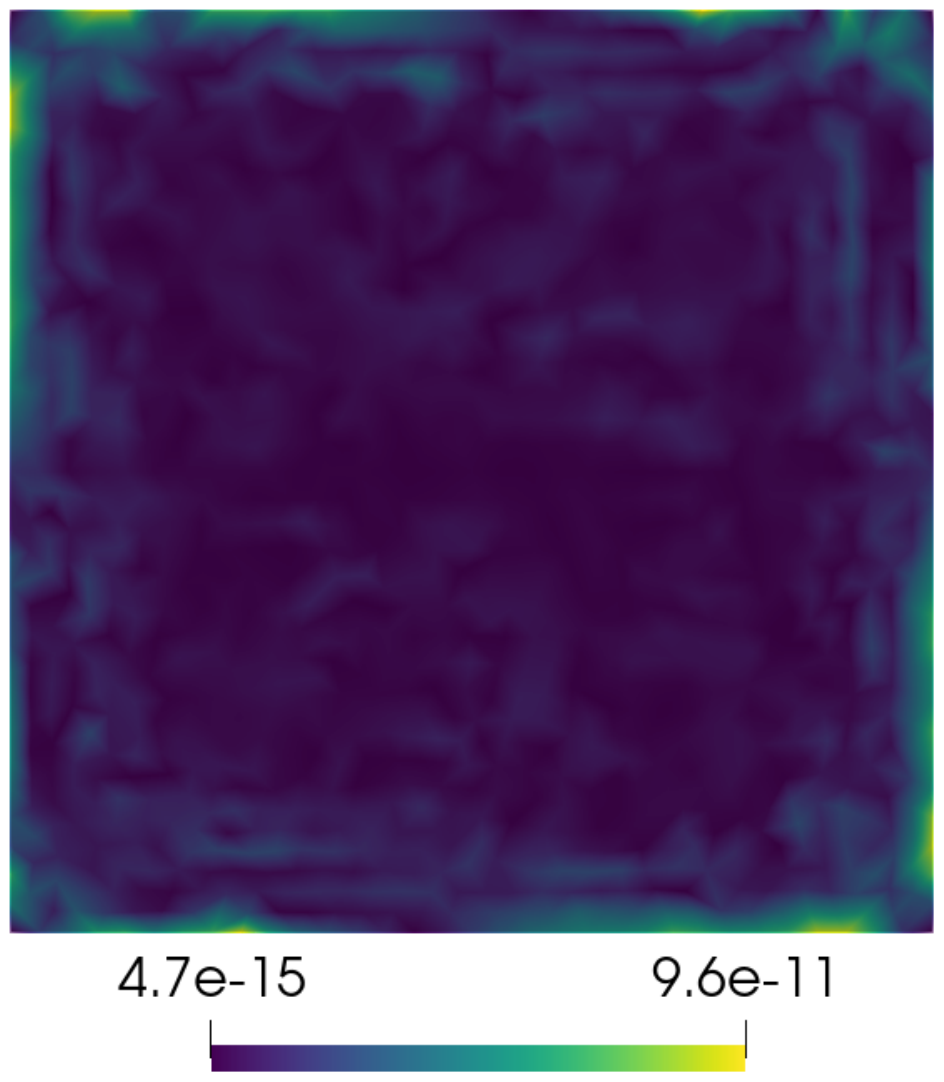} \\
        $t = 0.1$ & \includegraphics[align=c, width = 0.25\textwidth]{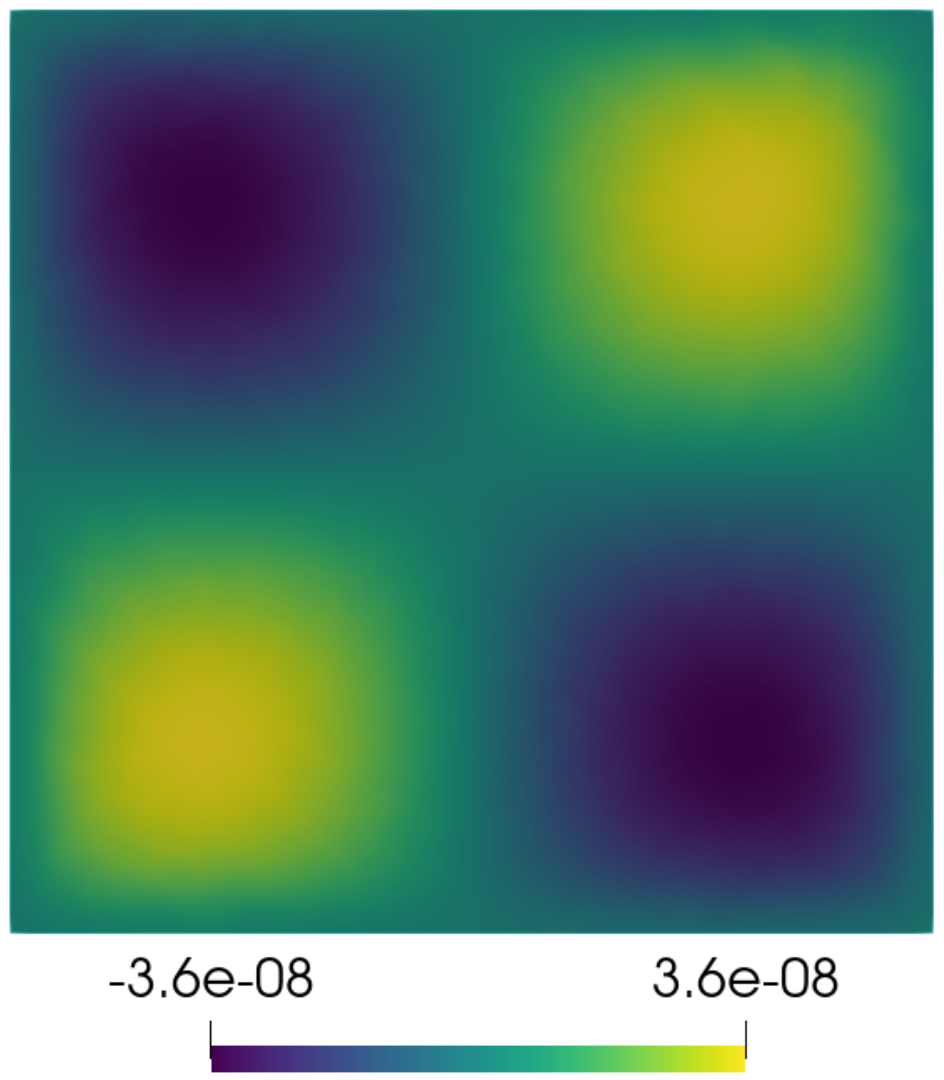} & \includegraphics[align=c, width = 0.25\textwidth]{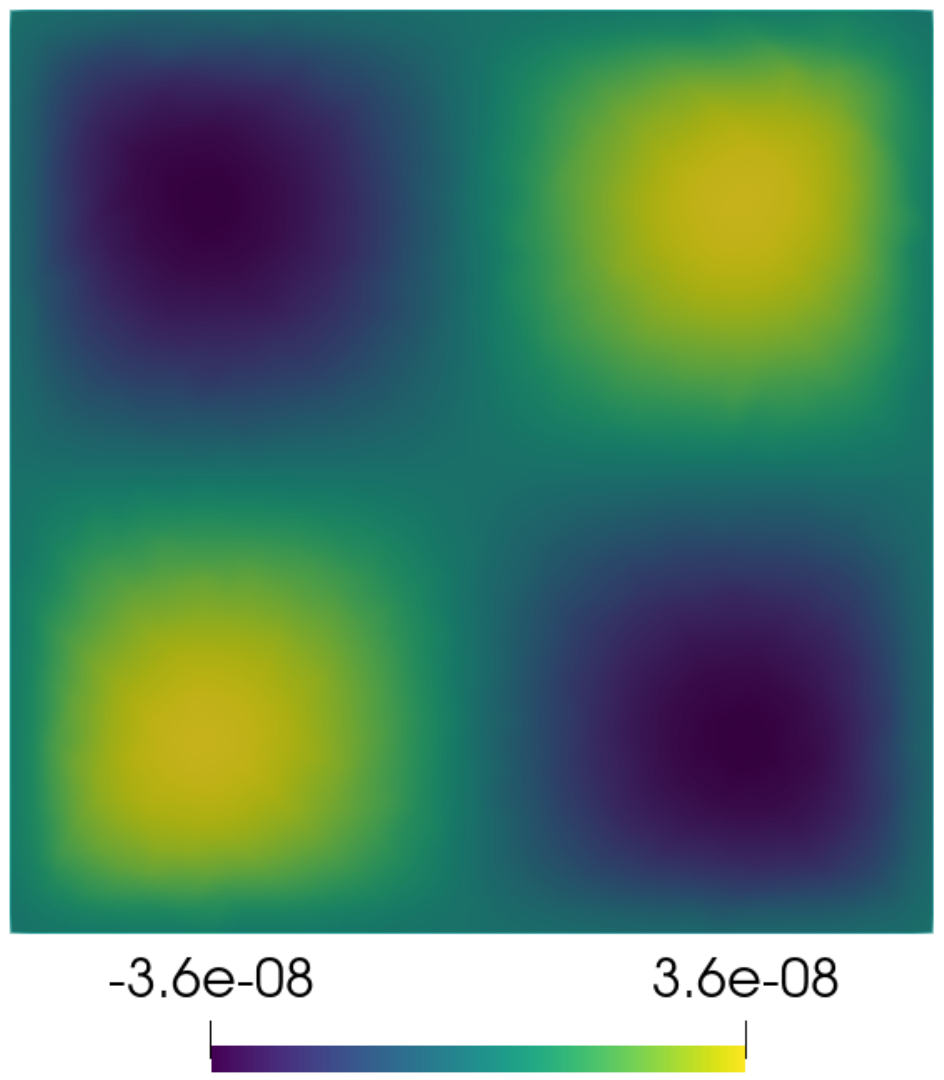} & \includegraphics[align=c, width = 0.25\textwidth]{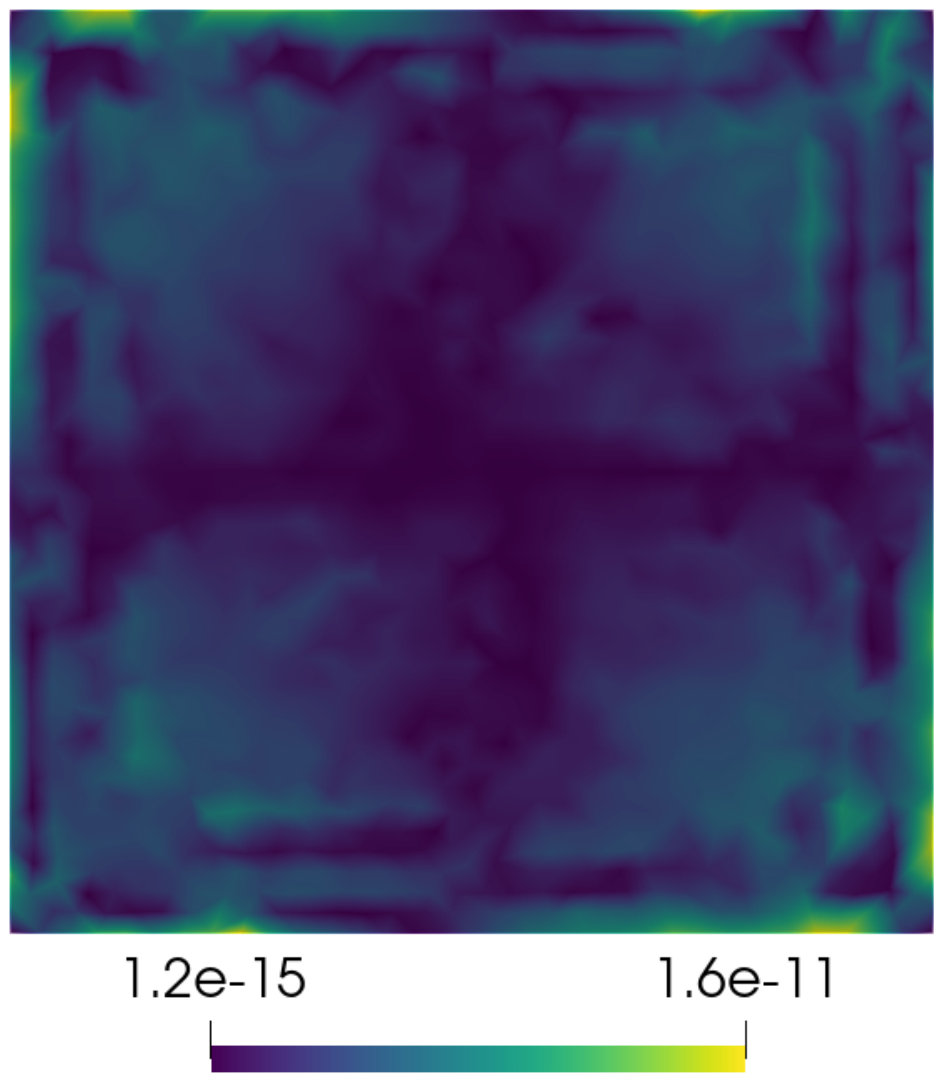}
    \end{tabular}
    \caption{Free vibrating plate: plate velocity $\dot{w}$ (left), vertical component $u_3$ of the fluid velocity at $z = 0$ (center), and their absolute difference (right) at $t = 0.01 \text{ (first row)}, 0.05 \text{ (second row)}, \text{ and }0.1 \text{ (third row)}$.}
    \label{fig:exp2_wdot_u3}
\end{figure}

\section{Conclusions} \label{sec:conclusions}

In this work, we developed and analyzed a finite element formulation for a coupled 3D fluid-2D plate interaction system, where the fluid is governed by the unsteady Stokes equations and the structure dynamics is modeled by a fourth order PDE. By reformulating the fourth-order equation as a system of second-order equations, we avoided the need for either $H^2-$conforming elements or nonconforming $\mathbb{P}_2$-Morley elements, and hence, we obtained flexibility in the choice of discrete spaces. The coupling condition between the fluid and the plate was enforced by introducing a Lagrange multiplier, leading to a saddle-point formulation of the problem.

We also established the well-posedness of both the time-discrete and fully-discrete problems, as well as the stability of the time-discrete problem, and derived a priori error estimates for the fully discrete scheme. A partitioned fixed-point algorithm was proposed for the numerical solution, in which the coupling condition in imposed as a Dirichlet boundary condition for the fluid subproblem. To overcome the mesh size limitation for the fluid domain noted in \cite{Geredeli_Kunwar_Lee2024}, we employed MUltifrontal Massively Parallel sparse direct Solver (MUMPS) \cite{MUMPS:1, MUMPS:2} for the fluid subproblem, which enabled simulations on finer meshes. Numerical experiments were presented to confirm the theoretical rates of convergence and demonstrate the applicability of the method to the physical problem of a free vibrating hinged plate.

In the future, we plan to develop a domain decomposition approach for this problem, similar to \cite{de_Castro2025}, based on a Schur complement strategy which enables the fluid and structure subproblems to be solved in parallel. Lastly, we plan to extend the present formulation to an Arbitrary Lagrangian-Eulerian (ALE) framework \cite{Donea1982}, allowing for the treatment of moving interfaces and time-dependent fluid domains.

\section*{Acknowledgements} 
{Hyesuk Lee was partially supported by the NSF under grant numbers DMS-2207971 and DMS-2513073.}

\bibliography{FSI} 

@article{de_Castro2025,
    author = {Amy de Castro and Hyesuk Lee and Margaret M. Wiecek},
    title = {{A Lagrange multiplier method for fluid-structure interaction: Well-posedness and domain decomposition}},
    journal = {Computers and Mathematics with Applications},
    volume = {181},
    pages = {193-215},
    year = {2025}
}

@article{Quaini2011,
    author = {Annalisa Quaini and Suncica Canic and David Paniagua},
    title = {{Numerical characterization of hemodynamics conditions near aortic valve after implantation of Left Ventricular Assist Device}},
    journal = {Mathematical Biosciences and Engineering},
    volume = {8},
    issue = {3},
    pages = {785-806},
    year = {2011}
}

@article{Duca2025,
    author = {Francesca Duca and Daniele Bissacco and Luca Crugnola and Chiara Faitini and Maurizio Domanin and Francesco Migliavacca1 and Santi Trimarchi and Christian Vergara},
    title = {{Computational analysis to assess hemodynamic forces in descending thoracic aortic aneurysms}},
    journal = {The Journal of Physiology},
    volume = {0},
    pages = {1-29},
    year = {2025}
}

@article{Quarteroni2000,
    author = {Alfio Quarteroni and Massimiliano Tuveri and Alessandro Veneziani},
    title = {{Computational vascular fluid dynamics: problems, models and methods}},
    journal = {Computing and Visualization in Science},
    volume = {2},
    pages = {163-197},
    year = {2000}
}

@article{Zheng2023,
    author = {Jiancai Zheng anf Nina Wang and Decheng Wan and Sergei Strijhak},
    title = {{Numerical investigations of coupled aeroelastic performance of wind turbines by elastic actuator line model}},
    journal = {Applied Energy},
    volume = {330},
    issue = {B},
    pages = {120361},
    year = {2023}
}

@article{Svacek2008,
    author = {Petr Svacek},
    title = {{Application of finite element method in aeroelasticity}},
    journal = {Journal of Computational and Applied Mathematics},
    volume = {215},
    issue = {2},
    pages = {586-594},
    year = {2008}
}

@article{Geredeli_Kunwar_Lee2024,
    author = {Pelin G. Geredeli and Hemanta Kunwar and Hyesuk Lee},
    title = {{Partitioning method for the finite element approximation of a 3D fluid-2D plate interaction system}},
    journal = {Numerical Methods for Partial Differential Equations},
    volume = {40},
    issue = {6},
    pages = {e23132},
    year = {2024}
}

@article{Avalos2014,
    author = {George Avalos and Thomas J. Clark},
    title = {{A mixed variational formulation for the wellposedness and numerical approximation of a PDE model arising in a 3-D fluid-structure interaction}},
    journal = {Evolution Equations and Control Theory},
    volume = {3},
    issue = {4},
    pages = {557-578},
    year = {2014}
}

@unpublished{Avalos2025,
  author = {George Avalos and Pelin G. Geredeli and Hemanta Kunwar and Hyesuk Lee},
  title = {{A novel approach to study the wellposedness of the 3D fluid-2D plate interaction PDE System}},
  year = {2025},
  note = {https://arxiv.org/abs/2509.03431},
}

@article{Chueshov2013,
    author = {Igor Chueshov and Iryna Ryzhkova},
    title = {{A global attractor for a fluid--plate interaction model }},
    journal = {Communications in Pure \& Applied Analysis},
    volume = {12},
    issue = {4},
    pages = {1635-1656},
    year = {2013}
}

@incollection{Ciarlet1974,
    title = {A Mixed Finite Element Method for the Biharmonic Equation},
    editor = {Carl {de Boor}},
    booktitle = {Mathematical Aspects of Finite Elements in Partial Differential Equations},
    publisher = {Academic Press},
    pages = {125-145},
    year = {1974},
    author = {P.G. Ciarlet and P.A. Raviart},
}

@article{Schaftingen2025,
    author = {Jean Van Schaftingen and Leon Winter},
    title = {{Trace theory for gauge-covariant Sobolev spaces}},
    journal = {Journal of Mathematical Analysis and Application},
    volume = {551},
    pages = {129697},
    year = {2025}
}

@article{Gagliardo1957,
    author = {E. Gagliardo},
    title = {{Caratterizzazioni delle tracce sulla frontiera relative ad alcune classi di funzioni in $n$ variabili}},
    journal = {Rendiconti del Seminario Matematico della Università di Padova},
    volume = {27},
    pages = {284-305},
    year = {1957}
}

@article{Sabbar2018,
    author = {Walaa A. Sabbar and Muneer A. Ismael and Mujtaba Almudhaffar},
    title = {{Fluid-structure interaction of mixed convection in a cavity-channel assembly of flexible wall}},
    journal = {International Journal of Mechanical Sciences},
    volume = {149},
    pages = {73-83},
    year = {2018}
}

@inproceedings{Hashim2012,
    author = {Uda Hashim and P.N. Diyana and Tijjani Adam},
    title = {{Numerical Simulation of Microfluidic Devices}},
    booktitle = {10th IEEE International Conference on Semiconductor Electronics (ICSE)},
    year = {2012}
}

@article{Nicolici2013,
    author = {S. Nicolici and R.M. Bilegan},
    title = {{Fluid structure interaction modeling of liquid sloshing phenomena in flexible tanks}},
    journal = {Nuclear Engineering and Design},
    volume = {258},
    pages = {51-56},
    year = {2013}
}

@article{Brummelen2009,
    author = {E. H. van Brummelen},
    title = {{Added Mass Effects of Compressible and Incompressible Flows in Fluid-Structure Interaction}},
    journal = {Journal of Applied Mechanics},
    volume = {76},
    issue = {2},
    pages = {021206},
    year = {2009}
}

@article{Gallistl2015,
    author = {Dietmar Gallistl},
    title = {{Morley finite element method for the eigenvalues of the biharmonic operator}},
    journal = {IMA Journal of Numerical Analysis},
    volume = {35},
    issue = {4},
    pages = {1779–1811},
    year = {2015}
}

@article{Cheng2008,
    author = {Lei Cheng and Robert D. White and Karl Grosh},
    title = {{Three-dimensional viscous finite element formulation for acoustic fluid–structure interaction}},
    journal = {Computer Methods in Applied Mechanics and Engineering},
    volume = {197},
    pages = {4160–4172},
    year = {2008}
}

@article{Hinton1986,
    author = {E. Hinton and H.C. Huang},
    title = {{A family of quadrilateral Mindlin plate elements with substitute shear strain fields}},
    journal = {Computers \& Structures},
    volume = {23},
    issue = {3},
    pages = {409-431},
    year = {1986}
}

@article{Geredeli2026,
    author = {Pelin G. Geredeli and Quyuan Lin and and Dylan McKnight and Mohammad Mahabubur Rahman},
    title = {{Analysis and numerical approximation to interactive dynamics of Navier Stokes-plate interaction PDE system}},
    journal = {Communications in Nonlinear Science and Numerical Simulation},
    volume = {152},
    pages = {109489},
    year = {2026}
}

@book{Solin2005,
    author = {Pavel Sol\'{i}n},
    title = {{Partial Differential Equations and the Finite Element Method}},
    publisher = {Wiley},
    year = {2005}
}

@article{Das2024,
    author = {Avijit Das and Bishnu P. Lamichhane and Neela Nataraj},
    title = {{A unified mixed finite element method for fourth-order time-dependent problems using biorthogonal systems}},
    journal = {Computers and Mathematics with Applications},
    volume = {165},
    pages = {52-69},
    year = {2024}
}

@article{Gudi2008,
    author = {Thirupathi Gudi and Neela Nataraj and Amiya K. Pani},
    title = {{Mixed Discontinuous Galerkin Finite Element Method for the Biharmonic Equation}},
    journal = {Journal of Scientific Computing},
    volume = {37},
    pages = {139–161},
    year = {2008}
}

@article{Glowinski1979,
    author = {R. Glowinski and O. Pironneau},
    title = {{Numerical methods for the first biharmonic equation and for the two-dimensional Stokes problem}},
    journal = {SIAM Review},
    volume = {21},
    issue = {2},
    pages = {167-212},
    year = {1979}
}

@article{Monk1987,
    author = {Peter Monk},
    title = {{A Mixed Finite Element Method for the Biharmonic Equation}},
    journal = {SIAM Journal on Numerical Analysis},
    volume = {24},
    issue = {4},
    pages = {737-749},
    year = {1987}
}

@article{Li2023,
    author = {Hengguang Li and Peimeng Yin and Zhimin Zhang},
    title = {{A $C^0$ finite element method for the biharmonic problem with Navier boundar’y conditions in a polygonal domain}},
    journal = {IMA Journal of Numerical Analysis},
    volume = {43},
    pages = {1779–1801},
    year = {2023}
}

@article{Causin2005,
    author = {P. Causin and J.F. Gerbeau and F. Nobile},
    title = {{Added-mass effect in the design of partitioned algorithms for fluid–structure problems}},
    journal = {Computer Methods in Applied Mechanics and Engineering},
    volume = {194},
    issue = {42},
    pages = {4506-4527},
    year = {2005}
}

@book{Richter2017,
    author = {Thomas Richter},
    title = {Fluid-structure Interactions: Models, Analysis and Finite Elements},
    publisher = {Springer },
    year = {2017}
}

@article{Brenner2018,
    author = {Susanne C. Brenner and Aycil Cesmelioglu and Jintao Cui and Li-Yeng Sung},
    title = {{A Nonconforming Finite Element Method for an Acoustic Fluid-Structure Interaction Problem}},
    journal = {Computational Methods in Applied Mathematics},
    volume = {18},
    issue = {3},
    pages = {383-406},
    year = {2018}
}

@article{MR3043640,
  AUTHOR = {Hecht, F.},
  TITLE = {New development in FreeFem++},
  JOURNAL = {Journal of Numerical Mathematics},
  VOLUME = {20}, YEAR = {2012},
  NUMBER = {3-4}, PAGES = {251--265}
}

@article{MUMPS:1,
   title   = {A Fully Asynchronous Multifrontal Solver Using Distributed Dynamic Scheduling},
   author  = {P.R. Amestoy and I. S. Duff and J. Koster and J.-Y. L'Excellent},
   journal = {SIAM Journal on Matrix Analysis and Applications},
   volume  = {23},
   number  = {1},
   year    = {2001},
   pages   = {15-41}
 }

@article{MUMPS:2,
  title = {{Performance and Scalability of the Block Low-Rank Multifrontal
  Factorization on Multicore Architectures}},
  author = {P.R. Amestoy and A. Buttari and J.-Y. L'Excellent and T. Mary},
  journal = {ACM Transactions on Mathematical Software},
  volume = 45,
  issue = 1,
  pages = {2:1--2:26},
  year={2019},
}

@article{Nguyen2021,
    author = {Duong T. A. Nguyen and Longfei Li and Hangjie Ji},
    title = {{Stable and accurate numerical methods for generalized Kirchhoff–Love plates}},
    journal = { Journal of Engineering Mathematics},
    volume = {130},
    pages = {6},
    year = {2021}
}

@article{Wang2025,
    author = {Xing Wang and Ivan Yotov},
    title = {{A Lagrange multiplier formulation for the fully dynamic Navier–Stokes–Biot system }},
    journal = {IMA Journal of Numerical Analysis},
    pages = {draf074},
    year = {2025}
}

@book{Quarteroni1994,
    author = {Alfio Quarteroni and Alberto Valli},
    title = {{Numerical Approximation of Partial Differential Equations}},
    publisher = {Springer Nature},
    year = {1994}
}

@article{Ambartsumyan2018,
    author = {Ilona Ambartsumyan and Eldar Khattatov and Ivan Yotov and Paolo Zunino†},
    title = {{A Lagrange multiplier method for a Stokes–Biot fluid–poroelastic structure interaction model}},
    journal = {Numerische Mathematik},
    volume = {140},
    pages = {513–553},
    year = {2018}
}

@article{Donea1982,
    author = {J. Donea and S. Giuliani and J.P. Halleux},
    title = {{An arbitrary Lagrangian-Eulerian finite element method for transient dynamic fluid-structure interactions}},
    journal = {Computer Methods in Applied Mechanics and Engineering},
    volume = {33},
    issue = {1-3},
    pages = {689-723},
    year = {1982}
}

@book{LionsBooks1972,
  author    = { J. L. Lions and E. Magenes},
  title     = {{Non-Homogeneous Boundary Value Problems and Applications }},
  publisher = {Springer Nature},
  year      = {1972},
  volume    = {1}
}

@book{Evans2010,
    author = {Lawrence C. Evans},
    title = {{Partial Differential Equations}},
    publisher = {American Mathematical Society},
    year = {2010},
    edition = {2}
}

@book{Layton2008,
    author = {William Layton},
    title = {{Introduction to the Numerical Analysis of Incompressible Viscous Flows}},
    publisher = {Society for Industrial and Applied Mathematics},
    year = {2008}
}

@book{Ern_Guermond2004,
    author = {Alexandre Ern and Jean-luc Guermond},
    title = {Theory and Practice of Finite Elements},
    publisher = {Springer},
    year = {2004}
}

\end{document}